\documentclass[10pt]{amsart}
\usepackage{amsfonts}
\usepackage{amsmath}
\usepackage{lipsum}
\usepackage{mathrsfs}
\usepackage{hyperref}
\usepackage{enumerate}
\usepackage{array}
\usepackage{tabularx}
\usepackage{float}
\restylefloat{table}
\usepackage{amssymb,tikz}
\usepackage{caption}
\usepackage{subcaption}
\usepackage{amsbsy}
\usepackage{psfrag}
\usepackage{pstricks}
\usepackage{filecontents}
\usepackage{graphics}
\usepackage{graphicx}
\usepackage{epstopdf,bm}
\numberwithin{equation}{section}

\usepackage{enumitem}

\theoremstyle{definition}
\newtheorem{theorem}{Theorem}[section]
\newtheorem{remark}[theorem]{Remark}

\newtheorem{lemma}[theorem]{Lemma}
\newtheorem{example}[theorem]{Example}

\newtheorem{defi}[theorem]{Definition}
\newtheorem{Cor}[theorem]{Corollary}

\def\R{\mathbb{R}}

\def\cK{\mathcal{K}}

\def\O{\Omega}

\def\cV{\mathcal{V}}

\def\cT{\mathcal{T}}
\def\cE{\mathcal{E}}

\def\jump#1{\left[\hskip -3.5pt\left[#1\right]\hskip -3.5pt\right]}

\def\sjump#1{[\hskip -1.5pt[#1]\hskip -1.5pt]}

\usepackage{fancyhdr}
\begin{document}
\title[Mixed finite element method for a second order Dirichlet boundary control problem]
{Mixed finite element method for a second order Dirichlet boundary control problem}

\author{Divay Garg}\thanks{The first author's work is supported by CSIR Research Fellowship}

\address{Department of Mathematics, Indian Institute of Technology Delhi - 110016}

\email{divaygarg2@gmail.com}

\author{Kamana Porwal}\thanks{The second author's work is supported by DST Inspire Faculty Research  Grant}

\address{Department of Mathematics, Indian Institute of Technology Delhi - 110016}

\email{kamana@maths.iitd.ac.in}
\date{\today}


\begin{abstract}
The main aim of this article is to analyze mixed finite element method for the second order Dirichlet boundary control problem. Therein, we develop both a priori and a posteriori error analysis using the energy space based approach.
 We obtain optimal order a priori error estimates in the energy norm and $L^2$-norm with the help of auxiliary problems. 
 The reliability  and the efficiency of proposed \emph{a posteriori} error estimator is discussed using the Helmholtz decomposition. Numerical experiments are presented to confirm the theoretical findings. 
\end{abstract}
\maketitle
\section{Introduction}
The optimal control problems subjected to partial differential equations have numerous applications in science and engineering. The objective of optimal control problems is to find the optimal control which minimizes the given cost functional with certain constraints being satisfied. Finite element methods are extensively used  for the numerical approximation of the optimal control problems as they are advantageous over other numerical methods in terms of implementation, accuracy and adaptability. As per the literature review, previous studies have marked a significant development in this area and the literature in this direction is too immense to mention all the results here. One can find key contributions towards this area in the articles cited here and references therein. 
For a general theory of optimal control problems constrained by partial differential equations and their numerical approximations, we refer to \cite{lions,trolzbook}. For the contributions towards the finite element analysis for the distributed optimal control problems, we refer to \cite{Falk:1973:OCP,super2,Geveci:1979:OCP,apostdistributed1,super1,distributed1,apostdistributed2,distributed2,distributed3,apostdistributed3}. The articles \cite{Falk:1973:OCP,Geveci:1979:OCP,distributed1,distributed2,distributed3} study the convergence analysis of finite element methods for the constrained distributed optimal control problems.  In \cite{super2,super1}, the authors have discussed super convergence results and the articles \cite{apostdistributed1,apostdistributed2,apostdistributed3} are devoted to  \emph{a posteriori} error analysis of the finite element method for the distributed control problems. In the article \cite{yanpingchen}, authors have used mixed finite element method to derive error estimates and super-convergence results for the distributed optimal control problem. The error estimates for the Neumann boundary control problems governed by linear state equations and semi linear state equations can be found in \cite{neumann1} and \cite{neumann2,neumann3}, respectively. The article \cite{neumann4} concerns the finite element analysis with graded mesh refinement for the Neumann boundary control problem. {Recently in \cite{leng:2017}, the authors proved the convergence and studied the optimal complexity of an adaptive finite element method for control constrained problems on $L^2$ errors. Therein, the contraction property and the quasi optimality of the adaptive finite element method is derived. The authors in \cite{leng:2018} study the distributed convex optimal control problem with integral control constraint and proved the optimal convergence rate for the adaptive finite element method under a mild assumption on the initial mesh.} \\
\par \noindent
 There is relatively less literature available for the Dirichlet boundary control problem due to variational difficulty. In the article \cite{20}, the authors have discussed the numerical approximation of semi linear elliptic Dirichlet boundary control problems with pointwise control constraints. In this article, an optimal error estimate for the finite element approximation of the optimal control in the $L^2$-norm is derived while the article \cite{38} improved the error estimate for the state and adjoint state variables. In both the articles \cite{20,38}, an ultra weak formulation has been used for formulating the model problem. The convergence analysis of the conforming finite element method for the Dirichlet boundary control problem is studied in \cite{42} using the energy space based approach. Therein, the Steklov-Poincare operator arising from the harmonic extension is used to define the continuous and discrete problems. The authors in \cite{chowdhurykkt} have used an alternative energy space based approach in which the control is sought from $H^1(\Omega)$ space and the resulting control is a harmonic function without being explicitly imposed. In \cite{gudi}, the authors have used energy space based approach and established the optimal order energy norm error estimates of the conforming finite element method for the constrained Dirichlet boundary control problem governed by diffusion problem. A variation of this approach using conforming finite element method is discussed in \cite{Karkulik:2020}.
Recently in \cite{meshless}, two meshless methods are proposed for solving the Dirichlet boundary optimal control problems governed by elliptic PDEs. The finite element analysis of Dirichlet control problems using an energy regularization is carried out in \cite{Winkler:2020}.
 We refer to the articles \cite{GLTY:2019,BY:2017} for a posteriori error analysis of finite element methods for boundary control problems governed by elliptic PDEs.\\
 
 \par \noindent
{ In many applications, it is important to obtain accurate approximation of the scalar variable and its gradient simultaneously. A common way to achieve this goal is to use mixed finite element methods \cite{ RT:1977,BDM:1985,brezzifortinbook,RT:1991}. Moreover, mixed finite element methods have the property that they maintain the discrete conservation law at the element level.
In \cite{gongyan}, mixed finite element method is used for the approximation of the Dirichlet boundary control problem governed by elliptic PDEs. Therein, the control is sought from $L^2(\partial \Omega)$ and piece-wise constant finite element space is employed for the discretization of control. The authors have obtained the optimal a priori error estimates of order $O(h^{1-1/s})$ with $s \geq 2$ for polygonal domains and quasi optimal error estimates of order $O(h|ln\,h|)$ for smooth domains. } 
In this article, we follow the energy space based approach for the error analysis concerning the second order Dirichlet boundary control problem. This approach produces a sufficiently regular control. For the variational formulation, the state equation is converted to the mixed system using the mixed variational scheme for second order elliptic equations and then the continuous optimality system is derived. The control is sought from the $H^1(\Omega)$ space and it satisfies the Neumann problem.  {The state equation needs to be understood in the very weak sense when the control is sought in $L^2(\Gamma)$ space \cite{38}, since the trace operator maps $H^1(\Omega)$ onto $H^{1/2}(\Gamma)$ where $\Gamma$ represents the boundary of the domain $\Omega$. Hence the normal derivative of the costate appears on the boundary of the domain in the first order optimality condition which makes the problem more complicated due to its discontinuity in the standard finite element methods. Thus, the use of mixed finite element methods is advantageous as it inserts naturally the normal derivative of costate on the boundary in the weak formulation. In order to discretize the continuous optimality system, we use the lowest order Raviart-Thomas space to numerically approximate the state and costate variables whereas the continuous piece-wise linear finite element space is used for the discretization of control. Based on this formulation, we have derived the optimal \emph{a priori} error estimates for the control of order $O(h)$ in the energy norm and of order $O(h^2)$ in $L^2$-norm. We have also achieved the optimal error estimates of order $O(h)$ for the state, costate and gradient of the costate in $L^2$-norm.} Moreover, we have developed a reliable and efficient \emph{a posteriori} error estimator using an auxiliary system of equations. Finally, numerical experiments are presented to illustrate the theoretical results on \emph{a priori} as well as \emph{a posteriori} error analysis. We remark here that the approach adopted in this article for analyzing the mixed finite element method for the Dirichlet boundary control problem differs significantly than that of the article \cite{gongyan}.\\

\par\noindent
The rest of the article is arranged as follows. In Section \ref{secdbcp}, we discuss the Dirichlet boundary control problem in a precised manner and pose the mixed variational scheme for it. Therein, we also obtain the corresponding optimality system and deduce the regularity of the optimal variables. In  Section \ref{secfem}, we introduce some useful notations and preliminary results, formulate the mixed finite element method for the continuous Dirichlet boundary control problem and obtain the discrete optimality system.  We discuss \emph{a priori} error estimates for the optimal control in the energy norm and $L^2$-norm in Section \ref{secapriori}. In Section \ref{secapost}, we develop \emph{a posteriori} error analysis with the help of an auxiliary problem and the Helmholtz decomposition.  Therein, we derive reliability and efficiency estimates for the proposed a posteriori error estimator. We conclude the article in Section \ref{secnumer} by presenting numerical test results supporting the theoretical findings.

\section{Continuous Problem}\label{secdbcp}

\noindent
Let $\Omega \subset \R^2$ be a bounded convex domain with polygonal boundary $\partial \Omega=\Gamma.$ We assume the boundary $\Gamma$  to be the union of line segments $\Gamma_i (1 \leq i \leq k)$ such that their interiors are pairwise disjoint in the induced topology. We use the standard notation $W^{s,p}(\Omega)$ for Sobolev spaces on $\Omega$ with norm $\|\cdot\|_{s,p,\Omega}$ and semi-norm $|\cdot|_{s,p,\Omega}$(see \cite{adams}). Further, we denote $W^{s,2}(\Omega)$ by $H^s(\Omega)$ with norm $\|\cdot\|_{s, \Omega}$ and semi-norm $|\cdot|_{s,\Omega}$.  Let $(\cdot, \cdot)$ (resp. $\|\cdot\|$) denote the $L^2(\Omega)$ inner product (respectively norm). We denote  $H^{-{1}/{2}}( \Gamma )$ and $H^{{1}/{2}}( \Gamma )$ duality pairing by $\langle \cdot, \cdot \rangle$.  

Define the cost functional $J(s,g):=\frac{1}{2}\|s-y_d\|^2+\frac{\alpha}{2}\|\nabla g\|^2$ and  consider the following Dirichlet boundary control  problem governed by the second order elliptic PDE:
\begin{equation}\label{b1}
    \min_{(s,g)\in \cK} J(s,g)
\end{equation}
 subject to 

 \begin{align}
 -\Delta s &=f \ \ \text{in}\ \  \Omega, \label{b2}\\
 \hspace{5mm}   s &=g \ \ \text{on} \ \ \Gamma,\label{b21}
 \end{align}
where $y_d,f\in L^2(\Omega)$ are given functions and $\cK$ is the admissible space (to be specified later), while $\alpha>0$ is a fixed parameter (referred as regularization parameter). 
 In this article, we adopt a mixed finite element technique to solve the Dirichlet boundary control problem \eqref{b1}-\eqref{b21}.  By introducing the flux variable, $\textbf{k}=-\nabla s$,  \eqref{b2}  can be written in the following mixed form: 
 \begin{align}\label{bb4}
     \begin{cases}
       \textbf{k}&= -\nabla s, \\
      \text{div}\,\textbf{k} & = f .
     \end{cases}
 \end{align}
 Set $V=H(\text{div}, \Omega)$, $W=L^2(\Omega)$, $Q=H^1(\Omega)$, where
  the space $H(\text{div}, \Omega)$ with $\|\cdot\|_{H(\text{div}, \Omega)}$ norm  is defined as 
 \begin{align*}
     H(\text{div}, \Omega)&:= \{\textbf{v}\in (L^2(\Omega))^2, \  \text{div}\, \textbf{v} \in L^2(\Omega) \}, \\
     \|\textbf{v}\|_{H(\text{div},\Omega)} &:= \|\textbf{v}\|+ \|\text{div}\,\textbf{v}\|.
 \end{align*}
Next, we introduce the trace operator  defined on $H(\text{div}, \Omega)$ \cite[Lemma 1.1, Chapter III]{brezzifortinbook}.
 \begin{lemma}\label{l1}
  There exists a trace operator $\gamma : H(\text{div},\Omega) \rightarrow H^{-\frac{1}{2}}(\Gamma)$ defined by $\gamma(\mathbf{v})=\mathbf{v}\cdot n$ in the sense that
 \begin{align*}
 \langle \mathbf{v}\cdot n,w \rangle=\int_{\Omega} w \,\text{div}\,\mathbf{v} \,dx+\int_{\Omega}\mathbf{v}\cdot\nabla w\,dx\;\;\forall \ w\in H^1(\Omega).\label{note}
 \end{align*}
 The operator $\gamma$ is a continuous mapping from $H(\text{div}, \Omega)$ onto $H^{-\frac{1}{2}}(\Gamma)$ such that
 \begin{equation*}
     \|\mathbf{v}\cdot n\|_{-\frac{1}{2},\Gamma} \leq \|\mathbf{v}\|_{H(\text{div}, \Omega)}.
 \end{equation*}
 \end{lemma}
\par \noindent
 In view of \eqref{bb4}, the weak mixed formulation of \eqref{b2}-\eqref{b21} reads as follows:
 find $(s, \mathbf{k}) \in W \times V$ such that 
 \begin{align}
       (\textbf{k}, \textbf{v}) - (s, \text{div}\, \textbf{v}) &= - \langle \textbf{v} \cdot n,g \rangle \quad \forall ~ \textbf{v} \in V ,\label{b5} \\ 
     (w, \text{div}\, \textbf{k}) &= (f, w) \quad \forall~  w  \in W.  \label{bb5}
 \end{align}
\par \noindent
{ There exists an unique solution $(s,\mathbf{k})$ of \eqref{b5}-\eqref{bb5} (we refer \cite{brezzifortinbook} for details)}.
We now define the solution operator $S$ as follows:
 \begin{defi}
For a given $f\in W$ and $g\in H^{\frac{1}{2}}(\Gamma)$, the solution operator $S:W\times H^{\frac{1}{2}}(\Gamma) \rightarrow W\times V$ is defined by $S(f,g)=(s,\mathbf{k})$, where $(s,\mathbf{k})$ satisfy \eqref{b5}-\eqref{bb5}.
\end{defi}
The mixed form of the Dirichlet boundary optimal control problem consists of finding $(y,\textbf{p},u)\in W \times V \times Q$ such that 
{\begin{eqnarray}\label{b7}
J(y,u) = \min_{(s,g) \in W \times Q} J(s,g),
\end{eqnarray}
subject to the condition that $(s,\textbf{k})=S(f,g)$ where $J(s,g)$ is the underlying cost functional.} 

\par \noindent
Let   $a(\cdot, \cdot): H^1(\O)\times H^1(\O) \rightarrow \mathbb{R}$ denote the bilinear form defined by
\begin{equation*}
    a(v, w)=\int_{\O}\nabla v \cdot \nabla w\,dx.
\end{equation*}
We have the following result concerning the first order necessary optimality conditions.
\begin{theorem}
There exists a unique solution $(y,\textbf{p},u) \in W \times V \times Q$ of the optimal control problem \eqref{b7}. {Furthermore, the triplet $(y,\textbf{p},u)$ is the solution of \eqref{b7} iff there exists a unique pair $(z,\textbf{r}) \in W \times V$ such that $(y,\textbf{p},z,\textbf{r},u)$ satisfies the following optimality system}: 
\begin{align}
(\mathbf{p},\mathbf{v})-(y,\text{div}\,\mathbf{v})&=-\langle \mathbf{v}\cdot n,u\rangle \ \ \forall \  \mathbf{v}\in V, \label{b10}\\
(w,\text{div}\,\mathbf{p})&=(f,w)\ \ \forall\  w\in W, \label{b11} \\
(\mathbf{r},\mathbf{v})-(z,\text{div}\,\mathbf{v})&=0 \ \ \forall \ \mathbf{v}\in V , \label{b12}\\
(w,\text{div}\,\mathbf{r})&=(y-y_d,w) \ \ \forall\  w\in W , \label{b13}\\
\alpha a(u, q)&=-\langle \mathbf{r}\cdot n,q\rangle \ \  \forall \  q\in Q \label{b14}.
\end{align}
\begin{proof}
The proof follows from the standard arguments as in \cite{trolzbook,chowdhurykkt}.
\end{proof}
\end{theorem}
\par \noindent
 Note that,  {$(u,q)$} when restricted to the boundary $\Gamma$ belong to  $H^{\frac{1}{2}}(\Gamma) \times H^{\frac{1}{2}}(\Gamma)$ since $(u, q) \in Q \times Q $. Therefore, in light of Lemma \ref{l1}, $\langle \mathbf{v}\cdot n,u\rangle$ and  $\langle \mathbf{r}\cdot n,q\rangle$ arising in equations \eqref{b10} and \eqref{b14} respectively are well defined.
\begin{remark}
In view of \cite[Remark 2.4]{chowdhurykkt}, the minimum energy in the minimization problem \eqref{b7} can be realized with an equivalent $H^{\frac{1}{2}}(\Gamma)$ { semi-norm} of the control $u$ in the sense that
the $H^{\frac{1}{2}}(\Gamma)$ semi-norm of $q \in H^{\frac{1}{2}}(\Gamma)$ can be defined by the following Dirichlet norm:
\begin{equation*}
    |q|_{\frac{1}{2},\Gamma} := \|\nabla z_q\|= \min_{w\in Q,w=q\  \text{on} \  \Gamma } \|\nabla w\|,
\end{equation*}
where the minimizer $z_q \in Q$ satisfies 
\begin{align*}
    -\Delta z_q&=0 \ \text{in} \ \Omega, \\
    z_q &=q \ \text{on} \ \Gamma.
\end{align*}
Uniqueness of the optimal control $u\in Q$ implies $u=z_u$ and hence 
\begin{equation*}
    |u|_{\frac{1}{2},\Gamma}:= \|\nabla u\|.
\end{equation*}
\end{remark}
\noindent
Now using Green's formula in \eqref{b14}, we see that $u$ weakly solves the following Neumann problem
\begin{equation}\label{b16}
    \begin{cases}
      -\Delta u = 0 \ \ \text{in} \ \ \Omega,\\
      \alpha \frac{\partial u}{\partial n} = - \mathbf{r}\cdot n \ \ \text{on} \ \ \Gamma.
    \end{cases}
\end{equation}
{Taking $q=1$ in \eqref{b14}, we observe that $\langle \mathbf{r}\cdot n,1\rangle=0$, which is the compatibility condition for the Neumann problem defined by \eqref{b16}.
Since $y-y_d \in W$, from the elliptic regularity theory on convex polygonal domains \cite[Theorem 3.1.2.1]{grisvardbook}, we find that $(z, \mathbf{r})\in H^2(\Omega)\times [H^1(\Omega)]^2$. By the trace theorem \cite[Theorem 1.5.2.1]{grisvardbook}, $\mathbf{r} \cdot n \in H^{\frac{1}{2}}(\Gamma_i) ~~ \forall ~~ 1\leq i \leq k.$ Hence the elliptic regularity theory for the Neumann problem \cite[Theorem 3.1.2.3]{grisvardbook} implies that $u \in H^2(\Omega).$ Now $f \in W$ and trace  $u \in H^{\frac{3}{2}}(\Gamma_{i}) ~~ \forall ~~ 1\leq i \leq k$, therefore $y\in H^2(\O)$ by the elliptic regularity for the Dirichlet problem on the convex polygonal domains \cite[Theorem 3.1.2.1]{grisvardbook}.}
\section{Finite element approximation}\label{secfem}
\noindent
In this section, we consider the mixed
finite element approximation of the Dirichlet boundary optimal control problem under consideration. We introduce the following notations, which are used throughout the article.
\begin{itemize}
	\item $\cT_h$ is a regular triangulation(see \cite{ciarletbook}) of $\Omega$.
	\item $\cE^i_h$ is the set of all interior edges of $\cT_h$.
	\item $\cE^b_h$ is the set of all boundary edges of $\cT_h$.
	\item $\cE_h= \cE^i_h \cup \cE^b_h$ is the set of all the edges of $\cT_h$.
	 \item $\cV_h^i \text{ is the set of all vertices interior to } \Omega$.
    \item $\cV_h^b \text{ is the set of all vertices that belong to } \Gamma$.
	\item $\cV_h = \cV_h^i \cup \cV_h^b  \text{ is the set of all vertices of }\cT_h$.
	\item $\cT_{p}$ denotes the set of all elements sharing the vertex $p\in \cV_h$ and $|\cT_p|$ denotes cardinality of $\cT_{p}$.
	\item $h_T$ is the diameter of an element $T\in \cT_h$ and $h= \displaystyle \max_{T\in \cT_h} h_T$.
	\item $|e|$ is the length of an edge $e \in \cE_h$ and $m_e$ is the midpoint of $e$.
	\item $\mathbb{P}_k(T)$ is the set of all polynomials of degree less than or equal to $k$ over $T$.
	\item $\mathbb{P}_{k+1}^c(\cT_h)$ is the set of continuous piecewise polynomials of degree $k+1$.
	\item $H^s(\Omega,\cT_h)=\{v\in L^2(\Omega): v|_T\in H^s(T) ~\forall ~ T\in \cT_h\}$ is the piece-wise Sobolev space with respect to the triangulation $\cT_h$.\\
	\item  $| \cdot |_{s,h}^2= \sum_{T \in \cT_h} |\cdot|^2_{s,T}$ is the semi-norm on $H^s(\Omega,\cT_h)$.
	\item We say $a \lesssim b$ if there {exists} a positive constant $C$ independent of the mesh parameter $h$ such that $a\leq C b$.
\end{itemize}
 Let $e \in \cE_h^i$ be the edge shared by two neighbouring triangles $T_+$ and $T_-$ i.e. $e= \partial T_+ \cap \partial T_-$. Further, suppose $n_+$ is the unit normal of $e$ pointing from $T_+$ to $T_-$, and $n_- = -n_+$. We denote $v|_{T_+} $ by $v_+$ and $v|_{T_-}$ by $v_-$. For any scalar valued function $v \in H^2(\Omega, \cT_h)$ and $w \in H^1(\Omega, \cT_h)$, we define the jumps $\sjump{\cdot}$ across the edge $e$ as follows:-

  $$ \jump {\frac{\partial v}{\partial n}}= \frac{\partial v_+}{\partial n_+}\Big|_{e}- \frac{\partial v_-}{\partial n_+}\Big|_e \quad \quad \text{and} \quad \quad \sjump{w}= \left( w_+\Big|_e- w_-\Big|_e\right).$$
For any edge $e \in \cE_h^b$, there is a triangle  $T \in \cT_h$ such that $e=\partial T \cap \Gamma$ and let $n_e$ be the unit normal on $e$ that points outside $T$. For any $v \in H^2(\Omega, \cT_h)$ and $w \in H^1(\Omega, \cT_h)$, we set $$\jump{\frac{\partial v}{\partial n}}= \frac{\partial v|_T}{\partial n_e} \quad \quad \text{and} \quad \quad \sjump{w}=w|_T.$$
For $T \in \cT_h$ , we define the lowest order local Raviart-Thomas space by 
\begin{equation*}
    RT(T)=P_0(T)^2 + \textbf{x}\,P_0(T),
\end{equation*}
where $\mathbf{x}=(x_1, x_2) \in \R^2$.
Now the global Raviart-Thomas space is defined by
\begin{equation*}
    V_h:=\{\textbf{v} \in V: \textbf{v}|_T \in RT(T) \ \forall \  T \in \cT_h\}.
\end{equation*}
 We define two more discrete spaces for the approximation of state $y$ and control $u$ as follows
\begin{equation*}
\begin{split}
    W_h &:=\{w_h \in W: w_h|_T \in P_0(T) \ \forall \ T \in \cT_h \},\\ 
    Q_h &:=\{ q_h \in Q: q_h|_T \in P_1(T) \ \forall \ T \in \cT_h   \}.
    \end{split}
\end{equation*}
Consider the mixed finite element approximation to the variational problem  \eqref{b5}-\eqref{bb5}:
find $(s_h, \mathbf{k_h}) \in W_h \times V_h $ such that
\begin{align}
(\mathbf{k_h},\mathbf{v_h})-(s_h,\text{div}\,\mathbf{v_h})&=-\langle \mathbf{v_h}\cdot n,g\rangle \ \ \forall \  \mathbf{v_h}\in V_h, \label{c11} \\
(w_h,\text{div}\,\mathbf{k_h})&=(f,w_h)\ \ \forall\  w_h\in W_h. \label{cc11}
\end{align}
Here onwards, we shall use $C$ to denote the generic positive constant which can assume different values at different appearances, it is independent of the discrete solutions and will depend only on the minimum angle of the triangulation but not on the mesh size. \\ 
We have the following well established \emph{a priori} estimate for standard mixed finite element approximation \cite{vidarthomee,douglas,GR:1996}.
\begin{lemma}\label{l2}
Let $(s,\mathbf{k}) \in W \times V$ be the solution of problems \eqref{b5}-\eqref{bb5} with data $(g,f) \in H^{\frac{3}{2}}(\Gamma) \times W$ and let $(s_h, \mathbf{k_h}) \in W_h \times V_h$ be the solution of \eqref{c11}-\eqref{cc11}. Then, $s \in H^2(\Omega)$ and
\begin{equation*}
    \|\mathbf{k}-\mathbf{k_h}\|+ \|s-s_h\| \leq Ch\left(\|s\|_{2, \Omega}+\|f\|\right).
\end{equation*}
Furthermore, assuming $\text{div} \mathbf{k} \in H^1(\Omega)$, the following estimate holds
\begin{equation*}
    \|\text{div}(\mathbf{k}-\mathbf{k_h})\|_{-s,\Omega}  \leq Ch^{1+s} |\text{div}\mathbf{k}|_{1, \Omega}, \quad s=0,1.
\end{equation*}


\end{lemma}

\noindent
Analogous to the continuous setting, we define the discrete solution operator.
\begin{defi}
For a given $f\in W$ and $g\in H^{\frac{1}{2}}(\Gamma)$, the discrete solution operator $S_h:W\times H^{{1}/{2}}(\Gamma) \rightarrow W_h\times V_h$ is defined by $S_h(f,g)=(s_h,\mathbf{k_h})$, where $(s_h,\mathbf{k_h})$ satisfy \eqref{c11}-\eqref{cc11}.
\end{defi}
\noindent The mixed finite element approximation of the Dirichlet boundary control problem \eqref{b7} consists of finding $(y_h, \mathbf{p_h}, u_h)\in W_h\times V_h\times Q_h$ such that
{
\begin{equation} \label{c1}
     J(y_h,u_h)=\min_{(s_h,g_h)\in W_h\times Q_h} J(s_h,g_h),
\end{equation}
subject to the condition that $(s_h,\mathbf{k_h})=S_h(f, g_h)$ where $J(s_h,g_h):=\frac{1}{2}\|s_h-y_d\|^2+\frac{\alpha}{2}\|\nabla g_h\|^2.$}\\\\
In the next theorem, we obtain the discrete optimality system. This system of equations is obtained by the first order necessary optimality conditions as in the case of continuous system \eqref{b10}-\eqref{b14}.
\begin{theorem}
There exists a unique solution $(y_h,\mathbf{p_h},u_h) \in W_h \times V_h \times Q_h$ of the optimal control problem \eqref{c1}. {Furthermore, $(y_h,\mathbf{p_h},u_h)$ is the solution of \eqref{c1} if and only if there exists a unique pair $(z_h,\mathbf{r_h}) \in W_h \times V_h$ such that $(y_h,\mathbf{p_h},z_h,\mathbf{r_h},u_h)$ satisfies the following system of equations}: 
\begin{align}
(\mathbf{p_h},\mathbf{v_h})-(y_h,\text{div}\,\mathbf{v_h})&=-\langle \mathbf{v_h}\cdot n,u_h\rangle \ \ \forall \  \mathbf{v_h}\in V_h, \label{c4}\\
(w_h,\text{div}\,\mathbf{p_h})&=(f,w_h)\ \ \forall\  w_h\in W_h, \label{c5} \\
(\mathbf{r_h},\mathbf{v_h})-(z_h,\text{div}\,\mathbf{v_h})&=0 \ \ \forall \ \mathbf{v_h}\in V_h , \label{c6}\\
(w_h,\text{div}\,\mathbf{r_h})&=(y_h-y_d,w_h) \ \ \forall\  w_h\in W_h , \label{c7}\\
\alpha a( u_h, q_h)&=-\langle \mathbf{r_h}\cdot n,q_h\rangle \ \  \forall \  q_h\in Q_h \label{c8}.
\end{align}
\end{theorem}
\par \noindent
Below, we define the Raviart-Thomas interpolation, $L^2$ projection map and their properties which will be useful in the error analysis.
\begin{lemma}\label{rtinter}
There exist linear operators $\Pi_h : V \rightarrow V_h$ and $P_h:W\rightarrow W_h$ ($L^2$ projection onto $W_h$) such that for all $\mathbf{v}\in V$ and $w\in W,$
\begin{align}
    (\text{div}\,(\mathbf{v}-\Pi_h\mathbf{v}),w_h)&=0 \ \ \ \forall \ w_h \in W_h,\label{interr1}\\
     \|\mathbf{v} - \Pi_h \mathbf{v}\|_{s,\Omega} &\leq Ch^{1-s}|\mathbf{v}|_{1,\Omega},  \  \mathbf{v} \in H^1(\Omega) \times H^1(\Omega), \ \ s=0,1,\label{interr2}\\
        \|\text{div}\,(\mathbf{v}-\Pi_h\mathbf{v})\|_{-s,\Omega}& \leq Ch^{1+s}|\text{div}\, \mathbf{v}|_{1,\Omega},   \ \text{div}\,\mathbf{v} \in H^1(\Omega),\ \ s=0,1, \label{inter3}\\
    (w - P_h w , w_h)&=0 \ \ \forall \ w_h \in W_h,\label{inter4}\\
     \|w - P_h w\|_{-s,\Omega} &\leq Ch^{1+s}|w|_{1,\Omega},  \ w \in H^1(\Omega) \ \ s=0,1\label{inter5}. 
\end{align}
\begin{proof}
The readers may refer \cite{vidarthomee,brezzifortinbook} for the proof of this Lemma.
\end{proof}
\end{lemma}
In the next lemma, we state a well known Cl\'ement type approximation result \cite{clement} required in the subsequent analysis.
\begin{lemma}\label{clement:approx}
Let $\xi\in H^1(\Omega)$. Then there { exists}  $\xi_h \in \mathbb{P}_{k+1}^c(\cT_h)$ such that for any $T\in \cT_h$ and $e \in \partial T$, it holds that
\begin{align*}
\|\xi-\xi_h\|_{L^2(e)}&\leq C|e|^{\frac{1}{2}}\|\nabla \xi\|_{L^2(\Tilde{T})}, \\
h_T|\xi-\xi_h|_{H^1(T)}+\|\xi-\xi_h\|_{L^2(T)}&\leq Ch_{T}\|\nabla \xi\|_{L^2(\Tilde{T})},
\end{align*}
where $\Tilde{T}=\{T'\in \cT_h: T \cap T'\neq \emptyset \}$ and $k$ is any non negative integer.
\end{lemma}
\noindent
 In the further sections, for any $v\in Q\cap C^0(\bar{\Omega)}$, $v_I\in Q_h$ denotes its nodal Lagrange interpolation. We refer to \cite{ciarletbook,brennerbook} for the definition and approximation properties of $v_I$. 
\section{\emph{A Priori} error analysis}\label{secapriori}
In this section, we discuss the error estimates in the energy norm and $L^2$-norm between the solutions of \eqref{b10}-\eqref{b14} and \eqref{c4}-\eqref{c8}. In deriving the estimates for the control in the energy norm, we first obtain the bound on $H^1$-semi-norm and subsequently the optimal order error estimates in the energy norm is derived with the help of intermediate results. In order to achieve the optimal order convergence in the $L^2$-norm, we introduce the enriched discrete control and make use of some auxiliary results. We begin by proving intermediate results required for the further analysis. \\
For the optimal control ${u}\in Q$, let $(y^h_f(u), \mathbf{p}^h_f(u))=S_h(f,u)$ and { $(z^h_u,\mathbf{r}^h_u)=S_h(y^h_f(u)-y_d,0)$} which means $(y^h_f(u), \mathbf{p}^h_f(u),z^h_u,\mathbf{r}^h_u) \in W_h\times V_h\times W_h\times V_h$ satisfies the following system of equations:
\begin{align}
(\mathbf{p}^h_f(u),\mathbf{v_h})-(y^h_f(u),\text{div}\,\mathbf{v_h})&=-\langle \mathbf{v_h}\cdot n,u\rangle \ \ \forall \  \mathbf{v_h}\in V_h, \label{d1}\\
(w_h,\text{div}\,\mathbf{p}^h_f(u))&=(f,w_h)\ \ \forall\  w_h\in W_h, \label{d2} \\
(\mathbf{r}^h_u,\mathbf{v_h})-(z^h_u,\text{div}\,\mathbf{v_h})&=0 \ \ \forall \ \mathbf{v_h}\in V_h , \label{d3}\\
(w_h,\text{div}\,\mathbf{r}^h_u)&=(y^h_f(u)-y_d,w_h) \ \ \forall\  w_h\in W_h . \label{d4}
\end{align}
Now for optimal state $y \in W$, find $(z^h_y, \mathbf{r}^h_y) \in W_h \times V_h$ such that 
\begin{align}
    (\mathbf{r}^h_y,\mathbf{v_h})-(z^h_y,\text{div}\,\mathbf{v_h})&=0 \ \ \forall \ \mathbf{v_h}\in V_h , \label{d5}\\
(w_h,\text{div}\,\mathbf{r}^h_y)&=(y-y_d,w_h) \ \ \forall\  w_h\in W_h . \label{d6}
\end{align}
We note that $(z^h_y, \mathbf{r}^h_y)=S_h(y-y_d,0)$ and  $(y^h_f(u),\mathbf{p}^h_f(u))=S_h(f,u)$ are the standard mixed finite element approximations of $(z,\mathbf{r})$ and $(y,\mathbf{p})$ respectively. Taking into account the standard error estimates of mixed finite element method (stated in Lemma \ref{l2}), we have the following result \cite{yanpingchen,gongyan}:
\begin{lemma}\label{stdmixed}
Let $(y,\mathbf{p},z,\mathbf{r},u)$ be the solutions of \eqref{b10}-\eqref{b14}. Let $(z^h_y,\mathbf{r}^h_y)$ and $(y^h_f(u), \mathbf{p}^h_f(u))$ be the solution of the auxiliary problem \eqref{d5}-\eqref{d6} and \eqref{d1}-\eqref{d2} respectively. Then, we have the following estimates:
\begin{align*}
\|z-z^h_y\| + \|\mathbf{r}-\mathbf{r}^h_y\|_{H(\text{div},\Omega)} &\leq Ch\|z\|_{2,\Omega} ,\\
\|y-y^h_f(u)\| + \|\mathbf{p}-\mathbf{p}^h_f(u)\|_{H(\text{div},\Omega)} &\leq Ch\|y\|_{2,\Omega}.
\end{align*} 
\end{lemma}
\par \noindent
This lemma is useful in establishing the auxiliary estimates which play a key role in deriving the desired error estimates in the energy norm. We first derive the  error estimates in $H^1$ semi-norm in the following theorem.

\begin{theorem}\label{energyestimate}
Let $(y,\mathbf{p}, u, z, \mathbf{r}) \in W \times V \times Q \times W \times V$ be the solution of continuous optimality system \eqref{b10}-\eqref{b14} and let $(y_h, \mathbf{p_h}, u_h, z_h, \mathbf{r_h}) \in W_h \times V_h \times Q_h \times W_h \times V_h$ be the solution of discrete optimality system \eqref{c4}-\eqref{c8}. Then the following estimate holds: 
\begin{equation*}
    |u-u_h|_{1, \Omega}+\|y^h_f(u)-y_h\| \leq Ch\left(\|u\|_{2, \Omega}+\|z\|_{2, \Omega}+\|y\|_{2, \Omega}+\|y-y_d\|\right).
\end{equation*}
\begin{proof}
Upon subtracting the corresponding equations of the system \eqref{c4}-\eqref{c7} from \eqref{d1}-\eqref{d4}, we obtain the following error equations:
\begin{align}
    (\mathbf{p}^h_f(u)-\mathbf{p_h},\mathbf{v_h})-(y^h_f(u)-y_h, \text{div}\,\mathbf{v_h})&= -\langle \mathbf{v_h}\cdot n,u-u_h\rangle \ \forall \ \mathbf{v_h} \in V_h,\label{des1}\\
    (w_h,\text{div}(\mathbf{p}^h_f(u)-\mathbf{p_h}))&=0\ \ \forall\  w_h\in W_h, \label{des2} \\
(\mathbf{r}^h_u-\mathbf{r_h},\mathbf{v_h})-(z^h_u-z_h,\text{div}\,\mathbf{v_h})&=0 \ \ \forall \ \mathbf{v_h}\in V_h , \label{des3}\\
(w_h,\text{div}(\mathbf{r}^h_u-\mathbf{r_h}))&=(y^h_f(u)-y_h,w_h) \ \ \forall\  w_h\in W_h . \label{des4}
\end{align}
By substituting $\mathbf{v_h}= \mathbf{r_h}-\mathbf{r}^h_u\in V_h$ in \eqref{des1} and using \eqref{des2}-\eqref{des4}, we find that
\begin{align}
    \langle(\mathbf{r_h}-\mathbf{r}^h_u)\cdot n, u_h-u\rangle &=(\mathbf{p}^h_f(u)-\mathbf{p_h},\mathbf{r_h}-\mathbf{r}^h_u)-(y^h_f(u)-y_h, \text{div}(\mathbf{r_h}-\mathbf{r}^h_u))\notag\\
    &=-(y^h_f(u)-y_h,y_h-y^h_f(u))\notag\\
    &=\|y_h-y^h_f(u)\|^2.\label{des5}
\end{align}
Let $u_I \in Q_h$ be the nodal Lagrange interpolation of $u$, see \cite[p. 81]{ciarletbook}. Using equations \eqref{b14} and \eqref{c8} for $q=u-u_h$ and $q_h=u_h-u_I$ respectively, we find
\begin{align}
    \alpha a( u, u-u_h)+\langle \mathbf{r}\cdot n,u-u_h\rangle &=0,\label{des6}\\
    \alpha a(u_h, u_h-u_I)+\langle \mathbf{r_h}\cdot n,u_h-u_I\rangle &=0.\label{des7}
\end{align}
From equations \eqref{des5}-\eqref{des7}, we observe that
\begin{align}
        \alpha|u-u_h|^2_{1,\Omega} + \|y^h_f(u)-y_h\|^2&=\alpha a(u, u-u_h)+\langle \mathbf{r}^h_u\cdot n,u-u_h\rangle \notag\\ &\quad- \alpha a(u_h, u-u_h)- \langle \mathbf{r_h}\cdot n,u-u_h\rangle\notag\\
        &=\alpha a(u, u-u_h)+ \langle \mathbf{r}\cdot n,u-u_h\rangle +\langle (\mathbf{r}^h_u-\mathbf{r})\cdot n,u-u_h\rangle \notag\\ 
        &\quad- \alpha a(u_h, u-u_h)- \langle \mathbf{r_h}\cdot n,u-u_h\rangle \notag\\
        &=\langle (\mathbf{r}^h_u-\mathbf{r})\cdot n,u-u_h\rangle +\alpha a(u_h, u_h-u_I)\notag\\&\quad+\alpha a(u_h, u_I-u)- \langle \mathbf{r_h}\cdot n,u-u_I\rangle + \langle \mathbf{r_h}\cdot n,u_h-u_I\rangle\notag\\
        &=\langle (\mathbf{r}^h_u-\mathbf{r})\cdot n,u-u_h\rangle +\alpha a(u_h, u_I-u)+\langle \mathbf{r_h}\cdot n,u_I-u\rangle\notag\\
        &=I+II+III.\label{d10}
\end{align}

We now  estimate each term on the right hand side of \eqref{d10} as follows.
\begin{align}
     I& =\langle (\mathbf{r}^h_u-\mathbf{r}^h_y)\cdot n,u-u_h\rangle +\langle (\mathbf{r}^h_y-\mathbf{r})\cdot n,u-u_h\rangle\notag\\
     &=I^a+I^b.\label{eq:Est1}
\end{align}
Using Lemma \ref{l1}, Lemma \ref{rtinter} and equations \eqref{b13}, \eqref{d6}, we find
\begin{equation*}
\begin{split}
    I^b= \langle (\mathbf{r}^h_y-\mathbf{r})\cdot n,u-u_h\rangle &=\int_{\Omega}\text{div}(\mathbf{r}^h_y-\mathbf{r})\,(u-u_h)\,dx+\int_{\Omega}(\mathbf{r}^h_y-\mathbf{r})\cdot \nabla(u-u_h)\,dx\\
    &=\int_{\Omega}(\text{div}\,\mathbf{r}^h_y)\,(P_h(u-u_h))\,dx-\int_{\Omega}(\text{div}\,\mathbf{r})\,(u-u_h)\,dx\\&\quad
    +\int_{\Omega}(\mathbf{r}^h_y-\mathbf{r})\cdot \nabla(u-u_h)\,dx\\
    &=\int_{\Omega}(y-y_d)\,(P_h(u-u_h)-(u-u_h))\,dx+\int_{\Omega}(\mathbf{r}^h_y-\mathbf{r})\cdot \nabla(u-u_h)\,dx.
\end{split}
\end{equation*}
Now, using Cauchy-Schwarz inequality, Lemma \ref{stdmixed}, Lemma \ref{rtinter} and Young's inequality, we find that
\begin{align}
    I^b&\leq \|y-y_d\|\|P_h(u-u_h)-(u-u_h)\|+\|\mathbf{r}^h_y-\mathbf{r}\||u-u_h|_{1,\Omega}\notag\\
    &\leq Ch\left(\|y-y_d\|+\|z\|_{2,\Omega}\right)|u-u_h|_{1,\Omega}\notag\\
    &\leq C(\delta)h^2\left(\|y-y_d\|^2+\|z\|^2_{2,\Omega}\right)+\delta|u-u_h|_{1,\Omega}^2.\label{d12}
\end{align}
Taking $v_h=\mathbf{r}^h_u-\mathbf{r}^h_y$ in \eqref{d1} and \eqref{c4} and subtract the resulting equations to find
\begin{equation*}
\begin{split}
    \langle (\mathbf{r}^h_u-\mathbf{r}^h_y)\cdot n,u-u_h\rangle&=\left(y^h_f(u)-y_h,\text{div}(\mathbf{r}^h_u-\mathbf{r}^h_y)\right)-\left(\mathbf{p}^h_f(u)-\mathbf{p_h},\mathbf{r}^h_u-\mathbf{r}^h_y\right).
\end{split}
\end{equation*}
In view of the equations \eqref{d1}-\eqref{d6} and  \eqref{c4}-\eqref{c8}, we have 
\begin{equation*}
    \begin{split}
       I^a= \langle (\mathbf{r}^h_u-\mathbf{r}^h_y)\cdot n,u-u_h\rangle &=\left(\text{div}(\mathbf{r}^h_u-\mathbf{r}^h_y),y^h_f(u)-y_h\right)-\left(\mathbf{p}^h_f(u)-\mathbf{p_h},\mathbf{r}^h_u-\mathbf{r}^h_y\right)
        \\
        &=\left(y^h_f(u)-y,y^h_f(u)-y_h\right)-\left(z^h_u-z^h_y,\text{div}(\mathbf{p}^h_f(u)-\mathbf{p_h})\right)\\
        &=\left(y^h_f(u)-y,y^h_f(u)-y_h\right).
    \end{split}
\end{equation*}
Using Cauchy-Schwarz inequality, Young's inequality and Lemma \ref{stdmixed}, we get that
\begin{align}
        I^a & \leq \|y^h_f(u)-y\|\|y^h_f(u)-y_h\| \notag\\
        &\leq C(\delta)\|y^h_f(u)-y\|^2+\delta\|y^h_f(u)-y_h\|^2 \notag\\
        & \leq C(\delta)h^2\|y\|_{2,\Omega}^2 +\delta\|y^h_f(u)-y_h\|^2.\label{d15}
\end{align}
Combining \eqref{eq:Est1}, \eqref{d12} and \eqref{d15}, we find
\begin{equation}\label{d16}
     |I| \leq C(\delta)h^2\left(\|z\|_{2,\Omega}^2+\|y\|_{2,\Omega}^2+\|y-y_d\|^2\right) +\delta\left(\|y^h_f(u)-y_h\|^2+|u-u_h|^2_{1, \Omega}\right).
\end{equation}
Now we proceed to estimate $II$ and $III$.
Substituting $q=u_I-u\in Q$ in \eqref{b14} , we find
\begin{equation}\label{des20}
  \alpha a(u, u_I-u)+\langle \mathbf{r}\cdot n,u_I-u\rangle=0.  
\end{equation}
In view of \eqref{des20}, we can write
\begin{align}
   II+III&=\alpha a(u_h-u, u_I-u)+\langle( \mathbf{r_h}-\mathbf{r})\cdot n,u_I-u\rangle\notag\\
    &=\alpha a(u_h-u, u_I-u)+\langle( \mathbf{r_h}-\mathbf{r}^h_y)\cdot n,u_I-u\rangle\notag\\
    &\quad+\langle( \mathbf{r}^h_y-\mathbf{r})\cdot n,u_I-u\rangle\notag\\
    &=II^a+II^b+II^c.\label{d18}
\end{align}
Using Cauchy-Schwarz inequality, Young's inequality and approximation property of $u_I$ \cite[Theorem 3.1.4]{ciarletbook}, we have 
\begin{align}
   II^a= \alpha a(u_h-u, u_I-u) &\leq C\|\nabla(u_h-u)\|\|\nabla(u_I-u)\| \notag\\
    & \leq \delta|u-u_h|^2_{1,\Omega} + C(\delta)h^2|u|^2_{2,\Omega}.\label{d19}
\end{align}
Again, by Lemma \ref{l1}, Lemma \ref{rtinter} and equations \eqref{b13}, \eqref{d6}, we find that
\begin{equation*}
    \begin{split}
        II^c = \langle( \mathbf{r}^h_y-\mathbf{r})\cdot n,u_I-u\rangle &=\int_{\Omega}\text{div}(\mathbf{r}^h_y-\mathbf{r})\,(u_I-u)\,dx+\int_{\Omega}(\mathbf{r}^h_y-\mathbf{r})\cdot \nabla(u_I-u)\,dx\\
    &=\int_{\Omega}(\text{div}\,\mathbf{r}^h_y)\,(P_h(u_I-u))\,dx-\int_{\Omega}(\text{div}\,\mathbf{r})\,(u_I-u)\,dx\\&\quad
    +\int_{\Omega}(\mathbf{r}^h_y-\mathbf{r})\cdot \nabla(u_I-u)\,dx\\
    &=\int_{\Omega}(y-y_d)\,(P_h(u_I-u)-(u_I-u))\,dx+\int_{\Omega}(\mathbf{r}^h_y-\mathbf{r})\cdot \nabla(u_I-u)\,dx.
    \end{split}
\end{equation*}
Using Cauchy-Schwarz inequality, approximation properties of $u_I$ and Lemma \ref{stdmixed}, we have
\begin{align}
      II^c&\leq \|y-y_d\|\|P_h(u_I-u)-(u_I-u)\|+\|\mathbf{r}^h_y-\mathbf{r}\||u_I-u|_{1,\Omega}\notag\\
    &\leq Ch\left(\|y-y_d\|+\|z\|_{2,\Omega}\right)|u_I-u|_{1,\Omega}\notag\\
    &\leq Ch^2\left(\|y-y_d\|^2+\|z\|^2_{2,\Omega}+\|u\|_{2,\Omega}^2\right).\label{d20}
\end{align}
Finally, we estimate the term $II^b$. For $u_I \in Q_h$, let $(y^h_f(u_I),\mathbf{p}^h_f(u_I))=S_h(f,u_I) \in W_h \times V_h$ satisfies
\begin{align}
        (\mathbf{p}^h_f(u_I),\mathbf{v_h})-(y^h_f(u_I),\text{div}\,\mathbf{v_h})&=-\langle \mathbf{v_h}\cdot n,u_I\rangle \ \forall \ \mathbf{v_h} \in V_h, \label{d23}\\
      (w_h, \text{div}\,\mathbf{p}^h_f(u_I))&= (f,w_h) \ \forall \ w_h \in W_h.\label{d24}
\end{align}
Upon subtracting equations \eqref{d1} and \eqref{d23}, we find
\begin{equation*}
    \langle \mathbf{v_h}\cdot n,u_I-u\rangle =(y^h_f(u_I)-y^h_f(u),\text{div}\,\mathbf{v_h})-(\mathbf{p}^h_f(u_I)-\mathbf{p}^h_u,\mathbf{v_h}) \ \forall \ \mathbf{v_h} \in V_h.
\end{equation*}
Take $\mathbf{v_h}=\mathbf{r_h}-\mathbf{r}^h_y \in V_h$ in the last equation to get
\begin{equation*}
     II^b=   \langle (\mathbf{r_h}-\mathbf{r}^h_y)\cdot n,u_I-u\rangle =(y^h_f(u_I)-y^h_f(u),\text{div}(\mathbf{r_h}-\mathbf{r}^h_y))-(\mathbf{p}^h_f(u_I)-\mathbf{p}^h_u,\mathbf{r_h}-\mathbf{r}^h_y).
\end{equation*}
Using the equations \eqref{c4}-\eqref{c8}, \eqref{d1}-\eqref{d6}, \eqref{d23}-\eqref{d24}, Cauchy-Schwarz inequality and Young's inequality, we find that
\begin{align}
      II^b&=(y_h-y,y^h_f(u_I)-y^h_f(u))+ (z^h_{y}-z_h,\text{div}(\mathbf{p}^h_f(u_I)-\mathbf{p}^h_f(u)))\notag\\
         &=(y_h-y,y^h_f(u_I)-y^h_f(u))\notag\\
         &\leq C(\delta)\|y^h_f(u_I)-y^h_f(u)\|^2 + \delta\|y-y_h\|^2.\label{d25}
\end{align}
\par \noindent
Now, to estimate $\|y^h_f(u_I)-y^h_f(u)\|$, let $(\tilde{\phi}, \mathbf{x})=S(y^h_f(u_I)-y^h_f(u),0) \in W\times V$ i.e. $(\tilde{\phi}, \mathbf{x})$ solves the following system:
\begin{equation*}
    \begin{split}
      (\mathbf{x},\mathbf{v})-  (\tilde{\phi},\text{div}\,\mathbf{v})&=0  \ \  \forall \ \mathbf{v} \in V,\\
        (w, \text{div}\,\mathbf{x})&=(y^h_f(u_I)-y^h_f(u),w) \ \forall \ \  w \in W.
    \end{split}
\end{equation*}
By elliptic regularity theory on convex polygonal domains, we have $\tilde{\phi} \in H^2(\Omega)$ and 
\begin{equation}\label{d27}
    \|\tilde{\phi}\|_{2,\Omega} + \|\mathbf{x}\|_{H(\text{div},\Omega)}\leq C\|y^h_f(u_I)-y^h_f(u)\|.
\end{equation}
 Further, let $(\tilde{\phi}_h, \mathbf{x_h})=S_h(y^h_f(u_I)-y^h_f(u),0) \in W_h \times V_h$ denote the mixed finite element approximation of $(\tilde{\phi}, \mathbf{x})$, we have
\begin{align}
        (\mathbf{x_h},\mathbf{v_h})-(\tilde{\phi} _h,\text{div}\,\mathbf{v_h})&=0  \ \forall \ \mathbf{v_h} \in V_h, \label{d28}\\
        (w_h, \text{div}\,\mathbf{x_h})&=(y^h_f(u_I)-y^h_f(u),w_h) \ \forall \ w_h \in W_h.\label{d29}
\end{align}
Setting $w_h=y^h_f(u_I)-y^h_f(u) \in W_h$ in \eqref{d29}, we get
\begin{equation}
   \|y^h_f(u_I)-y^h_f(u)\|^2=(y^h_f(u_I)-y^h_f(u),\text{div}\,\mathbf{x_h}).
\end{equation}
By taking $\mathbf{v_h}=\mathbf{p}^h_f(u_I)-\mathbf{p}^h_f(u) \in V_h$ in \eqref{d28}, we find
\begin{equation}\label{d31}
    (\tilde{\phi}_h,\text{div}(\mathbf{p}^h_f(u_I)-\mathbf{p}^h_f(u)))-(\mathbf{x_h},\mathbf{p}^h_f(u_I)-\mathbf{p}^h_f(u))=0.
\end{equation}
Now, using equations \eqref{d23}-\eqref{d24}, \eqref{d1}-\eqref{d4}, \eqref{d31}, Cauchy-Schwarz inequality and approximation properties of $u_I$ \cite[Theorem 3.1.4]{ciarletbook} to find
\begin{equation}\label{d32}
    \begin{split}
     \|y^h_f(u_I)-y^h_f(u)\|^2&=(y^h_f(u_I)-y^h_f(u),\text{div}\,\mathbf{x_h})+(\Tilde{\phi}_h,\text{div}(\mathbf{p}^h_f(u_I)-\mathbf{p}^h_f(u)))-(\mathbf{x_h},\mathbf{p}^h_f(u_I)-\mathbf{p}^h_f(u))\\
     &=(y^h_f(u_I)-y^h_f(u),\text{div}\,\mathbf{x_h})-(\mathbf{x_h},\mathbf{p}^h_f(u_I)-\mathbf{p}^h_f(u))\\
     &=\langle \mathbf{x_h} \cdot n,u_I-u\rangle\\
     &\leq \|u-u_I\|_{L^2(\Gamma)}\|\mathbf{x_h}\|_{L^2(\Gamma)}\\
     &\leq Ch^{{3}/{2}}\|u\|_{2, \Omega}\|\mathbf{x_h}\|_{L^2(\Gamma)}.
    \end{split}
\end{equation}
Let $P_h\mathbf{x}$ be the $L^2$ projection of $\mathbf{x}$ from $L^2(\Omega)$ to $W_h \times W_h$. Using the inverse inequality \cite[Section 4.5]{brennerbook}, Lemma \ref{rtinter} and equation \eqref{d27}, we have 
 \begin{align}\label{above}
         \|\mathbf{x_h}\|_{L^2(\Gamma)}&\leq \|\mathbf{x_h}-P_h\mathbf{x}\|_{L^2(\Gamma)}+\|P_h\mathbf{x}-\mathbf{x}\|_{L^2(\Gamma)}+\|\mathbf{x}\|_{L^2(\Gamma)}\notag\\
         &\leq Ch^{-\frac{1}{2}}\|\mathbf{x_h}-P_h\mathbf{x}\|+ Ch^{\frac{1}{2}}\|\mathbf{x}\|_{1,\Omega}+C\|\mathbf{x}\|_{1,\Omega}\notag\\
         &\leq Ch^{-\frac{1}{2}}(\|\mathbf{x_h}-\mathbf{x}\|+\|\mathbf{x}-P_h\mathbf{x}\|)+C\|\mathbf{x}\|_{1,\Omega}\notag\\
         &\leq Ch^{-\frac{1}{2}}h(\|\mathbf{x}\|_{1,\Omega}+\|\Tilde{\phi}\|_{2,\Omega})+C\|\mathbf{x}\|_{1,\Omega}\notag\\
         &\leq C(\|\mathbf{x}\|_{1,\Omega}+\|\Tilde{\phi}\|_{2,\Omega})\notag\\
         &\leq C\|y^h_f(u_I)-y^h_f(u)\|. 
 \end{align}
 Therefore, using the estimates \eqref{above} and \eqref{d32}, we have
 \begin{equation}\label{d34}
     \|y^h_f(u_I)-y^h_f(u)\|\leq Ch^{{3}/{2}}\|u\|_{2, \Omega}.
 \end{equation}
By Lemma \ref{stdmixed} and triangle inequality, we find
 \begin{align}
     \|y-y_h\|&\leq \|y-y^h_f(u)\|+\|y^h_f(u)-y_h\|\notag\\
     &\leq Ch\|y\|_{2, \Omega}+\|y^h_f(u)-y_h\|. \label{d35}
 \end{align}
 Using equations \eqref{d25}, \eqref{d34} and \eqref{d35} to get 
 \begin{equation}\label{d36}
     \begin{split}
         II^b &\leq Ch^2\left(\|u\|^2_{2, \Omega}+\|y\|^2_{2, \Omega}\right) +\delta\|y^h_f(u)-y_h\|^2.
     \end{split}
 \end{equation}
 Combining the equations \eqref{d16}, \eqref{d19}, \eqref{d20}, \eqref{d36} and setting $\delta$ to be small enough, we find that
 \begin{equation*}
     \begin{split}
         |u-u_h|^2_{1,\Omega} + \|y^h_f(u)-y_h\|^2 &\leq Ch^2\left(\|u\|^2_{2,\Omega}+\|z\|^2_{2, \Omega}+\|y\|^2_{2, \Omega}+\|y-y_d\|^2\right).
     \end{split}
 \end{equation*}
 Take square root on both the sides to conclude the proof. 
\end{proof}
\end{theorem}
\begin{Cor} The following hold:
\begin{equation*}
\begin{split}
     \|y-y_h\| & \leq Ch\left(\|u\|_{2,\Omega}+\|z\|_{2,\Omega}+\|y\|_{2, \Omega}+\|y-y_d\|\right),\\
     \|\mathbf{r-r_h}\|& \leq Ch\left(\|u\|_{2,\Omega}+\|z\|_{2,\Omega}+\|y\|_{2, \Omega}+\|y-y_d\|\right),\\
     \|z-z_h\| & \leq Ch\left(\|u\|_{2,\Omega}+\|z\|_{2,\Omega}+\|y\|_{2, \Omega}+\|y-y_d\|\right).
\end{split}
\end{equation*}
\begin{proof}
Using the triangle inequality, Lemma \ref{stdmixed} and Theorem \ref{energyestimate}, we have
\begin{align} 
        \|y-y_h\|&\leq \|y-y^h_f(u)\|+\|y^h_f(u)-y_h\|\notag\\
        &\leq Ch\left(\|u\|_{2,\Omega}+\|z\|_{2,\Omega}+\|y\|_{2, \Omega}+\|y-y_d\|\right).\label{eq:Esty}
\end{align}
In view of \eqref{c6}-\eqref{c7} and \eqref{d5}-\eqref{d6}, we have the following system
\begin{align}
    (\mathbf{r}^h_y-\mathbf{r_h}, \mathbf{v_h})-(z^h_y-z_h, \text{div}\,\mathbf{v_h})&=0 \ \ \forall \ \mathbf{v_h} \in V_h,\label{babus}\\
    (w_h, \text{div}\,(\mathbf{r}^h_y-\mathbf{r_h}))&=(y-y_h, w_h) \ \ \forall \ w_h  \in W_h.\label{brez}
\end{align}
Using discrete Babuska-Brezzi condition \cite[Chapter II]{brezzifortinbook} to find
\begin{equation}\label{serty}
    \|z^h_y-z_h\|\leq C\|\mathbf{r}^h_y-\mathbf{r_h}\|.
\end{equation}
Substituting $\mathbf{v_h}=\mathbf{r}^h_y-\mathbf{r_h}$ in \eqref{babus} and then using \eqref{brez} to get
\begin{equation*}
    \|\mathbf{r}^h_y-\mathbf{r_h}\|^2=(y-y_h, z^h_y-z_h).
\end{equation*}
A use of Cauchy-Schwarz inequality and \eqref{serty}  gives
\begin{equation*}
    \|\mathbf{r}^h_y-\mathbf{r_h}\|\leq C\|y-y_h\|.
\end{equation*}
Now, using Lemma \ref{stdmixed} and \eqref{eq:Esty}, we find
\begin{equation*}
\begin{split}
    \|\mathbf{r-r_h}\|&\leq  \|\mathbf{r}-\mathbf{r}^h_y\|+ \|\mathbf{r}^h_y-\mathbf{r_h}\|\\
    & \leq C\left(h\,\|z\|_{2,\Omega}+\|y-y_h\|\right)\\
    &\leq Ch\left(\|u\|_{2,\Omega}+\|z\|_{2,\Omega}+\|y\|_{2, \Omega}+\|y-y_d\|\right).
\end{split}
\end{equation*}
Finally, by using Lemma \ref{stdmixed} and \eqref{serty}, we get
\begin{equation*}
    \begin{split}
        \|z-z_h\|& \leq \|z-z^h_y\|+\|z^h_y-z_h\|\\
        &\leq C\left(h\,\|z\|_{2, \Omega}+ \|y-y_h\|\right)\\
        &\leq Ch\left(\|u\|_{2,\Omega}+\|z\|_{2,\Omega}+\|y\|_{2, \Omega}+\|y-y_d\|\right).
    \end{split}
\end{equation*}
This concludes the proof of this Lemma.
\end{proof}
\end{Cor}
We now proceed to obtain some relations which will be helpful in establishing the energy norm estimates for the control.
Let $(y_f,\mathbf{p}_f)=S(f,0)$ and $(y_u, \mathbf{p}_u)=S(0,u)$. By the linearity of the solution operator $S$ and uniqueness of $(y, \mathbf{p})$, we have $y=y_f+y_u$ and $\mathbf{p}=\mathbf{p}_f+\mathbf{p}_u$. From the elliptic regularity theory for convex polygon domains \cite[Theorem 3.1.2.1]{grisvardbook}, we have $y_u\in H^2(\Omega).$ Using the integration by parts and standard density arguments, we find that
\begin{equation}\label{auxx1}
    a(y_u, t)=0 \ \ \forall \ \ t\in H^2(\Omega) \cap H^1_0(\Omega).
\end{equation}
Using the equation \eqref{b14}, we have
\begin{equation}\label{auxx2}
    a(u, t)=0 \ \ \forall \ t\in H^1_0(\Omega).
\end{equation}
In view of the equations \eqref{auxx1}, \eqref{auxx2} and the Poincare inequality, we find that $y_u=u.$ Similarly, we can write $y_h=y^h_f+y^h_{u_h}$ and $\mathbf{p_h}=\mathbf{p}^h_f+\mathbf{p}^h_{u_h}$ where $(y^h_f,\mathbf{p}^h_f)=S_h(f,0)$ and $(y^h_{u_h}, \mathbf{p}^h_{u_h})=S_h(0,u_h).$\\
We prove the consistency result in the following lemma, yielding the perturbed Galerkin orthogonality which plays a key role in deriving the error estimates in the energy as well as $L^2$-norms. 
\begin{lemma}\label{pgo}
For any $q_h \in Q_h$,  let $(y^h_{q_h}, \mathbf{p}^h_{q_h})=S_h(0, q_h)$. Then, it holds that
 \begin{equation*}
    \alpha a(u-u_h, q_h)+(y_u-y^h_{u_h}, y^h_{q_h})=(y_f^h-y_f,y^h_{q_h})+\langle  (\mathbf{r}^h_y-\mathbf{r})\cdot n,q_h\rangle. 
\end{equation*}
\begin{proof}
Let $q_h \in Q_h$, we have $(y^h_{q_h}, \mathbf{p}^h_{q_h})\in W_h\times V_h$ satisfies
\begin{align}
    (\mathbf{p}^h_{q_h},\mathbf{v_h})-(y^h_{q_h},\text{div}\,\mathbf{v_h})&=-\langle \mathbf{v_h}\cdot n,q_h\rangle \ \ \forall \  \mathbf{v_h}\in V_h, \label{l2errr1} \\
(w_h,\text{div}\,\mathbf{p}^h_{q_h})&=0\ \ \forall\  w_h\in W_h. \label{l2errr2}
\end{align}
From \eqref{c8}, we have that $\alpha a(u_h, q_h)=-\langle \mathbf{r_h}\cdot n,q_h \rangle$. Taking $\mathbf{v_h}=\mathbf{r_h}$ in \eqref{l2errr1} and then using the equations \eqref{c6}, \eqref{c7} and \eqref{l2errr2}, we find
\begin{equation*}
\begin{split}
   -\langle \mathbf{r_h}\cdot n ,q_h\rangle &=(\mathbf{p}^h_{q_h},\mathbf{r_h})-(y^h_{q_h},\text{div}\,\mathbf{r_h}) \\
   &=(z_h, \text{div}\,\mathbf{p}^h_{q_h})-(y_h-y_d, y^h_{q_h} )\\
   &=-(y_h-y_d,y^h_{q_h})=-(y^h_f+y^h_{u_h}-y_d, y^h_{q_h}).
\end{split}
\end{equation*}
  Therefore, we have
\begin{equation}\label{imppa}
    \alpha a(u_h, q_h)+(y^h_{u_h}, y^h_{q_h})=(y_d-y^h_f, y^h_{q_h}).
\end{equation}
{Let $(z_{y_u},\mathbf{r}_{y_u})=S(y_u,0),$ $(z_{y_f},\mathbf{r}_{y_f})=S(y_f-y_d,0)$ and their  discrete approximations are $(z^h_{y_u},\mathbf{r}^h_{y_u})=S_h(y_u,0)$ and $(z^h_{y_f},\mathbf{r}^h_{y_f})=S_h(y_f-y_d,0)$, respectively. Using the definitions of $S$ and $S_h$, we have
\begin{align}
    (y_u, y^h_{q_h})&=(y_u, y^h_{q_h}-q_h)+(y_u, q_h)
    =(\text{div}\,\mathbf{r}_{y_u}, y^h_{q_h}-q_h)+(y_u, q_h)\notag\\
    &=(\text{div}\,\mathbf{r}_{y_u}, y^h_{q_h}-P_h q_h)+(\text{div}\,\mathbf{r}_{y_u}, P_h q_h-q_h)+(y_u, q_h)\notag\\
    &=(\text{div}\,\mathbf{r}^h_{y_u}, y^h_{q_h}-P_h q_h)+(\text{div}\,\mathbf{r}_{y_u}, P_h q_h-q_h)+(y_u, q_h),\label{1098}
\end{align}
and 
\begin{align}
    (y_f-y_d, y^h_{q_h})&= (y_f-y_d, y^h_{q_h}-q_h)+(y_f-y_d, q_h)
    =(\text{div}\,\mathbf{r}_{y_f}, y^h_{q_h}-q_h)+(y_f-y_d, q_h)\notag\\
    &=(\text{div}\,\mathbf{r}_{y_f}, y^h_{q_h}-P_h q_h)+(\text{div}\,\mathbf{r}_{y_f},P_h q_h-q_h)+(y_f-y_d, q_h)\notag\\
    &=(\text{div}\,\mathbf{r}^h_{y_f},y^h_{q_h}-P_h q_h)+(\text{div}\,\mathbf{r}_{y_f},P_h q_h-q_h)+(y_f-y_d, q_h).\label{2098}
\end{align}
Adding \eqref{1098} and \eqref{2098} and using the facts that $y=y_u+y_f$, $\text{div}\,\mathbf{r}=\text{div}(\mathbf{r}_{y_u}+\mathbf{r}_{y_f})$ and $\text{div}\,\mathbf{r}^h_{y}=\text{div}(\mathbf{r}^h_{y_u}+\mathbf{r}^h_{y_f})$, we find that
\begin{equation*}
    \begin{split}
        (y-y_d, y^h_{q_h})&=(\text{div}\,\mathbf{r}^h_y, y^h_{q_h}-P_h q_h)+(\text{div}\,\mathbf{r},P_h q_h-q_h)+(y-y_d, q_h)\\
        &=(\text{div}\,\mathbf{r}^h_y, y^h_{q_h})-(\text{div}\,\mathbf{r}^h_y, P_h q_h)+(\text{div}\,\mathbf{r}, P_h q_h)
         -(\text{div}\,\mathbf{r}, q_h)+(y-y_d, q_h)\\
        &=(\text{div}\,\mathbf{r}^h_y, y^h_{q_h}),
    \end{split}
\end{equation*}
where in obtaining the last equation, we have used \eqref{b13} together with the fact that $(\text{div}(\mathbf{r}-\mathbf{r}^h_y),P_h q_h)=0$. Using the definitions of $(y^h_{q_h}, \mathbf{p}^h_{q_h})\in W_h\times V_h,$  $(z^h_y, \mathbf{r}^h_y)\in W_h\times V_h$ and \eqref{b14}, we find
\begin{equation*}
    \begin{split}
      (y-y_d, y^h_{q_h})  &=(\text{div}\,\mathbf{r}^h_y, y^h_{q_h})
      =(\mathbf{p}^h_{q_h}, \mathbf{r}^h_y)+\langle \mathbf{r}^h_y\cdot n,q_h\rangle\\
      &=(z^h_y, \text{div}\,\mathbf{p}^h_{q_h})+\langle(\mathbf{r}^h_y-\mathbf{r})\cdot n,q_h\rangle+\langle \mathbf{r}\cdot n,q_h\rangle\\
      &=\langle (\mathbf{r}^h_y-\mathbf{r})\cdot n,q_h\rangle-\alpha a(u, q_h).
    \end{split}
\end{equation*}
Therefore, we have 
\begin{equation}\label{imppb}
    \alpha a(u, q_h)+(y_u, y^h_{q_h})=-(y_f-y_d, y^h_{q_h})+\langle (\mathbf{r}^h_y-\mathbf{r})\cdot n,q_h\rangle.
\end{equation}
Upon subtracting \eqref{imppa} from \eqref{imppb}, we get the desired result.}
\end{proof}
\end{lemma}
\par \noindent
Next, we obtain an intermediate estimate required for the error analysis.
\begin{lemma}\label{inter1}
For any $(q,q_h)\in Q\times Q_h$, let $(y_q,\mathbf{p}_q)=S(0,q)$ and $(y_{q_h},\mathbf{p}_{q_h})=S(0,q_h)$. Then,
    \begin{equation*}\label{lem6}
     \|\mathbf{p}_q-\mathbf{p}_{q_h}\|+ \|y_q-y_{q_h}\|\leq C\|q-q_h\|_{1,\Omega}.
    \end{equation*}
    \begin{proof}
    Since $(y_q,\mathbf{p}_q)=S(0,q) \in W \times V$ and $(y_{q_h},\mathbf{p}_{q_h})=S(0,q_h) \in W \times V$, we have 
    \begin{align}
        (\mathbf{p}_q-\mathbf{p}_{q_h},\mathbf{v})-(y_q-y_{q_h},\text{div}\,\mathbf{v})&=-\langle  \mathbf{v}\cdot n,q-q_h\rangle \ \ \forall \ \mathbf{v} \in V, \label{lem1}\\
        (w,\text{div}(\mathbf{p}_q-\mathbf{p}_{q_h}))&=0 \ \ \ \ \ \ \ \ \ \ \forall \  w \in W. \label{lem2}
    \end{align}
   Taking $\mathbf{v}=\mathbf{p}_q-\mathbf{p}_{q_h} \in V$ in \eqref{lem1} and using \eqref{lem2} to find
    \begin{equation*}
        \|\mathbf{p}_q-\mathbf{p}_{q_h}\|^2=-\langle  (\mathbf{p}_q-\mathbf{p}_{q_h})\cdot n,q-q_h\rangle.
    \end{equation*}
 From \eqref{lem2}, it is evident  that $\|\text{div}(\mathbf{p}_q-\mathbf{p}_{q_h})\|=0$.  Now, using Cauchy-Schwarz inequality and Lemma \ref{l1}, we find
  \begin{equation}\label{lem3}
      \|\mathbf{p}_q-\mathbf{p}_{q_h}\| \leq C\|q-q_h\|_{\frac{1}{2},\Gamma}\leq C\|q-q_h\|_{1,\Omega}.
  \end{equation}
   Since $y_q-y_{q_h}\in W$, there { exists} $\mathbf{v} \in V$ such that $\text{div}\,\mathbf{v}=y_q-y_{q_h}$ and $\|\mathbf{v}\|_{H(\text{div},\Omega)}\leq \|y_q-y_{q_h}\|$ \cite[Lemma 2.2, p. 6]{duranotes}. Substituting $\text{div}\,\mathbf{v}=y_q-y_{q_h}$ in \eqref{lem1} and then using Cauchy-Schwarz inequality together with \eqref{lem3}, we obtain
  \begin{align}
      \|y_q-y_{q_h}\|^2&=(\mathbf{p}_q-\mathbf{p}_{q_h},\mathbf{v})+\langle \mathbf{v}\cdot n,q-q_h\rangle\notag\\
      &\leq \|\mathbf{p}_q-\mathbf{p}_{q_h}\|\|\mathbf{v}\|+\|q-q_h\|_{\frac{1}{2}, \Gamma}\|\mathbf{v}\cdot n\|_{-\frac{1}{2}, \Gamma}\notag\\
      &\leq \|\mathbf{p}_q-\mathbf{p}_{q_h}\|\|\mathbf{v}\|+\|q-q_h\|_{\frac{1}{2}, \Gamma}\|\mathbf{v}\|_{H(\text{div},\Omega)}\notag\\
      & \leq C\|q-q_h\|_{\frac{1}{2}, \Gamma}\|\mathbf{v}\|_{H(\text{div}, \O)}\leq C\|q-q_h\|_{1,\Omega}\|y_q-y_{q_h}\|.\label{lem4}
  \end{align}
 In view of \eqref{lem3} and \eqref{lem4}, we have the desired estimate of this lemma.
    \end{proof}
    \end{lemma}
    \begin{lemma}\label{inter2}
Let $(y^h_q,\mathbf{p}^h_q)=S_h(0,q)$ and $(y^h_{q_h},\mathbf{p}^h_{q_h})=S_h(0,q_h)$ for $(q,q_h)\in Q\times Q_h$. Then,
    \begin{equation*}\label{lem7}
     \|\mathbf{p}^h_q-\mathbf{p}^h_{q_h}\|+ \|y^h_q-y^h_{q_h}\|\leq C\|q-q_h\|_{1, \Omega}.
     \end{equation*}
     \begin{proof}
     For a given  $w_h\in W_h$, there exists $\mathbf{v_h} \in V_h$ such that $\text{div}\,\mathbf{v_h}=w_h$ and $\|\mathbf{v_h}\|_{H(\text{div},\Omega)}
     \leq \|w_h\|$ \cite[Lemma 3.5, p. 17]{duranotes}.
   Taking this into account, the proof of this lemma follows using the similar arguments as in the Lemma \ref{inter1}.
     \end{proof}
\end{lemma}
\par \noindent
We now derive the error estimates for the control in the energy norm.
\begin{theorem}\label{energynorm}
It holds that,
\begin{equation*}
    \|u-u_h\|_{1, \Omega}\leq Ch\left(\|u\|_{2, \Omega}+\|z\|_{2, \Omega}+\|y\|_{2, \Omega}+\|y-y_d\|\right).
\end{equation*}
\begin{proof}
In view of Theorem \ref{energyestimate}, it suffices to estimate $\|u-u_h\|.$ 
Let $u_I \in Q_h$ be the Lagrange nodal interpolation of $u$. 
Let $q_h=u_I-u_h \in Q_h,$ $(y^h_{q_h}, \mathbf{p}^h_{q_h})=S_h(0, q_h)$ and $(y^h_u, \mathbf{p}^h_u)=S_h(0, u).$ Using Lemma \ref{pgo}, we find
\begin{align} \label{eq:AuxE}
        \|y^h_{q_h}\|^2+ \alpha|q_h|^2_{1, \Omega}&=(y^h_u-y_u, y^h_{q_h})+(y^h_{u_I}-y^h_u, y^h_{q_h})+ \alpha a(u_I-u, q_h) \notag\\
        & \quad +(y^h_f-y_f, y^h_{q_h})+\langle (\mathbf{r}^h_y-\mathbf{r})\cdot n,q_h \rangle,
\end{align}
where $(y^h_{u_I}, \mathbf{p}^h_{u_I})=S_h(0, u_I).$ We now estimate each term in the right hand side of \eqref{eq:AuxE}. Using Cauchy-Schwarz inequality, Young's inequality and Lemma \ref{l2}, we find that
\begin{align}
    (y^h_u-y_u, y^h_{q_h}) & \leq \|y^h_u-y_u\|\|y^h_{q_h}\| \leq C(\delta)h^2\|y_u\|^2_{2, \Omega}+ \delta\,\|y^h_{q_h}\|^2, \label{eq:AuxE1}\\
    (y^h_f-y_f, y^h_{q_h}) & \leq \|y^h_f-y_f\|\|y^h_{q_h}\| \leq C(\delta)h^2\|y_f\|^2_{2, \Omega}+ \delta\,\|y^h_{q_h}\|^2.\label{eq:AuxE2}
\end{align}
Again using Cauchy-Schwarz inequality, Young's inequality, Lemma \ref{inter2} and approximation properties of $u_I$, we have
\begin{align*}
(y^h_{u_I}-y^h_u, y^h_{q_h}) & \leq \|y^h_{u_I}-y^h_u\|\|y^h_{q_h}\| \leq C\|u_I-u\|_{1, \Omega}\|y^h_{q_h}\|\leq C(\delta)h^2\|u\|^2_{2, \Omega}+ \delta\,\|y^h_{q_h}\|^2,
\\
\alpha a(u_I-u, q_h)& \leq C|u-u_I|_{1, \Omega}|q_h|_{1,\Omega}
\leq C(\delta)h^2\|u\|^2_{2, \Omega}+ \delta\,|q_h|_{1,\Omega}^2.
\end{align*}
Lastly, we consider $\langle (\mathbf{r}^h_y-\mathbf{r})\cdot n ,q_h\rangle.$ In view of Lemma \ref{l1}, Lemma \ref{rtinter}, equations \eqref{b13} and \eqref{d6}, we observe that
\begin{equation*}
    \begin{split}
        \langle (\mathbf{r}^h_y-\mathbf{r})\cdot n,q_h \rangle &=\int_{\Omega}\text{div}\,(\mathbf{r}^h_y-\mathbf{r})\,q_h\,dx+\int_{\Omega}(\mathbf{r}^h_y-\mathbf{r})\cdot \nabla q_h\,dx\\
        &=\int_{\Omega}(\text{div}\,\mathbf{r}^h_y)\,(P_h\,q_h)\,dx {-}\int_{\Omega}(\text{div}\,\mathbf{r})\,q_h\,dx+\int_{\Omega}(\mathbf{r}^h_y-\mathbf{r})\cdot \nabla q_h\,dx\\
        &=\int_{\Omega}(y-y_d)\,(P_h\,q_h-q_h)\,dx+\int_{\Omega}(\mathbf{r}^h_y-\mathbf{r})\cdot \nabla q_h\,dx.
    \end{split}
\end{equation*}
Therefore, using Cauchy-Schwarz inequality, Young's inequality, Lemma \ref{rtinter} and Lemma \ref{stdmixed}, we get
\begin{align}
         \langle  (\mathbf{r}^h_y-\mathbf{r})\cdot n ,q_h\rangle& \leq \|y-y_d\|\|P_h\,q_h-q_h\|+\|\mathbf{r}^h_y-\mathbf{r}\||q_h|_{1, \Omega} \notag\\
         &\leq Ch\left(\|y-y_d\|+\|z\|_{2, \Omega}\right)|q_h|_{1, \Omega} \notag\\
         &\leq C(\delta)h^2\left(\|z\|^2_{2, \Omega}+\|y-y_d\|^2\right)+ \delta\,|q_h|_{1,\Omega}^2. \label{eq:AuxE5}
\end{align}
Combining estimates \eqref{eq:AuxE1}-\eqref{eq:AuxE5} together with \eqref{eq:AuxE} and choosing  $\delta$ sufficiently small, we obtain
\begin{equation}\label{plot}
    \|y^h_{q_h}\|+|q_h|_{1, \Omega}\leq Ch\left(\|u\|_{2, \Omega}+\|y\|_{2, \Omega}+\|z\|_{2, \Omega}+\|y-y_d\|\right).
\end{equation}
Further, let $(y(\psi), \mathbf{p(\psi)})=S(\psi, 0)$ and $(y^h(\psi), \mathbf{p^h(\psi)})=S_h(\psi, 0),$ where $\psi=P_h\,q_h-y^h_{q_h}.$ By elliptic regularity results 
\begin{equation}\label{regu}
    \|\mathbf{p(\psi)}\|_{H(\text{div}, \Omega)}+\|y(\psi)\|_{2, \Omega}\leq C\|\psi\|.
\end{equation}
We see that $(y^h(\psi), \mathbf{p^h(\psi)})=S_h(\psi, 0)$ satisfy the following system of equations:
\begin{align}
   (\mathbf{p^h(\psi)},\mathbf{v_h})-(y^h(\psi),\text{div}\,\mathbf{v_h})&=0 \ \ \forall \ \mathbf{v_h}\in V_h , \notag\\
(w_h,\text{div}\,\mathbf{p^h(\psi)})&=(\psi, w_h) \ \ \forall\  w_h\in W_h . \label{dip6} 
\end{align}
Taking $w_h=\psi$ in \eqref{dip6} and using Lemma \ref{rtinter} and integration by parts, we get
\begin{equation*}
    \begin{split}
        \|\psi\|^2&=(\text{div}\,\mathbf{p^h(\psi)}, \psi)
        =(\text{div}\,\mathbf{p^h(\psi)}, P_h\,q_h)-(y^h_{q_h}, \text{div}\,\mathbf{p^h(\psi)})\\
        &=(\text{div}\,\mathbf{p^h(\psi)}, q_h)-(\mathbf{p}^h_{q_h}, \mathbf{p^h(\psi)})-\langle  \mathbf{p^h(\psi)}\cdot n,q_h\rangle\\
       &=-(\mathbf{p^h(\psi)}, \nabla q_h)+\langle \mathbf{p^h(\psi)}\cdot n, q_h \rangle-(\mathbf{p}^h_{q_h}, \mathbf{p^h(\psi)})-\langle  \mathbf{p^h(\psi)}\cdot n,q_h\rangle\\
        &=-\left(\mathbf{p^h(\psi)}, \nabla q_h\right)-\left(y^h(\psi),\text{div}\,\mathbf{p}^h_{q_h}\right)\\
        &=-\left(\mathbf{p^h(\psi)}, \nabla q_h\right).
    \end{split}
\end{equation*}
Now using Cauchy-Schwarz inequality, \eqref{plot}, Lemma \ref{l2} and \eqref{regu}, we find
\begin{align*}
    \|\psi\|^2 &\leq Ch\left(\|u\|_{2, \Omega}+\|y\|_{2, \Omega}+\|z\|_{2, \Omega}+\|y-y_d\|\right)\|\mathbf{p^h(\psi)}\|\\
    &\leq Ch\left(\|u\|_{2, \Omega}+\|y\|_{2, \Omega}+\|z\|_{2, \Omega}+\|y-y_d\|\right)(\|\mathbf{p^h(\psi)-p(\psi)}\|+\|\mathbf{p(\psi)}\|)\\
     &\leq Ch\left(\|u\|_{2, \Omega}+\|y\|_{2, \Omega}+\|z\|_{2, \Omega}+\|y-y_d\|\right)(h\|y(\psi)\|_{2, \Omega}+\|\mathbf{p(\psi)}\|)\\
     &\leq Ch\left(\|u\|_{2, \Omega}+\|y\|_{2, \Omega}+\|z\|_{2, \Omega}+\|y-y_d\|\right)\|\psi\|.
\end{align*}
Thus, we obtain
\begin{equation}\label{dx1}
    \|\psi\| \leq Ch\left(\|u\|_{2, \Omega}+\|y\|_{2, \Omega}+\|z\|_{2, \Omega}+\|y-y_d\|\right).
\end{equation}
Therefore, using  Lemma \ref{rtinter}, \eqref{dx1} and \eqref{plot}, we have
\begin{align}
   \|u_I-u_h\|=  \|q_h\| & \leq \|q_h-P_h\,q_h\|+\|P_h\,q_h-y^h_{q_h}\|+\|y^h_{q_h}\|\notag\\
    &\leq Ch|q_h|_{1, \Omega}+ \|P_h\,q_h-y^h_{q_h}\|+\|y^h_{q_h}\|\notag\\
    &\leq Ch\left(\|u\|_{2, \Omega}+\|y\|_{2, \Omega}+\|z\|_{2, \Omega}+\|y-y_d\|\right).\label{asm2}
\end{align}
Finally, a use of triangle inequality, approximation properties of $u_I$ and \eqref{asm2} yields
\begin{equation}\label{asm3}
    \|u-u_h\|\leq Ch\left(\|u\|_{2, \Omega}+\|y\|_{2, \Omega}+\|z\|_{2, \Omega}+\|y-y_d\|\right).
\end{equation}
The proof of this theorem is completed by taking into account \eqref{asm3} and Theorem \ref{energyestimate}.
\end{proof}
\end{theorem}
\par \noindent
Next, we proceed to obtain some auxiliary results in order to derive  $L^2$-norm error estimates. We introduce an enriched discrete optimal control $\tilde{u}_h$ \cite{chowdhurykkt} which is essential for the subsequent analysis. 
For any boundary edge $e$ which is such that $e=\Gamma \cap \partial T$ for some $T \in \cT_h$, let $\mathbf{j_e} \in \mathbb{P}_6(T)$ be the bubble function defined by $\mathbf{j_e}=b_i^2b_j^2(3b_i-1)(3b_j-1)$ with $b_i$, $b_j$ as the barycentric coordinates associated to the endpoints of the edge $e$. Now, for $T \in \cT_h$,  if $T$ shares an edge $e$ with $\Gamma$,  we define the enriched Hermite element $\mathcal{EH}(T)$ on $T$ by 
\begin{equation*}
    \mathcal{EH}(T):=\mathcal{H}(T) \oplus span\{\mathbf{j_e}\},
\end{equation*}
where $\mathcal{H}(T)$ denotes the cubic Hermite finite element space defined on $T$.
 The enriched cubic Hermite finite element space is then defined by 
\begin{equation*}
    \mathcal{EQ}_h:=\{w\in C^0(\Bar{\Omega}): v|_{T} \in \mathcal{Q}(T) \ \ \forall \ T \in \cT_h\},
\end{equation*}
where $\mathcal{Q}(T)=\mathcal{H}(T)$ if $T$ does not share an edge to $\Gamma$ and $\mathcal{Q}(T)=\mathcal{EH}(T)$ if $T$ shares an edge $e$ to $\Gamma$. The enriched discrete optimal control $\tilde{u}_h\in \mathcal{EQ}_h$ is defined by averaging as follows:
\begin{align}
    \tilde{u}_h(v)&=u_h(v) \ \ \ \ \forall \ v  \ \in \cV_h,\nonumber\\
    \tilde{u}_h(b_T)&=u_h(b_T)\  \ \ \forall \ T \ \in \cT_h,\nonumber\\
    \nabla \tilde{u}_h(v)&=\frac{1}{| \cT_v|}\sum_{T\in \cT_v}\nabla (u_h|_{T})(v) \ \ \forall \ v \in \cV_h,\nonumber
    \\
    \int_{e}\tilde{u}_h\,ds&=\int_{e}u_h\,ds \ \ \ \ \forall \ e \in \cE^b_h,\label{boundaryvanish}
\end{align}
where $b_T$ denotes the barycenter of the element $T$.
We state the approximation properties of the enriched discrete optimal control in the following lemma. For the discrete optimal control $u_h$ and the enriched discrete optimal control $\tilde{u}_h$, we have the following estimates.
\begin{lemma}\label{lemmaenriching}
Let $\cT_h$ be a quasi uniform triangulation. Then, it holds that
\begin{equation*}
    \|u_h-\tilde{u}_h\|+h\|u_h-\tilde{u}_h\|_{1, \Omega}\leq Ch(\|u-u_h\|_{1, \Omega}+h\|u\|_{2, \Omega}).
\end{equation*}
\begin{proof}
We refer the readers to the article \cite{chowdhurykkt} for the proof. 
\end{proof}
\end{lemma}
\begin{lemma}\label{crucial}
Let $(y_{\tilde{u}_h},\mathbf{p}_{\tilde{u}_h})=S(0,\tilde{u}_h)\in { W\times V}$. Then, 
\begin{equation*}
    \|y_{\tilde{u}_h}-\tilde{u}_h\|\leq Ch\left(\|u-u_h\|_{1, \Omega}+h\|u\|_{2, \Omega}\right).
\end{equation*}
\begin{proof}
Note that $\tilde{u}_h|_{\Gamma_i}\in C^1(\Bar{\Gamma_i})$ for all $1\leq i\leq k$, therefore $\tilde{u}_h|_{\Gamma_i}\in H^{{3}/{2}}(\Gamma_i)$ for all $1\leq i\leq k$. Also $\tilde{u}_h$ is a continuous function on $\Bar{\Omega}$, we have $y_{\tilde{u}_h}\in H^2(\Omega)$ by the elliptic regularity theory on convex polygonal domains.
Since $(y_{\tilde{u}_h},\mathbf{p}_{\tilde{u}_h})=S(0,\tilde{u}_h)$, they satisfy the following system of equations:
\begin{align}
    (\mathbf{p}_{\tilde{u}_h},\mathbf{v})-(y_{\tilde{u}_h},\text{div}\,\mathbf{v})&=-\langle \mathbf{v}\cdot n,\tilde{u}_h\rangle \ \ \forall \ \mathbf{v} \ \in V,\label{lem61}\\
    (w,\text{div}\,\mathbf{p}_{\tilde{u}_h})&=0 \ \ \forall \ w \ \in W.\label{lem62}
\end{align}
Taking $\mathbf{v}=\nabla x$ in \eqref{lem61} for $x \in \mathcal{D}(\Omega)$, we get
\begin{equation*}
     (\mathbf{p}_{\tilde{u}_h},\nabla x)-(y_{\tilde{u}_h},\Delta x)=-\langle \nabla x\cdot n,\tilde{u}_h\rangle .
\end{equation*}
Integrating by parts and using \eqref{lem62}, we get
\begin{equation*}
    a(y_{\tilde{u}_h},x)=0 \ \ \forall \ x \ \in \mathcal{D}(\Omega).
\end{equation*}
Using the density of $\mathcal{D}(\Omega)$ in $H^1_0(\Omega)$, we conclude that 
\begin{equation}\label{lem63}
    a(y_{\tilde{u}_h},x)=0 \ \ \forall \ x \ \in H^1_0(\Omega).
\end{equation}
Let $w_0:=y_{\tilde{u}_h}-\tilde{u}_h \in H^1_0(\Omega)$. Using \eqref{lem63} and \eqref{b14}, we find
\begin{equation*}
    a(w_0,w_0)=-a(\tilde{u}_h,w_0)=a(u-\tilde{u}_h,w_0).
\end{equation*}
Therefore, using Cauchy-Schwarz inequality together with Lemma \ref{lemmaenriching}, we find
\begin{equation}\label{leo}
    |w_0|_{1, \Omega}\leq C\left(|u-u_h|_{1,\Omega}+|u_h-\tilde{u}_h|_{1,\Omega}\right)\leq C\left(\|u-u_h\|_{1,\Omega}+h\|u\|_{2,\Omega}\right).
\end{equation}
\par \noindent
 In order to obtain the $L^2$-norm error estimates, we make use of the Aubin-Nitsche duality techniques. Let $w\in H^1_0(\Omega)$ be the solution of 
\begin{align*}
     \begin{cases}
       -\Delta w &=w_0 \ \ \text{in}\ \  \Omega, \\
     \hspace{5mm}   w &=0 \ \ \text{on} \ \ \Gamma.
     \end{cases}
 \end{align*}
  By the elliptic regularity theory on the convex polygonal domains, $w\in H^2(\Omega)$ and $\|w\|_{2, \Omega}\lesssim \|w_0\|$. Let $w_I\in Q_h \cap H^1_0(\Omega)$ be the Lagrange interpolation of $w$. Since $w_0\in H^1_0(\Omega)$, we find
 \begin{equation*}
 \begin{split}
     \|w_0\|^2&=a(w_0,w) 
     =a(w_0,w-w_I)+a(w_0,w_I)
     =a(w_0,w-w_I)-a(\tilde{u}_h,w_I)\\
     &=a(w_0,w-w_I)+a(u_h-\tilde{u}_h,w_I-w)+a(u_h-\tilde{u}_h,w)-a(u_h,w_I)\\
     &=I+II+III+IV.
 \end{split}
 \end{equation*}
 Using Cauchy-Schwarz inequality, approximation properties of $w_I$, Lemma \ref{lemmaenriching} and \eqref{leo}, we find 
 \begin{align}
     I+II&\leq \left(|w_0|_{1, \Omega}+|u_h-\tilde{u}_h|_{1, \Omega}\right)|w-w_I|_{1, \Omega}\notag\\
     &\leq Ch\|w\|_{2,\Omega}\left(\|u-u_h\|_{1,\Omega}+h\|u\|_{2,\Omega}\right).\label{first}
 \end{align}
 Now, we apply integration by parts on $III$ to get 
 \begin{equation*}
     III=a(u_h-\tilde{u}_h,w)=-\int_{\Omega}(u_h-\tilde{u}_h)\Delta w\,dx+\int_{\partial \Omega}(u_h-\tilde{u}_h)\frac{\partial w}{\partial n}\,ds.
\end{equation*}
 We have $\int_{e}\tilde{u}_h\,ds=\int_{e}u_h\,ds  \ \forall \ e \in \cE^b_h.$ Hence we can rewrite $III$ as follows
 \begin{equation*}
     III=-\int_{\Omega}(u_h-\tilde{u}_h)\Delta w\, dx+\int_{\partial \Omega}(u_h-\tilde{u}_h)\frac{\partial (w-w_I)}{\partial n}\,ds.
 \end{equation*}
 A use of Cauchy-Schwarz inequality, approximation property of $w_I$, trace inequality and Lemma \ref{lemmaenriching} yields
 \begin{equation}\label{second}
     III \leq Ch\|w\|_{2,\Omega}\left(\|u-u_h\|_{1,\Omega}+h\|u\|_{2,\Omega}\right).
 \end{equation}
 Since $w_I\in H^1_0(\Omega)$, equation \eqref{c8} implies $IV=0$. Now,  the proof is completed by taking into account \eqref{first}, \eqref{second} and the estimate $\|w\|_{2,\Omega}\lesssim \|w_0\|$.
\end{proof}
\end{lemma}
\begin{lemma}\label{lemm1}
Let $(s_h,\mathbf{k}) \in W_h \times V$ be such that
\begin{align}
(\mathbf{k},\mathbf{v_h})-(s_h,\text{div}\,\mathbf{v_h})&=0 \ \ \forall \  \mathbf{v_h}\in V_h, \label{c111} \\
(w_h,\text{div}\,\mathbf{k})&=0\ \ \forall\  w_h\in W_h. \label{c121}
\end{align}
Then, there { exists} a positive constant $C$ such that
\begin{equation*}
    \|s_h\|\leq Ch\|\mathbf{k}\|_{H(\text{div},\Omega)}.
\end{equation*}
\begin{proof}
We apply the Aubin-Nitsche duality arguments. Let $w\in H^1_0(\Omega)$ be the weak solution of 
\begin{align*}
     \begin{cases}
       \Delta w &=w_1 \ \ \text{in}\ \  \Omega, \\
     \hspace{5mm}   w &=0 \ \ \text{on} \ \ \Gamma.
     \end{cases}
 \end{align*}
 where $w_1\in L^2(\Omega)$ is a given function.
Convexity of the domain $\Omega$ implies that $w\in H^2(\Omega)$ and $\|w\|_{2, \Omega}\lesssim \|w_1\|$. Using \eqref{interr1} and \eqref{c111}, we have
\begin{align}\label{eq:Eq11}
        (s_h,w_1)&=(s_h,\text{div}\,\nabla w)
        =(s_h,\text{div}(\Pi_h(\nabla w))) =(\mathbf{k},\Pi_h(\nabla w)) \notag\\
        &=(\mathbf{k},\Pi_h(\nabla w)-\nabla w)+(\mathbf{k},\nabla w) \notag\\
        &=I+II.
\end{align}
Using Cauchy-Schwarz inequality and Lemma \ref{rtinter}, we get
\begin{equation}\label{eq:Eq12}
    I\leq \|\mathbf{k}\|\|\Pi_h(\nabla w)-\nabla w\|\leq Ch\|\nabla w\|_{1,\Omega}\|\mathbf{k}\|\leq Ch\|w_1\|\|\mathbf{k}\|_{H(\text{div},\Omega)}.
\end{equation}
To estimate the term $II$, a use of integration by parts and \eqref{c121} yields
\begin{equation*}
    II=(\mathbf{k},\nabla w)=-(\text{div}\,\mathbf{k},w)=(\text{div}\,\mathbf{k},P_h w-w).
\end{equation*}
Now, applying Cauchy-Schwarz inequality and Lemma \ref{rtinter} to get
\begin{equation}\label{eq:Eq13}
    II\leq \|\text{div}\,\mathbf{k}\|\|P_h w-w\|\leq Ch\|w\|_{1,\Omega}\|\mathbf{k}\|_{H(\text{div},\Omega)}\leq Ch\|w_1\|\|\mathbf{k}\|_{H(\text{div}, \Omega)}.
\end{equation}
Combining \eqref{eq:Eq11}-\eqref{eq:Eq13} completes the proof of this lemma.
\end{proof}
\end{lemma}
\par \noindent
In the following theorem, we establish $L^2$-norm error estimates for the control  by exploiting the Aubin-Nitsche duality argument.
\begin{theorem}\label{l2estimates}
For the optimal control $u \in Q$ and discrete optimal control $u_h\in Q_h$, the following estimate holds
\begin{equation*}
    \|u-u_h\|\leq Ch^2\left(\|u\|_{2, \Omega}+\|z\|_{2, \Omega}+\|y\|_{2, \Omega}+\|y-y_d\|\right).
\end{equation*}
\begin{proof}
We begin the proof by defining the following auxiliary problem: to find $\phi\in Q$ such that
\begin{equation}\label{minauxi}
    J_a(\phi)=\min_{t\in Q} J_a(t),
\end{equation}
where $J_a(t):=\frac{1}{2}\|S_1(0,t)-(u-u_h)\|^2+\frac{\alpha}{2}\|\nabla t\|^2$ and $S_1(0,t)$ denotes the first component of $S(0,t).$ { By the optimal control theory \cite[Theorem 2.14]{trolzbook}, there exists a unique solution $\phi\in Q$ of the minimization problem \eqref{minauxi}. Using the first order necessary optimality conditions, the unique minimizer $\phi \in Q$ satisfies the following optimality condition}
\begin{equation}\label{NOC}
    \alpha a(t,\phi)+(y_t,y_\phi)=(u-u_h,y_t)\ \ \forall \ t \ \in Q,
\end{equation}
where $(y_t,\mathbf{p}_t)=S(0,t)$ and $(y_\phi,\mathbf{p_\phi})=S(0,\phi)$.
The optimality equation \eqref{NOC} can equivalently be written as
\begin{equation}\label{eqauxi}
    \alpha a(\phi, t)=-\langle \mathbf{r}_a\cdot n, t\rangle \ \ \forall \ t \in Q,
\end{equation}
where $(z_a,\mathbf{r}_a)=S(y_\phi-(u-u_h),0)$. By taking $t=1$ in \eqref{eqauxi}, find that 
\begin{equation*}
    \langle \mathbf{r}_a\cdot n, 1\rangle=0,
\end{equation*}
which is the compatibility condition for the Neumann problem weakly solved by $\phi$
\begin{equation*}
 \begin{cases}
 \begin{split}
  -\Delta \phi&=0 \ \ \text{in} \ \ \Omega,\\
  \alpha \frac{\partial \phi}{\partial n}&= - \mathbf{r}_a\cdot n \ \ \text{on} \ \ \Gamma.
  \end{split}
\end{cases}
\end{equation*}
Since $y_\phi-(u-u_h)\in W$, by the elliptic regularity theory on convex polygonal domains $z_a \in H^2(\Omega)$ and $\mathbf{r}_a \in [H^1(\Omega)]^2$. We also have $\|z_a\|_{2, \Omega}\leq C\|y_\phi-(u-u_h)\|$. { By the trace theorem, we have $\mathbf{r}_a\cdot n \in H^{\frac{1}{2}}(\Gamma_{i}) ~~\forall ~~ 1\leq i\leq k$, which implies $\phi \in H^2(\Omega)$ by the elliptic regularity theory for the Neumann problem \cite[Theorem 3.1.2.3]{grisvardbook}. Therefore $y_\phi\in H^2(\Omega)$ by the elliptic regularity theory on convex polygonal domains \cite[Theorem 3.1.2.1]{grisvardbook}.} Taking $t=\phi$ in \eqref{NOC}, then using the Cauchy-Schwarz inequality together with the fact $y_\phi=\phi$, we have the following estimate
\begin{equation*}
  \|\phi\|_{1, \Omega}=  \|y_\phi\|_{1, \Omega}\leq C\|u-u_h\|.
\end{equation*}
Subsequently we get $\|z_a\|_{2, \Omega}\leq C\|u-u_h\| $ and $\|\phi\|_{2, \Omega}\leq C\|u-u_h\|$.
\par \noindent
Let $(y_{u_h}, \mathbf{p}_{u_h})=S(0, u_h)$ and $(y^h_{\phi_I},\mathbf{p}^h_{\phi_I})=S_h(0, \phi_I)$ where $\phi_I \in Q_h$ is the Lagrange interpolation of $\phi$. Taking $t=u-u_h$ in \eqref{NOC} and using Lemma \ref{pgo}, we obtain
\begin{align}
        \|u-u_h\|^2&=(u-u_h,u-u_h)=(u-u_h,y_u-y_{u_h})+(u-u_h,y_{u_h}-u_h)\notag\\
        &= (u-u_h,y_{u-u_h})+(u-u_h,y_{u_h}-u_h)\notag\\
        &=\alpha a(\phi, u-u_h)+(y_\phi,y_{u-u_h})+(u-u_h,y_{u_h}-u_h)\notag\\
        &=\alpha a(\phi-\phi_I, u-u_h)+(y_\phi-y^h_{\phi_I},y_u-y_{u_h})+\alpha a(\phi_I, u-u_h)\notag\\
        & \quad+(y^h_{\phi_I},y_u-y^h_{u_h})+ (y^h_{\phi_I},y^h_{u_h}-y_{u_h})+(u-u_h,y_{u_h}-u_h)\notag\\
        &=\alpha a(\phi-\phi_I, u-u_h)+(y_\phi-y^h_{\phi_I},y_u-y_{u_h})+(y_f^h-y_f,y^h_{\phi_I})\notag\\
        & \quad+\langle (\mathbf{r}^h_y-\mathbf{r})\cdot n,\phi_I\rangle+ (y^h_{\phi_I},y^h_{u_h}-y_{u_h})+(u-u_h,y_{u_h}-u_h)\notag\\
        &=I+II+III+IV+V+VI.\label{imppp}
\end{align}
We now individually estimate each of the six terms in the right hand side of the \eqref{imppp}. Using Cauchy-Schwarz inequality and approximation property of $\phi_I$, we get
\begin{equation*}
    I\leq C|u-u_h|_{1,\Omega}|\phi-\phi_I|_{1,\Omega}\leq Ch\|\phi\|_{2,\Omega}|u-u_h|_{1,\Omega} \leq Ch|u-u_h|_{1,\Omega}\|u-u_h\|.
\end{equation*}
Next, we consider the term $II$. Let $(y^h_{\phi},\mathbf{p}^h_{\phi})=S_h(0,\phi).$ A use of Cauchy-Schwarz inequality, Lemma \ref{inter1}, Lemma \ref{inter2}, Lemma \ref{l2} and approximation properties of $\phi_I$ implies that 
\begin{equation*}
\begin{split}
    II&=(y_\phi-y^h_{\phi_I},y_u-y_{u_h})=(y_\phi-y^h_{\phi},y_u-y_{u_h})+(y^h_\phi-y^h_{\phi_I},y_u-y_{u_h})\\
    &\leq \left(\|y_\phi-y^h_{\phi}\|+\|y^h_\phi-y^h_{\phi_I}\|\right)\|y_u-y_{u_h}\|\\
    &\leq C\left(h\|\phi\|_{2,\Omega}+\|\phi-\phi_I\|_{1,\Omega}\right)\|u-u_h\|_{1, \Omega}\\
    &\leq Ch\|u-u_h\|_{1,\Omega}\|u-u_h\|.
\end{split}
\end{equation*}
Next, we consider the term $III$.
By Lemma \ref{rtinter}, we can write
\begin{equation*}
   III=(y_f^h-y_f,y^h_{\phi_I})=(P_h(y_f^h-y_f),y^h_{\phi_I}),
\end{equation*}
recall that $(y_f,\mathbf{p}_f)=S(f,0)$ and $(y^h_f,\mathbf{p}^h_f)=S_h(f,0)$. We find that $\left(P_h(y_f^h-y_f),\mathbf{p}^h_{f}-\mathbf{p}_f\right)$ satisfy the following equations:
\begin{align*}
     (\mathbf{p}^h_{f}-\mathbf{p}_f,\mathbf{v_h})-(P_h(y_f^h-y_f),\text{div}\,\mathbf{v_h})&=0 \ \ \forall \ \mathbf{v_h}\in \ V_h,\\
     (w_h,\text{div}(\mathbf{p}^h_{f}-\mathbf{p}_f))&=0 \ \ \forall \ w_h \in \ W_h.
\end{align*}
Therefore,  $(P_h(y_f^h-y_f), \mathbf{p}^h_{f}-\mathbf{p}_f)\in W_h \times V$ satisfy both the conditions \eqref{c111}-\eqref{c121} of Lemma \ref{lemm1}. Hence, from Lemma \ref{lemm1} and Lemma \ref{l2}, we have 
\begin{equation}\label{p00}
   \|P_h(y_f^h-y_f)\|\leq Ch\|\mathbf{p}^h_{f}-\mathbf{p}_f\|_{H(\text{div},\Omega)}\leq Ch^2\|y_f\|_{2,\Omega}\leq Ch^2\|y\|_{2,\Omega}.
\end{equation}
Now using triangle inequality, Lemma \ref{inter2}, Lemma \ref{l2} and approximation properties of $\phi_I$, we find
\begin{align}
    \|y^h_{\phi_I}\|& \leq  \|y^h_{\phi_I}-y^h_{\phi}\|+\|y^h_{\phi}-y_{\phi}\|+\|y_{\phi}\|\notag\\
    &\leq C\left(\|\phi-\phi_I\|_{1,\Omega}+h\|y_{\phi}\|_{2,\Omega}+\|\phi\|_{2,\Omega}\right)\notag\\
    & \leq C\|u-u_h\|.\label{giffy}
\end{align}
Combining \eqref{p00} and \eqref{giffy}, we have
\begin{equation*}
    III \leq Ch^2\|y\|_{2,\Omega}\|u-u_h\|.
\end{equation*}
Now we move on to estimate the term $IV$. We have,
\begin{equation*}
\begin{split}
     IV&=\langle  (\mathbf{r}^h_y-\mathbf{r})\cdot n,\phi_I\rangle
     =\langle  (\mathbf{r}^h_y-\mathbf{r})\cdot n,\phi_I-\phi\rangle+\langle  (\mathbf{r}^h_y-\mathbf{r})\cdot n,\phi\rangle.
\end{split}
\end{equation*}
Recalling that $(y_{\phi}, \mathbf{p}_{\phi})=S(0,\phi)$, we have
\begin{equation*}
    \begin{split}
    \langle (\mathbf{r}^h_y-\mathbf{r})\cdot n,\phi\rangle&=-(\mathbf{p}_{\phi}, \mathbf{r}^h_y-\mathbf{r})+(y_{\phi}, \text{div}(\mathbf{r}^h_y-\mathbf{r}))\\
    &=(\mathbf{p}^h_{\phi}-\mathbf{p}_{\phi}, \mathbf{r}^h_y-\mathbf{r})+(y_{\phi}-y^h_{\phi}, \text{div}(\mathbf{r}^h_y-\mathbf{r}))\\&\quad +(\mathbf{p}^h_{\phi},\mathbf{r}-\mathbf{r}^h_y)+(y^h_{\phi}, \text{div}(\mathbf{r}^h_y-\mathbf{r}))\\
    &=(\mathbf{p}^h_{\phi}-\mathbf{p}_{\phi}, \mathbf{r}^h_y-\mathbf{r})+(y_{\phi}-y^h_{\phi}, \text{div}(\mathbf{r}^h_y-\mathbf{r}))+(z-z^h_y,\text{div}\,\mathbf{p}^h_{\phi})\\
    &=(\mathbf{p}^h_{\phi}-\mathbf{p}_{\phi}, \mathbf{r}^h_y-\mathbf{r})+(y_{\phi}-y^h_{\phi}, \text{div}(\mathbf{r}^h_y-\mathbf{r})),
    \end{split}
\end{equation*}
where we have used that $\text{div}\,\mathbf{p}^h_{\phi}=0$ and $(w_h, \text{div}(\mathbf{r}^h_y-\mathbf{r}))=0 \ \forall \ w_h \in W_h.$ Therefore term $IV$ can be written as 
\begin{equation*}
\begin{split}
     IV&=\langle (\mathbf{r}^h_y-\mathbf{r})\cdot n,\phi_I-\phi\rangle+(\mathbf{p}^h_{\phi}-\mathbf{p}_{\phi}, \mathbf{r}^h_y-\mathbf{r})+(y_{\phi}-y^h_{\phi}, \text{div}(\mathbf{r}^h_y-\mathbf{r}))\\
    &=IV^a+IV^b+IV^c.
\end{split}
\end{equation*}
By the use of Cauchy-Schwarz inequality, approximation properties of $\phi_I$,  Lemma \ref{l1} and Lemma \ref{stdmixed}, we find
\begin{equation*}
\begin{split}
     IV^a& \leq C\|\phi_I-\phi\|_{\frac{1}{2},\Gamma}\|(\mathbf{r}^h_y-\mathbf{r})\cdot n\|_{-\frac{1}{2},\Gamma}\\
     &\leq C\|\phi_I-\phi\|_{1,\Omega}\|\mathbf{r}^h_y-\mathbf{r}\|_{H(\text{div},\Omega)}\\
     &\leq Ch^2\|u-u_h\|\|z\|_{2,\Omega}.
\end{split}
\end{equation*}
Again using Cauchy-Schwarz inequality, Lemma \ref{l2} and Lemma \ref{stdmixed}, we have
\begin{equation*}
\begin{split}
     IV^b+IV^c&\leq \left(\|\mathbf{p}^h_{\phi}-\mathbf{p}_{\phi}\|+\|y_{\phi}-y^h_{\phi}\|\right)\|\mathbf{r}^h_y-\mathbf{r}\|_{H(\text{div},\Omega)}\\
     &\leq Ch^2\|y_\phi\|_{2,\Omega}\|z\|_{2,\Omega}\\
     &\leq Ch^2\|u-u_h\|\|z\|_{2,\Omega}.
\end{split}
\end{equation*}
Combining, we get
\begin{equation*}
    IV\leq Ch^2\|u-u_h\|\|z\|_{2,\Omega}.
\end{equation*}
Let $(y_{\tilde{u}_h}, \mathbf{p}_{\tilde{u}_h})=S(0,\tilde{u}_h)$ and $(y^h_{\tilde{u}_h}, \mathbf{p}^h_{\tilde{u}_h})=S_h(0,\tilde{u}_h)$, which is the standard mixed finite element approximation of $(y_{\tilde{u}_h}, \mathbf{p}_{\tilde{u}_h})$ where $\tilde{u}_h$ is the enriched discrete control. We estimate the term $V$ as follows.
\begin{equation*}
\begin{split}
    V&=(y^h_{\phi_I},y^h_{u_h}-y_{u_h})\\
    &=(y^h_{u_h}-y^h_{\tilde{u}_h},y^h_{\phi_I})+(y^h_{\tilde{u}_h}-y_{\tilde{u}_h},y^h_{\phi_I})+(y_{\tilde{u}_h}-y_{u_h},y^h_{\phi_I})\\
    &=V^a+V^b+V^c.
\end{split}
\end{equation*}
Firstly, we claim that $V^a=0$.
Using \eqref{boundaryvanish} and definitions of  $(y^h_{\tilde{u}_h}, \mathbf{p}^h_{\tilde{u}_h})$ and $(y^h_{u_h}, \mathbf{p}^h_{u_h})$, we see
\begin{align}
        (\mathbf{p}^h_{u_h}-\mathbf{p}^h_{\tilde{u}_h}, \mathbf{v_h})-(y^h_{u_h}-y^h_{\tilde{u}_h},\text{div}\,\mathbf{v_h})&=0 \ \ \forall \ \mathbf{v_h} \in V_h,\label{mb1}\\
        (w_h, \text{div}(\mathbf{p}^h_{u_h}-\mathbf{p}^h_{\tilde{u}_h}))&=0 \ \ \forall \ w_h \in W_h.\label{mb2}
\end{align}
Putting $\mathbf{v_h}=\mathbf{p}^h_{u_h}-\mathbf{p}^h_{\tilde{u}_h}$ in \eqref{mb1} and using \eqref{mb2}, we get $\mathbf{p}^h_{u_h}=\mathbf{p}^h_{\tilde{u}_h}$. Subsequently using the surjectivity of $\text{div}:V_h\rightarrow W_h$ map, see \cite[Lemma 3.5, p. 17]{duranotes}, we have $y^h_{u_h}=y^h_{\tilde{u}_h}$ and hence $V^a=0$.  Using exactly the same arguments as in estimating the term $III$, and the fact that $\text{div}\,\mathbf{p}_{u_h}=\text{div}\,\mathbf{p}^h_{\tilde{u}_h}=\text{div}\,\mathbf{p}_{\tilde{u}_h}=0$, we have
\begin{align}
     V^c & \leq Ch\|u-u_h\|_{1, \Omega}\|u-u_h\|,\label{eq:Eq1}\\
     V^b & \leq Ch^2\|y_{\tilde{u}_h}\|_{2, \Omega}\|u-u_h\|.\label{eq:Eq2}
\end{align}
By the elliptic regularity theory on convex polygonal domains, we have 
\begin{equation}\label{eq:Eq3}
    \|y_{\tilde{u}_h}\|_{2, \Omega}\leq C\sum^{k}_{i=1}\|\tilde{u}_h\|_{\frac{3}{2}, \Gamma_i}.
\end{equation}
Let $u_c$ be a $C^1$ interpolation of Cl\'ement type \cite[Section 4.8]{brennerbook} of $u$, then using trace and inverse inequality and approximation properties of $u_c$, we find 
\begin{align}\label{eq:Eq4}
    h^2\sum^{k}_{i=1}\|\tilde{u}_h\|_{\frac{3}{2}, \Gamma_i}& \leq Ch^2\sum^{k}_{i=1}\left(\|\tilde{u}_h-u_c\|_{\frac{3}{2}, \Gamma_i}+\|u_c\|_{\frac{3}{2}, \Gamma_i}\right) \notag\\
        &\leq Ch^{\frac{3}{2}}\sum^{k}_{i=1}\|\tilde{u}_h-u_c\|_{1, \Gamma_i}+Ch^2\|u\|_{2, \Omega} \notag\\
         &\leq Ch\left(\|\tilde{u}_h-u_c\|_{1, \Omega}+h\|u\|_{2, \Omega}\right) \notag\\
         &\leq Ch\left(\|\tilde{u}_h-u_h\|_{1, \Omega}+\|u_h-u\|_{1, \Omega}+h\|u\|_{2, \Omega}\right) \notag\\
         &\leq Ch\left(\|u-u_h\|_{1, \Omega}+h\|u\|_{2, \Omega}\right).
\end{align}
In view of \eqref{eq:Eq1}-\eqref{eq:Eq4}, we have
\begin{equation*}
    V\leq Ch\|u-u_h\|\|u-u_h\|_{1, \Omega}.
\end{equation*}
Lastly, we handle the term $VI$,
\begin{equation*}
\begin{split}
     VI&=(u-u_h,y_{u_h}-u_h)\\
     &=(y_{u_h}-y_{\tilde{u}_h},u-u_h)+(y_{\tilde{u}_h}-\Tilde{u}_h,u-u_h)+(\tilde{u}_h-u_h,u-u_h)\\
     &=VI^a+VI^b+VI^c.
\end{split}
\end{equation*}
    Using the same arguments as before and Lemma \ref{rtinter}, Lemma \ref{inter1} and Lemma \ref{lemmaenriching}, we find that
    \begin{equation*}
    \begin{split}
        VI^a&=(y_{u_h}-y_{\tilde{u}_h}, u-u_h)\\
        &=(y_{u_h}-y_{\tilde{u}_h}-P_h\,(y_{u_h}-y_{\tilde{u}_h}), u-u_h)+(P_h\,(y_{u_h}-y_{\tilde{u}_h}), u-u_h)\\
        &\leq Ch\left(|y_{u_h}-y_{\tilde{u}_h}|_{1, \Omega}+\|y_{u_h}-y_{\tilde{u}_h}\|\right)\|u-u_h\|\\
        & \leq Ch\|u_h-\tilde{u}_h\|_{1, \Omega}\|u-u_h\|\\
        &\leq Ch\|u-u_h\|_{1, \Omega}\|u-u_h\|.
    \end{split}
    \end{equation*}
 A use of Lemma \ref{lemmaenriching} and Lemma \ref{crucial} yields
    \begin{equation*}
        VI^b+VI^c\leq Ch\left(\|u-u_h\|_{1,\Omega}+h\|u\|_{2,\Omega}\right)\|u-u_h\|.
    \end{equation*}
    Therefore, we get 
    \begin{equation*}
        VI\leq Ch\|u-u_h\|_{1,\Omega}\|u-u_h\|.
    \end{equation*}
   The proof is completed by combining \eqref{imppp} with the estimates on all the six terms $I, ~II, ~III, ~IV, ~V, ~VI$ and using Theorem \ref{energynorm}.
\end{proof}
\end{theorem}
\section{\emph{A Posteriori} Error Analysis}\label{secapost}
In this section, a reliable and efficient \emph{a posteriori} error estimator is derived with the help of auxiliary problems and Helmholtz decomposition \cite{duranotes}. To this end, we define the following auxiliary problem: Find $(\Tilde{y}, \Tilde{\mathbf{p}}, \Tilde{z}, \Tilde{\mathbf{r}}, \Tilde{u}) \in W\times V \times W \times V \times Q$ such that 
\begin{align}
(\tilde{\mathbf{p}},\mathbf{v})-(\tilde{y},\text{div}\,\mathbf{v})&=-\langle \mathbf{v}\cdot n,u_h\rangle \ \ \forall \  \mathbf{v}\in V, \label{e1}\\
(w,\text{div}\,\tilde{\mathbf{p}})&=(f,w)\ \ \forall\  w\in W, \label{e2} \\
(\tilde{\mathbf{r}},\mathbf{v})-(\tilde{z},\text{div}\,\mathbf{v})&=0 \ \ \forall \ \mathbf{v}\in V , \label{e3}\\
(w,\text{div}\,\tilde{\mathbf{r}})&=(y_h-y_d,w) \ \ \forall\  w\in W , \label{e4}\\
\alpha a(\tilde{u},  q)&=-\langle \mathbf{r_h}\cdot n,q\rangle \ \  \forall \  q\in Q \label{e5}.
\end{align}
We note that $(y_h, \mathbf{p_h})$ and $(z_h, \mathbf{r_h})$ are the standard mixed finite element approximation of $(\tilde{y}, \tilde{\mathbf{p}})$ and $(\tilde{z}, \tilde{\mathbf{r}})$, respectively. And $u_h$ is the standard conforming finite element approximation of $\tilde{u}$.
By substracting the corresponding equations of the system \eqref{e1}-\eqref{e5} from \eqref{b10}-\eqref{b14}, we get the error equations:
\begin{align}
(\mathbf{p}-\tilde{\mathbf{p}},\mathbf{v})-(y-\Tilde{y},\text{div}\,\mathbf{v})&=-\langle \mathbf{v}\cdot n,u- u_h\rangle \ \ \forall \  \mathbf{v}\in V, \label{e6}\\
(w,\text{div}(\mathbf{p}-\tilde{\mathbf{p}}))&=0\ \ \forall\  w\in W, \label{e7} \\
(\mathbf{r}-\tilde{\mathbf{r}},\mathbf{v})-(z-\Tilde{z},\text{div}\,\mathbf{v})&=0 \ \ \forall \ \mathbf{v}\in V , \label{e8}\\
(w,\text{div}(\mathbf{r}-\tilde{\mathbf{r}}))&=(y-y_h,w) \ \ \forall\  w\in W , \label{e9}\\
\alpha a(u-\tilde{u}, q)&=-\langle (\mathbf{r}-\mathbf{r_h})\cdot n,q\rangle \ \  \forall \  q\in Q \label{e10}.
\end{align}
Below, we prove a lemma which is useful in deriving \emph{a posteriori} error estimates.
\begin{lemma}\label{reliability}
The following estimate holds:
\begin{multline*}
   \left( |u-u_h|_{1, \Omega}+\|z-z_h\|+\|y-y_h\|+\|\mathbf{p-p_h}\|_{H(\text{div}, \Omega)}+\|\mathbf{r-r_h}\|_{H(\text{div},\Omega)}\right)\\ \leq C\left(|\tilde{u}-u_h|_{1, \Omega}+\|\tilde{y}-y_h\|+{ \|\tilde{z}-z_h\|}+\|\mathbf{\tilde{r}-r_h}\|_{H(\text{div}, \Omega)}+\|\mathbf{\tilde{p}-p_h}\|_{H(\text{div},\Omega)}\right).
\end{multline*}
\begin{proof}
Take $q=u-\tilde{u} \in Q$ in \eqref{e10} to get
\begin{equation*}
\begin{split}
    \alpha|u-\tilde{u}|^2_{1, \Omega}&= -\langle (\mathbf{r-r_h})\cdot n,u-\tilde{u}\rangle\\
    & =-\langle  (\mathbf{r-\tilde{r}})\cdot n,u-\tilde{u}\rangle-\langle  (\mathbf{\tilde{r}-r_h})\cdot n,u-\tilde{u}\rangle
    \\ &= -\langle  (\mathbf{r-\tilde{r}})\cdot n,u-u_h\rangle -\langle (\mathbf{r-\tilde{r}})\cdot n,u_h-\tilde{u}\rangle-\langle (\mathbf{\tilde{r}-r_h})\cdot n,u-\tilde{u}\rangle.
    \end{split}
    \end{equation*}
Using \eqref{e6} and \eqref{e8} for $\mathbf{v}=\mathbf{r-\tilde{r}}$ and $\mathbf{v}=\mathbf{p-\tilde{p}}$, respectively,  taking into account \eqref{e9} and \eqref{e7}, we find
\begin{equation*}
    \begin{split}
        \alpha|u-\tilde{u}|^2_{1, \Omega}&=(\mathbf{p-\tilde{p}}, \mathbf{r-\tilde{r}})-(y-\tilde{y},\text{div}(\mathbf{r-\tilde{r}}))-\langle (\mathbf{r-\tilde{r}})\cdot n,u_h-\tilde{u}\rangle-\langle (\mathbf{\tilde{r}-r_h})\cdot n,u-\tilde{u}\rangle\\&=(z-\tilde{z}, \text{div}(\mathbf{p-\tilde{p}}))-(y-y_h, y-\tilde{y})-\langle (\mathbf{r-\tilde{r}})\cdot n,u_h-\tilde{u}\rangle-\langle (\mathbf{\tilde{r}-r_h})\cdot n,u-\tilde{u}\rangle\\&
        =-(y-y_h, y-\tilde{y})-\langle (\mathbf{r-\tilde{r}})\cdot n,u_h-\tilde{u}\rangle-\langle (\mathbf{\tilde{r}-r_h})\cdot n,u-\tilde{u}\rangle
        \\ &=-(y-\tilde{y}, y-\tilde{y})-(\tilde{y}-y_h, y-\tilde{y})-\langle (\mathbf{r-\tilde{r}})\cdot n,u_h-\tilde{u}\rangle-\langle  (\mathbf{\tilde{r}-r_h})\cdot n,u-\tilde{u}\rangle.
    \end{split}
\end{equation*}
Adding $\|y-\tilde{y}\|^2$ to both the sides of the last equation to find
\begin{equation*}
    \begin{split}
        \alpha|u-\tilde{u}|^2_{1, \Omega}+\|y-\tilde{y}\|^2&=(y_h-\tilde{y}, y-\tilde{y})-\langle (\mathbf{r-r_h})\cdot n,u_h-\tilde{u}\rangle\\&\quad -\langle (\mathbf{r_h-\tilde{r}})\cdot n,u_h-\tilde{u}\rangle-\langle (\mathbf{\tilde{r}-r_h})\cdot n,u-\tilde{u}\rangle.
    \end{split}
\end{equation*}
Upon taking $q=u_h-\tilde{u} \in Q$ in \eqref{e10} and applying Lemma \ref{l1}, we get 
\begin{align}
        \alpha|u-\tilde{u}|^2_{1, \Omega}+\|y-\tilde{y}\|^2 &=(\tilde{y}-y, \tilde{y}-y_h)+\alpha a(u-\tilde{u}, u_h-\tilde{u})\notag\\&\quad+\langle (\mathbf{r_h-\tilde{r}})\cdot n,\tilde{u}-u_h\rangle +\langle (\mathbf{\tilde{r}-r_h})\cdot n,\tilde{u}-u\rangle
        \notag\\&=(\tilde{y}-y, \tilde{y}-y_h)+\alpha a(u-\tilde{u}, u_h-\tilde{u})\notag\\&\quad+\langle (\mathbf{r_h-\tilde{r}})\cdot n,u-u_h\rangle\notag\\
        &=(\tilde{y}-y, \tilde{y}-y_h)+\alpha a(u-\tilde{u}, u_h-\tilde{u})\notag\\&\quad+\int_{\Omega}\text{div}\,(\mathbf{r_h-\tilde{r}})(u-u_h)\,dx+\int_{\Omega}(\mathbf{r_h-\tilde{r}})\cdot \nabla(u-u_h)\,dx\notag\\
 &=(\tilde{y}-y, \tilde{y}-y_h)+\alpha a(u-\tilde{u}, u_h-\tilde{u})+\int_{\Omega}(\mathbf{r_h-\tilde{r}})\cdot \nabla(u-u_h)\,dx\notag\\&\quad+\int_{\Omega}\text{div}\,(\mathbf{r_h-\tilde{r}})((u-u_h)-P_h(u-u_h))\,dx,\label{e0}
\end{align}
where in the last equality we have used \eqref{e4} and \eqref{c7}.
Now, taking $w=\text{div}(\mathbf{r-\tilde{r}}) \in W$ in \eqref{e9}, we find
\begin{equation*}
    \|\text{div}(\mathbf{r-\tilde{r}})\|^2=(y-y_h, \text{div}(\mathbf{r-\tilde{r}})),
\end{equation*}
which yields
\begin{equation}\label{e12}
    \|\text{div}(\mathbf{r-\tilde{r}})\| 
    \leq \|y-y_h\|.
\end{equation}
Now, take $v=\mathbf{r-\tilde{r}} \in V$ in \eqref{e8} to get
\begin{equation}\label{e13}
    \|\mathbf{r-\tilde{r}}\|^2=(z-\tilde{z}, \text{div}(\mathbf{r-\tilde{r}})).
\end{equation}
Since $z-\tilde{z} \in W$, there exists $\mathbf{v}\in H^1(\Omega)^2$ such that $\text{div}\,\mathbf{v}=z-\tilde{z}$ and 
\begin{equation}\label{e14}
    \|\mathbf{v}\|_{1, \Omega} \leq C\|z-\tilde{z}\|.
\end{equation}
Further, using \eqref{e8}, \eqref{e14} and Cauchy-Schwarz inequality, we have
\begin{equation*}
    \begin{split}
        \|z-\tilde{z}\|^2&=(\mathbf{r-\tilde{r}}, \mathbf{v})\leq \|\mathbf{r-\tilde{r}}\|\|\mathbf{v}\| \leq C\|\mathbf{r-\tilde{r}}\|\|z-\tilde{z}\|,
    \end{split}
\end{equation*}
which yields
\begin{equation}\label{e15}
    \|z-\tilde{z}\|\leq C\|\mathbf{r-\tilde{r}}\|.
\end{equation}
 In view of \eqref{e15} and \eqref{e13}, we get 
\begin{equation*}
    \|\mathbf{r-\tilde{r}}\|\leq C\|\text{div}(\mathbf{r-\tilde{r}})\|.
\end{equation*}
Therefore,
\begin{equation}\label{e17}
    \|\mathbf{r-\tilde{r}}\|_{H(\text{div},\Omega)} \leq C \|\text{div}(\mathbf{r-\tilde{r}})\|.
\end{equation}
Combining  \eqref{e15}, \eqref{e17} and \eqref{e12}, we get 
\begin{align}
    \|z-\tilde{z}\| &\leq C\|\mathbf{r-\tilde{r}}\|_{H(\text{div},\Omega)}
    \leq C\|y-y_h\| \label{e18} \\
    &\leq C\left( \|y-\tilde{y}\|+\|\tilde{y}-y_h\|\right).\label{e22}
\end{align}
Using Cauchy-Schwarz inequality, Young's inequality, Lemma \ref{rtinter} in \eqref{e0}, we find 
\begin{equation}\label{e21}
    \begin{split}
        \alpha|u-\tilde{u}|^2_{1, \Omega}+\|y-\tilde{y}\|^2&\leq \delta_1\|y-\tilde{y}\|^2 +\frac{C_1}{\delta_1}\|\tilde{y}-y_h\|^2 + \alpha \delta_2|u-\tilde{u}|^2_{1, \Omega}\\
        & \quad+\frac{C_2 \alpha}{\delta_2}|u_h-\tilde{u}|^2_{1, \Omega}+ \delta_3|u-u_h|^2_{1,\Omega}+\frac{C_3}{\delta_3}\|\mathbf{r_h-\tilde{r}}\|^2_{H(\text{div}, \Omega)}.
    \end{split}
\end{equation}
From \eqref{e21}, \eqref{e22} and choosing $\delta_1, \delta_2, \delta_3$ small enough to get
\begin{align}\label{sid1}
|u-\tilde{u}|^2_{1, \Omega}+\|z-\tilde{z}\|^2+\|y-\tilde{y}\|^2& \leq C\left(|\tilde{u}-u_h|^2_{1, \Omega}+\|\tilde{y}-y_h\|^2+\|\mathbf{\tilde{r}-r_h}\|^2_{H(\text{div}, \Omega)}\right).
\end{align}
A use of triangle inequality and \eqref{sid1} yields
\begin{align}\label{e122}
|u-u_h|_{1, \Omega}+\|z-z_h\|+\|y-y_h\|&\leq C\left(|\tilde{u}-u_h|_{1, \Omega}+\|\tilde{y}-y_h\|+{ \|\tilde{z}-z_h\|}\right.\notag\\
&\quad\left.+\|\mathbf{\tilde{r}-r_h}\|_{H(\text{div}, \Omega)}\right).
\end{align}
Now we estimate $\|\mathbf{r-r_h}\|_{H(\text{div},\Omega)}$ and $\|\mathbf{p-p_h}\|_{H(\text{div}, \Omega)}$ as follows.
A use of triangle inequality, \eqref{e18} and \eqref{e122} gives
\begin{align}
    \|\mathbf{r-r_h}\|_{H(\text{div},\Omega)}&\leq  \|\mathbf{r-\tilde{r}}\|_{H(\text{div},\Omega)}+\|\mathbf{\tilde{r}-r_h}\|_{H(\text{div},\Omega)}\notag\\
    &\leq C\left(\|y-y_h\|+\|\mathbf{\tilde{r}-r_h}\|_{H(\text{div},\Omega)}\right)\notag\\
    &\leq C\left(|\tilde{u}-u_h|_{1, \Omega}+\|\tilde{y}-y_h\|+{ \|\tilde{z}-z_h\|}+\|\mathbf{\tilde{r}-r_h}\|_{H(\text{div}, \Omega)}\right).\label{e123}
    \end{align}
 Put $\mathbf{v}=\mathbf{p-\tilde{p}}$ in \eqref{e6} and using \eqref{e7}, Lemma \ref{l1} and Cauchy-Schwarz inequality, we find 
 \begin{align}
      \|\mathbf{p-\tilde{p}}\|^2_{H(\text{div},\Omega)}= \|\mathbf{p-\tilde{p}}\|^2&=\langle (\mathbf{p-\tilde{p}})\cdot n,u_h-u\rangle\notag\\
      &=\int_{\Omega}\text{div}\,(\mathbf{p-\tilde{p}})(u_h-u)\,dx+\int_{\Omega}(\mathbf{p-\tilde{p}})\cdot \nabla(u_h-u)\,dx\notag\\
     & =\int_{\Omega}(\mathbf{p-\tilde{p}})\cdot \nabla(u_h-u)\,dx\notag\\
    & \leq |u-u_h|_{1, \Omega}\|\mathbf{p-\tilde{p}}\|.\label{sid2}
\end{align}
By triangle inequality, \eqref{sid2} and \eqref{e122}, we find
    \begin{align}
        \|\mathbf{p-p_h}\|_{H(\text{div}, \Omega)}&\leq \|\mathbf{p-\tilde{p}}\|_{H(\text{div},\Omega)}+ \|\mathbf{\tilde{p}-p_h}\|_{H(\text{div}, \Omega)}\notag\\
        &\leq |u-u_h|_{1, \Omega}+\|\mathbf{\tilde{p}-p_h}\|_{H(\text{div}, \Omega)}\notag\\
        &\leq C\left(|\tilde{u}-u_h|_{1, \Omega}+\|\tilde{y}-y_h\|+{ \|\tilde{z}-z_h\|}+\|\mathbf{\tilde{r}-r_h}\|_{H(\text{div}, \Omega)}+\|\mathbf{\tilde{p}-p_h}\|_{H(\text{div}, \Omega)}\right).\label{sid33}
    \end{align}
The proof is completed by combining \eqref{e122}, \eqref{e123} and \eqref{sid33}.
 \end{proof}
  \end{lemma}
  \par \noindent
 Next, we introduce some notations which are required for further analysis. Let $t$ be the unit tangent vector on $e\in \cE_h$ oriented clockwise. For an interior side $e \in \cE_h^i$ shared by two neighboring triangles $T_1$ and $T_2$ with corresponding unit tangent vectors $t_1$ and $t_2$ on $e$, we define the tangential jump of $\mathbf{v}\in H(\text{div}, \Omega)$ across the interior edge $e$ as follows:
  $$\sjump{\mathbf{v} \cdot t}= \mathbf{v}|_{T_1} \cdot t_1- \mathbf{v}|_{T_2}\cdot t_2.$$
  For a boundary edge $e\in \cE_h^b$, there is a triangle $T \in \cT_h$ such that $e=\partial T\cap \Gamma.$  The tangential jump across the boundary edge $e$ is defined as follows:
  $$\sjump{\mathbf{v} \cdot t} =\mathbf{v}|_{T} \cdot t.$$
  For $\mathbf{v}=(v_1, v_2) \in H^1(\Omega)^2$, $w\in H^1(\Omega)$, we define  $\text{rot}\,\mathbf{v}$  and $\textbf{curl}\,w$ as follows:
  $$ \text{rot}\,\mathbf{v}=\frac{\partial v_2}{\partial x_1}-\frac{\partial v_1}{\partial x_2}$$
  $$\textbf{curl}\,w=\left(\frac{\partial w}{\partial x_2}, -\frac{\partial w}{\partial x_1}\right).$$
  The following lemma is crucial to establish the reliability of the error estimator. The error is decomposed by using a generalized Helmholtz decomposition under the assumption that the domain is simply connected. 
  \begin{lemma}\label{helmhotlz}
For $\mathbf{v}\in W \times W$, there exist $w_1\in H^1_0(\Omega)$ and $w_2 \in Q$ such that
 \begin{equation*}
     \mathbf{v}=\nabla w_1+\mathbf{curl}\,w_2.
 \end{equation*}
 Moreover, the following estimate holds:
 \begin{equation*}
     \|\nabla w_1\|+\|\nabla w_2\|\lesssim \|\mathbf{v}\|.
 \end{equation*}
 \begin{proof}
 For the proof, we refer the readers to \cite[Lemma 4.1, p. 27]{duranotes}.
 \end{proof}
  \end{lemma}
We now define the estimator terms.

\par \noindent
The volume residuals are defined as follows
\begin{align*}
    \eta_{1,T}&=\|y_h-y_d-\text{div}\,\mathbf{r_h}\|_{L^2(T)},   \quad \quad \quad \eta_{1}=\left(\sum_{T\in \cT_h}\eta_{1,T}^2\right)^{\frac{1}{2}},\\
    \eta_{2,T}&=h_T\|\text{rot}\,\mathbf{r_h}\|_{L^2(T)}, \quad   \ \ \quad  \quad \quad \quad \ \ \eta_{2}=\left(\sum_{T\in \cT_h}\eta_{2,T}^2\right)^{\frac{1}{2}},\\
     \eta_{3,T}&=h_T\|\text{rot}\,\mathbf{p_h}\|_{L^2(T)}, \quad \quad \ \  \quad \quad \quad \ \ \eta_{3}=\left(\sum_{T\in \cT_h}\eta_{3,T}^2\right)^{\frac{1}{2}},\\
    \eta_{4, T}&=\|f-\text{div}\,\mathbf{p_h}\|_{L^2(T)},  \quad  \ \  \quad \quad \quad \quad \eta_{4}=\left(\sum_{T\in \cT_h}\eta_{4,T}^2\right)^{\frac{1}{2}},\\
    \eta_{5, T}&=h_T\|\mathbf{p_h}+\nabla y_h\|_{L^2(T)},    \ \ \quad \quad \quad  \quad \eta_{5}=\left(\sum_{T\in \cT_h}\eta_{5,T}^2\right)^{\frac{1}{2}},\\
    \eta_{6, T}&=h_T\|\mathbf{r_h}+\nabla z_h\|_{L^2(T)},   \ \ \ \quad \quad \quad  \quad \eta_{6}=\left(\sum_{T\in \cT_h}\eta_{6,T}^2\right)^{\frac{1}{2}}.
\end{align*}
The edge residuals are defined by
\begin{align*}
    \eta_{1,e}&=|e|^{\frac{1}{2}}\|\sjump{\mathbf{r_h}\cdot t}\|_{L^2(e)}, \quad \quad \quad \quad \quad \quad \eta_{7}=\left(\sum_{e\in \cE_h^i}\eta_{1,e}^2\right)^{\frac{1}{2}},\\
    \eta_{2,e}&=|e|^{\frac{1}{2}}\|\sjump{\mathbf{p_h}\cdot t}\|_{L^2(e)}, \quad \quad \quad \quad \quad \quad \eta_{8}=\left(\sum_{e\in \cE_h^i}\eta_{2,e}^2\right)^{\frac{1}{2}},\\
     \eta_{3,e}&=|e|^{\frac{1}{2}}\|\sjump{y_h}\|_{L^2(e)}, \quad \quad \quad \quad \quad \quad \quad \eta_{9}=\left(\sum_{e\in \cE_h^i}\eta_{3,e}^2\right)^{\frac{1}{2}},\\
     \eta_{4,e}&=|e|^{\frac{1}{2}}\|\sjump{z_h}\|_{L^2(e)}, \quad \quad \quad \quad \quad \quad \quad \eta_{10}=\left(\sum_{e\in \cE_h^i}\eta_{4,e}^2\right)^{\frac{1}{2}},\\
      \eta_{5,e}&=\alpha|e|^{\frac{1}{2}}\left\|\jump{\frac{\partial u_h}{\partial n}}\right\|_{L^2(e)}, \ \quad \quad \quad \quad \eta_{11}=\left(\sum_{e\in \cE_h^i}\eta_{5,e}^2\right)^{\frac{1}{2}},
      \end{align*}
      and the boundary residuals are given by
      \begin{align*}
      \eta_{6,e}&=|e|^{\frac{1}{2}}\left\|\alpha \dfrac{\partial u_h}{\partial n}+\mathbf{r_h}\cdot n\right\|_{L^2(e)}, \ \ \quad \quad \eta_{12}=\left(\sum_{e\in \cE_h^b}\eta_{6,e}^2\right)^{\frac{1}{2}},\\
      \eta_{7,e}&=|e|^{\frac{1}{2}}\left\| \dfrac{\partial u_h}{\partial t}+\mathbf{p_h}\cdot t\right\|_{L^2(e)}, \ \ \quad  \quad \quad \eta_{13}=\left(\sum_{e\in \cE_h^b}\eta_{7,e}^2\right)^{\frac{1}{2}},\\
      \eta_{8,e}&=|e|^{\frac{1}{2}}\left\| u_h-y_h\right\|_{L^2(e)}, \ \ \quad \quad \quad \quad \quad \eta_{14}=\left(\sum_{e\in \cE_h^b}\eta_{8,e}^2\right)^{\frac{1}{2}},\\
      \eta_{9,e}&=|e|^{\frac{1}{2}}\left\| z_h\right\|_{L^2(e)}, \ \ \quad \quad \quad \quad \quad \quad  \quad\eta_{15}=\left(\sum_{e\in \cE_h^b}\eta_{9,e}^2\right)^{\frac{1}{2}},\\
      \eta_{10,e}&=|e|^{\frac{1}{2}}\|\mathbf{r_h}\cdot t\|_{L^2(e)}, \quad \quad \quad \quad \quad \quad \quad \eta_{16}=\left(\sum_{e\in \cE_h^b}\eta_{10,e}^2\right)^{\frac{1}{2}}.
\end{align*}
 \par \noindent
 The total error estimator is given by
\begin{equation*}
    \eta_h=\left(\eta_{1}^2+\eta_{2}^2+\eta_{3}^2+\eta_{4}^2+\eta_{5}^2+\eta_{6}^2+\eta_{7}^2+\eta_{8}^2+\eta_{9}^2+\eta_{10}^2+\eta_{11}^2+\eta_{12}^2+\eta_{13}^2 +\eta_{14}^2+\eta_{15}^2+\eta_{16}^2\right)^{\frac{1}{2}}.
\end{equation*}
In the next theorem, we prove the reliability of the error estimator $\eta_h$. 
\begin{theorem}\textbf{(Reliability of the error estimator)}\label{reliestimate}
It holds that
\begin{equation*}
    |u-u_h|_{1, \Omega}+\|z-z_h\|+\|y-y_h\|+ \|\mathbf{p-p_h}\|_{H(\text{div}, \Omega)}+ \|\mathbf{r-r_h}\|_{H(\text{div}, \Omega)} \lesssim \eta_h.
\end{equation*}
\end{theorem}
\begin{proof} In view of Lemma \ref{reliability}, it suffices to estimate $|\tilde{u}-u_h|_{1, \Omega}+\|\tilde{y}-y_h\|+{ \|\tilde{z}-z_h\|}+\|\mathbf{\tilde{r}-r_h}\|_{H(\text{div}, \Omega)}+\|\mathbf{\tilde{p}-p_h}\|_{H(\text{div},\Omega)}$.
Using \eqref{e1}-\eqref{e5} and \eqref{c4}-\eqref{c8}, we find
\begin{align}
    (\mathbf{\tilde{p}-p_h},\mathbf{v_h})-(\tilde{y}-y_h,\text{div}\,\mathbf{v_h})&=0 \ \ \forall \  \mathbf{v_h}\in V_h, \label{e23}\\
(w_h,\text{div}(\mathbf{\tilde{p}-p_h}))&=0\ \ \forall\  w_h\in W_h, \label{e24} \\
(\mathbf{\tilde{r}-r_h},\mathbf{v_h})-(\tilde{z}-z_h,\text{div}\,\mathbf{v_h})&=0 \ \ \forall \ \mathbf{v_h}\in V_h , \label{e25}\\
(w_h,\text{div}(\mathbf{\tilde{r}-r_h}))&=0 \ \ \forall\  w_h\in W_h , \label{e26}\\
\alpha a(\tilde{u}-u_h, q_h)&=0 \ \  \forall \  q_h\in Q_h \label{e27}.
\end{align}
Using \eqref{e27} and similar arguments as in \cite{verfurthpaper1,46}, we find 
\begin{equation}\label{e90}
    |\tilde{u}-u_h|^2_{1, \Omega} \lesssim {\eta_{11}^2+\eta_{12}^2}.
\end{equation}
 Using the system of equations \eqref{e23}-\eqref{e26} and similar arguments as in \cite{carsten,duranotes}, we obtain
 \begin{equation}\label{e91}
     \|\mathbf{\tilde{r}-r_h}\|^2_{H(\text{div},\Omega)} 
 {\lesssim \eta_{1}^2+\eta_{2}^2+\eta_{7}^2+\eta_{16}^2},
 \end{equation}
 \begin{equation}\label{e92}
     \|\tilde{y}-y_h\|^2\lesssim { \eta_{5}^2+\eta_{9}^2+\eta_{14}^2+\|\mathbf{\tilde{p}-p_h}\|^2},
 \end{equation}
 and
 \begin{equation}\label{lme92}
  { \|\tilde{z}-z_h\|^2\lesssim  \eta_{6}^2+\eta_{10}^2+\eta_{15}^2+\|\mathbf{\tilde{r}-r_h}\|^2}.
 \end{equation}
It remains to estimate $\|\mathbf{\tilde{p}-p_h}\|_{H(\text{div}, \O)}$. By Lemma \ref{helmhotlz}, there exist $\gamma \in H^1_0(\Omega)$ and $\beta \in H^1(\Omega)$ such that 
\begin{equation}\label{e34}
    \mathbf{\tilde{p}-p_h}=\nabla \gamma + \textbf{curl}\,\beta,
\end{equation}
{and}
\begin{equation}\label{e35}
    \|\nabla \gamma\|+\|\nabla \beta\|\lesssim \|\mathbf{\tilde{p}-p_h}\|.
\end{equation}
Using the error decomposition \eqref{e34}, we have
\begin{equation}\label{e37}
       \|\mathbf{\tilde{p}-p_h}\|^2 =\int_\Omega (\mathbf{\tilde{p}-p_h})\cdot \nabla \gamma\,dx +\int_{\Omega}(\mathbf{\tilde{p}-p_h})\cdot \mathbf{curl}\,\beta\,dx.
\end{equation}
Integrating by parts and using \eqref{e2}, we have
\begin{equation*}
    \begin{split}
        \int_\Omega (\mathbf{\tilde{p}-p_h})\cdot \nabla \gamma\,dx &=-\int_{\Omega}\text{div}(\mathbf{\tilde{p}-p_h})\,\gamma\,dx
        =\int_{\Omega}(\text{div}\,\mathbf{p_h}-f)\,\gamma\,dx\\
        &=\int_{\Omega}(\text{div}\,\mathbf{p_h}-f)(\gamma-P_h\gamma)\,dx,
    \end{split}
\end{equation*}
where in obtaining the last equation, we have used $\int_{\Omega}(\text{div}\,\mathbf{p_h}-f)\,P_h\gamma \,dx= 0$ which follows from \eqref{c5}.
{ Using Cauchy-Schwarz inequality and Lemma \ref{rtinter}, we find that
\begin{equation}\label{e38}
     \int_\Omega (\mathbf{\tilde{p}-p_h})\cdot \nabla \gamma\,dx\lesssim \|\text{div}\,\mathbf{p_h}-f\|\|\nabla \gamma\|.
\end{equation}}
\par \noindent
Now, consider the second term of the right hand side of \eqref{e37}. Let $\beta_h \in \mathbb{P}_1^c(\cT_h)$ be an approximation of $\beta$ as defined in Lemma \ref{clement:approx} such that $\mathbf{curl}\,\beta_h \in V_h$. Using \eqref{e23}, integration by parts and $\mathbf{\tilde{p}}\cdot t=-\frac{\partial u_h}{\partial t}$ on $e \in \cE_h^b$, we get
\begin{align} \label{ef2}
    \int_{\Omega}(\mathbf{\tilde{p}-p_h})\cdot \mathbf{curl}\,\beta\,dx &=\int_{\Omega}(\mathbf{\tilde{p}-p_h})\cdot \mathbf{curl} (\beta-\beta_h)\,dx+\int_{\Omega}(\mathbf{\tilde{p}-p_h})\cdot \mathbf{curl}\,\beta_h\,dx \notag\\
    &=\sum_{T \in \cT_h}\int_{T}(\mathbf{\tilde{p}-p_h})\cdot \mathbf{curl}(\beta- \beta_h)\,dx \notag\\
    & =\sum_{T\in \cT_h}\Bigg\{\int_T(\text{rot}(\mathbf{\tilde{p}-p_h})(\beta- \beta_h)\,dx-\int_{\partial T}(\mathbf{\tilde{p}-p_h})\cdot t\,(\beta- \beta_h)\,dx\Bigg\} \notag\\
     &=-\sum_{T\in \cT_h}\int_T(\text{rot}\,\mathbf{p_h})(\beta- \beta_h)\,dx-\sum_{e\in \cE_h}\int_{e}\sjump{(\mathbf{\tilde{p}-p_h})\cdot t}(\beta- \beta_h)\,ds \notag \\
 &=-\sum_{T\in \cT_h}\int_T(\text{rot}\,\mathbf{p_h})(\beta-\beta_h)\,dx+\sum_{e\in \cE_h^i}\int_{e}\sjump{\mathbf{p_h}\cdot t}(\beta- \beta_h)\,ds \notag\\
       & \quad +\sum_{e\in \cE_h^b}\int_{e}\left(\frac{\partial u_h}{\partial t}+\mathbf{p_h}\cdot t \right)\left(\beta-\beta_h\right)\,ds,   
\end{align}
%
where in the second step we have used that
\begin{equation}\label{ef1}
    \int_{\Omega}(\mathbf{\tilde{p}-p_h})\cdot \mathbf{curl}\,\beta_h\,dx=\int_{\Omega}(\tilde{y}-y_h)\,\text{div}(\mathbf{curl}\,\beta_h)\,dx=0.
\end{equation}
%
%
{Using Cauchy-Schwarz inequality and Lemma \ref{clement:approx}, we get 
\begin{align}
    \int_{\Omega}(\mathbf{\tilde{p}-p_h})\cdot \mathbf{curl}\,\beta\,dx &\lesssim \left(\sum_{T\in \cT_h}h_T^2\|\text{rot}\,\mathbf{p_h}\|^2_{L^2(T)} +\sum_{e\in \cE_h^i}|e|\|\sjump{\mathbf{p_h}\cdot t}\|^2_{L^2(e)}\right.\notag\\& \qquad +\left. \sum_{e\in \cE_h^b}|e|\left\|\frac{\partial u_h}{\partial t}+\mathbf{p_h}\cdot t\right\|^2_{L^2(e)}\right)^{\frac{1}{2}}\|\nabla \beta\|_{L^2(\Omega)}\label{e42}.
\end{align}}
Finally, using the equations \eqref{e35}, \eqref{e37}, \eqref{e38} and \eqref{e42}, we get
\begin{equation}\label{e46}
   \|\mathbf{\tilde{p}-p_h}\|^2_{H(\text{div},\Omega)} \lesssim { \eta_{3}^2+\eta_{4}^2+\eta_{8}^2+\eta_{13}^2}.
\end{equation}
The proof is concluded by combining \eqref{e90}, \eqref{e91}, \eqref{e92}, \eqref{lme92} and \eqref{e46}.
\end{proof}
\par \noindent
Next, we proceed to discuss the efficiency estimates of the error estimator $\eta_h$. For $T\in \cT_h$, we denote by $D_T$, union of the elements of triangulation that share an edge with $T$. Also for an edge $e\in \cE_h$, $T_e$ denotes the union of elements having $e$ as an edge. Note that if $e\in \cE_h^b$, then $T_e=T$ such that $e=\partial T\cap \Gamma$.
We now briefly sketch the proof for local efficiency estimates for the error estimators.

\begin{theorem}\textbf{(Efficiency for the error estimator)}\label{efficiencyestimate}
The following estimates hold:
\begin{align}
    \eta_{1,T}&\lesssim \|y-y_h\|_{L^2(T)}+ \|\mathbf{r-r_h}\|_{H(\text{div},T)}, \label{eq:Eff1}\\
    \eta_{2,T}&\lesssim \|\mathbf{r-r_h}\|_{L^2(D_T)},\label{eq:Eff2}\\
    \eta_{3,T}&\lesssim \|\mathbf{p-p_h}\|_{L^2(T)},\label{eq:Eff3}\\
    \eta_{4,T}&\lesssim \|\mathbf{p-p_h}\|_{H(\text{div},T)},\label{eq:Eff4}\\
    \eta_{5,T}&\lesssim  \|y-y_h\|_{L^2(T)}+h_T\|\mathbf{p-p_h}\|_{L^2(T)},\label{eq:Eff5}\\
   \eta_{6,T}&\lesssim  \|z-z_h\|_{L^2(T)}+h_T\|\mathbf{r-r_h}\|_{L^2(T)},\label{eq:Eff14}\\
    \eta_{1,e}&\lesssim \|\mathbf{r-r_h}\|_{L^2(D_T)},\label{eq:Eff6}\\
    \eta_{2,e}&\lesssim \|\mathbf{p-p_h}\|_{L^2(T_e)},\label{eq:Eff7}\\
  \eta_{3,e}&\lesssim \|y-y_h\|_{L^2(T_e)}+|e|\|\mathbf{p-p_h}\|_{L^2(T_e)},\label{eq:Eff8}\\
  \eta_{4,e}&\lesssim \|z-z_h\|_{L^2(T_e)}+|e|\|\mathbf{r-r_h}\|_{L^2(T_e)},\label{eq:Eff15}\\
 \eta_{5,e}&\lesssim |u-u_h|_{H^1(T_e)},\label{eq:Eff9}\\
 \eta_{6,e}&\lesssim |u-u_h|_{H^1(T_e)}+\|\mathbf{r-r_h}\|_{H(\text{div},T_e)},\label{eq:Eff10}\\
 \eta_{7,e}&\lesssim |u-u_h|_{H^1(T_e)}+\|\mathbf{p-p_h}\|_{L^2(T_e)},\label{eq:Eff11}\\
 \eta_{8,e}&\lesssim  \|y-y_h\|_{L^2(T_e)}+|e|\,\|\mathbf{p-p_h}\|_{L^2(T_e)}+\|u-u_h\|_{L^2(T_e)}+|e||u-u_h|_{H^1(T_e)},\label{eq:Eff12}\\
 \eta_{9,e}&\lesssim  \|z-z_h\|_{L^2(T_e)}+|e|\,\|\mathbf{r-r_h}\|_{L^2(T_e)},\label{eq:Eff16}\\
 \eta_{10,e}&\lesssim \|\mathbf{r-r_h}\|_{L^2(D_T)}.\label{eq:Eff13}
\end{align}
\end{theorem}
\begin{proof}
$\bullet$ (Lower bound of $ \eta_{1,T}$:)
In view of \eqref{b13}, \eqref{e4} and Cauchy-Schwarz inequality, we have
\begin{equation*}
    \begin{split}
        \eta_{1,T}&=\|\text{div}(\mathbf{\tilde{r}-r_h})\|_{L^2(T)}
         \leq \|\text{div}(\mathbf{\tilde{r}-r})\|_{L^2(T)}+\|\text{div}(\mathbf{r-r_h})\|_{L^2(T)}\\
        & \lesssim \|y-y_h\|_{L^2(T)}+\|\mathbf{r-r_h}\|_{H(\text{div},T)}.
    \end{split}
\end{equation*}
$\bullet$ (Lower bounds of $ \eta_{2,T}, \eta_{1,e}$ and $\eta_{10,e}$:)
Set $\eta^2_{T,e}=\eta^2_{2,T}+\eta^2_{1,e}+\eta^2_{10,e}$. Let $\omega \in H^1_0(D_T)$ be an interior bubble function defined on $D_T$. Since $\omega$ vanishes at the boundary of $D_T$, it can be extended by zero outside $D_T$, say $\tilde{\omega} \in H^1_0(\Omega)$ be its extension by zero to $\Omega$. Then using the similar arguments as in \cite[Theorem 4.2, p. 29]{duranotes} by taking into account \eqref{b12} and \eqref{e3}, we find
\begin{equation*}
\begin{split}
    \eta^2_{T,e}&=\int_{D_T}(\mathbf{\tilde{r}-r_h})\cdot \mathbf{curl}\,\omega\,dx
    =\int_{D_T}(\mathbf{\tilde{r}-r})\cdot \mathbf{curl}\,\omega\,dx+\int_{D_T}(\mathbf{r-r_h})\cdot \mathbf{curl}\,\omega\,dx\\
    &=\int_{\Omega}(\mathbf{\tilde{r}-r})\cdot \mathbf{curl}\,\tilde{\omega}\,dx+\int_{D_T}(\mathbf{r-r_h})\cdot \mathbf{curl}\,\omega\,dx\\
    &=\int_{D_T}(\mathbf{r-r_h})\cdot \mathbf{curl}\,\omega\,dx.
\end{split}
\end{equation*}
Using $\|\nabla \omega\|_{L^2(D_T)}\lesssim \eta_{T,e}$, \cite[Lemma 4.2, p. 28]{duranotes} and Cauchy-Schwarz inequality, we get
\begin{equation*}
    \eta_{T,e}\lesssim \|\mathbf{r-r_h}\|_{L^2(D_T)}.
\end{equation*}
$\bullet$ (Lower bound of $ \eta_{3,T}$:)
Let $T\in \cT_h$ be arbitrary and $b_T \in \mathbb{P}_3(T)$ be an interior bubble function which takes the unit value at the barycenter of $T$ and vanishes on $\partial T$. Define $\theta=b_T\,(\text{rot}\,\mathbf{p_h})$ on $T$ and $\tilde{\theta} \in H^1_0(\Omega)$ be the extension of $\theta$ by zero to $\Omega$.
Using norm equivalence on finite dimensional vector spaces, we have
\begin{equation*}
\begin{split}
    \|\text{rot}\,\mathbf{p_h}\|^2_{L^2(T)}&\lesssim \int_{T}b_T\,|\text{rot}\,\mathbf{p_h}|^2\,dx =\int_{T}\theta\,\text{rot}\,\mathbf{p_h}\,dx=\sum_{T\in \cT_h}\int_{T}\tilde{\theta}\,\text{rot}\,\mathbf{p_h}\,dx.
\end{split}
\end{equation*}
Using equations \eqref{ef1}, \eqref{ef2}, \eqref{b10} and \eqref{e1} we get
\begin{align*}
     \sum_{T\in \cT_h}\int_{T}\tilde{\theta}\,\text{rot}\,\mathbf{p_h}\,dx 
       &=- \int_{\Omega}(\mathbf{\tilde{p}-p_h})\cdot \mathbf{curl}\,\tilde{\theta}\,dx\\
       &=- \int_{\Omega}(\mathbf{\tilde{p}-p})\cdot \mathbf{curl}\,\tilde{\theta}\,dx- \int_{\Omega}(\mathbf{p-p_h})\cdot \mathbf{curl}\,\tilde{\theta}\,dx\\
       &=-\int_{\Omega}(y-\tilde{y})\,\text{div}(\mathbf{curl}\,\tilde{\theta})\,dx+\int_{\Gamma}(u-u_h)\,(\mathbf{curl}\,\tilde{\theta}\cdot n)\,ds-\int_{T}(\mathbf{p-p_h})\cdot \mathbf{curl}\,\theta\,dx\\
       &=\int_{T}(\mathbf{p_h-p})\cdot \mathbf{curl}\,\theta\,dx.
\end{align*}
Therefore, using Cauchy-Schwarz inequality and $\|\nabla \tilde{\theta}\|_{L^2(T)} \leq h^{-1}_{T}\|\text{rot}\,\mathbf{p_h}\|_{L^2(T)}$ \cite{verfurthpaper1,verfurthpaper2}, we get 
\begin{equation}\label{etav3}
    \eta_{3,T}\lesssim \|\mathbf{p-p_h}\|_{L^2(T)}.
\end{equation}
$\bullet$ (Lower bounds of $ \eta_{4,T}$, $ \eta_{5,T}$ and $ \eta_{6,T}$:)
The estimate \eqref{eq:Eff4} is immediate with the observation that
\begin{equation*}
    \|f-\text{div}\,\mathbf{p_h}\|_{L^2(T)}=\|\text{div}(\mathbf{p-p_h})\|_{L^2(T)}\leq \|\mathbf{p-p_h}\|_{H(\text{div},T)},
\end{equation*}
and a use of the arguments as in \cite[Lemma 4.5, p. 33 ]{duranotes} yields
\begin{equation}\label{ev5}
    \eta_{5,T} \lesssim \|y-y_h\|_{L^2(T)}+h_T\|\mathbf{p-p_h}\|_{L^2(T)},
\end{equation}
\begin{equation}\label{nev6}
    \eta_{6,T} \lesssim \|z-z_h\|_{L^2(T)}+h_T\|\mathbf{r-r_h}\|_{L^2(T)}.
\end{equation}
$\bullet$ (Lower bound of $\eta_{2,e}$:)
We begin by observing that for any $\beta \in Q$, using \eqref{ef1} and \eqref{ef2}, we have 
\begin{equation}\label{div11}
    \begin{split}
     \sum_{e\in \cE_h^i}\int_{e}\sjump{\mathbf{p_h}\cdot t}\,\beta\,ds&=\sum_{T\in \cT_h}\int_T(\text{rot}\,\mathbf{p_h})\,\beta\,dx+ \int_{\Omega}(\mathbf{\tilde{p}-p_h})\cdot \mathbf{curl}\,\beta\,dx\\
       & \quad -\sum_{e\in \cE_h^b}\int_{e}\left(\frac{\partial u_h}{\partial t}+\mathbf{p_h}\cdot t \right)\,\beta\,ds.
    \end{split}
\end{equation}
Let $e\in \cE_h^i$ be an interior edge. Let $b_e\in P_4(T_e)$ be an edge bubble function which takes the unit value at the midpoint of $e$ and vanishes on $\partial T_e\setminus e$ . Define $\theta=b_e\,\sjump{\mathbf{p_h}\cdot t}$ on $T_e$, where $t$ is the unit tangent vector on $e$. Let $\tilde{\theta}\in H^1_0(\Omega)$ be an extension of this function by zero to $\Omega$. Further, $\theta$ satisfies the following estimates \cite{verfurthpaper1,verfurthpaper2} 
\begin{align}
        \|\nabla \theta\|_{L^2(T_e)}\lesssim |e|^{-\frac{1}{2}}\|\sjump{\mathbf{p_h}\cdot t}\|_{L^2(e)},\label{bubble1}\\
        \|\theta\|_{L^2(T_e)}\lesssim |e|^{\frac{1}{2}}\|\sjump{\mathbf{p_h}\cdot t}\|_{L^2(e)}.\label{bubble2}
\end{align}
 Using the norm equivalence on finite dimensional vector spaces and \eqref{div11}, we have
\begin{align}
     \|\sjump{\mathbf{p_h}\cdot t}\|^2_{L^2(e)}&\lesssim \int_{e}\theta\,\sjump{\mathbf{p_h}\cdot t}\,ds=\sum_{e\in \cE_h^i}\int_{e}\sjump{\mathbf{p_h}\cdot t}\,\tilde{\theta}\,ds\notag\\
     &=\sum_{T\in \cT_h}\int_T(\text{rot}\,\mathbf{p_h})\,\tilde{\theta}\,dx+ \int_{\Omega}(\mathbf{\tilde{p}-p_h})\cdot \mathbf{curl}\,\tilde{\theta}\,dx-\sum_{e\in \cE_h^b}\int_{e}\left(\frac{\partial u_h}{\partial t}+\mathbf{p_h}\cdot t \right)\,\tilde{\theta}\,ds\notag\\
     &=\int_{\Omega}(\mathbf{\tilde{p}-p})\cdot \mathbf{curl}\,\tilde{\theta}\,dx+\int_{\Omega}(\mathbf{p-p_h})\cdot \mathbf{curl}\,\tilde{\theta}\,dx+\int_{T_e}(\text{rot}\,\mathbf{p_h})\,\theta\,dx\notag\\
     &=\int_{\Omega}(y-\tilde{y})\,\text{div}(\mathbf{curl}\,\tilde{\theta})\,dx-\int_{\Gamma}(u-u_h)\,(\mathbf{curl}\,\tilde{\theta}\cdot n)\,ds\notag\\
     & \quad +\int_{T_e}(\mathbf{p-p_h})\cdot \mathbf{curl}\,\theta\,dx+\int_{T_e}(\text{rot}\,\mathbf{p_h})\,\theta\,dx\notag\\
     &=\int_{T_e}(\mathbf{p-p_h})\cdot \mathbf{curl}\,\theta\,dx+\int_{T_e}(\text{rot}\,\mathbf{p_h})\,\theta\,dx.\label{ef3}
\end{align}

Finally, using Cauchy-Schwarz inequality, \eqref{etav3} and estimates \eqref{bubble1}, \eqref{bubble2} in \eqref{ef3}, we get the estimate \eqref{eq:Eff7}.

$\bullet$ (Lower bound of $\eta_{7,e}$:)
Let $e\in \cE_h^b$ and $T \in \cT_h$ be the triangle having an edge $e$. Let $b_e \in P_2(T)$ be an edge bubble function such that it  vanishes on $\partial T \setminus e$ and takes the unit value at the midpoint of $e$. We define $\theta = b_e\,\left(\frac{\partial u_h}{\partial t}+\mathbf{p_h}\cdot t\right)$ on $T$ where $t$ is the unit tangent vector on the edge $e$. We extend this function by zero to $\Omega$ and call it to be $\tilde{\theta}$. The function $\theta$ satisfies the following estimates \cite{verfurthpaper1,verfurthpaper2} 
\begin{align}
\|\nabla \theta\|_{L^2(T_e)}&\lesssim |e|^{-\frac{1}{2}}\Big\|\frac{\partial u_h}{\partial t}+\mathbf{p_h}\cdot t\Big\|_{L^2(e)},\label{vt1}\\
\|\theta\|_{L^2(T_e)}&\lesssim |e|^{\frac{1}{2}}\Big\|\frac{\partial u_h}{\partial t}+\mathbf{p_h}\cdot t\Big\|_{L^2(e)}.\label{vt2}
\end{align}
Now using the norm equivalence on finite dimensional spaces and equations \eqref{ef1}, \eqref{ef2}, we have
\begin{equation*}
    \begin{split}
       \Big \|\frac{\partial u_h}{\partial t}+\mathbf{p_h}\cdot t\Big\|^2_{L^2(e)}&\lesssim \int_{e}\theta\,\left(\frac{\partial u_h}{\partial t}+\mathbf{p_h}\cdot t\right)\,ds=\sum_{e\in\cE_h^b}\int_{e}\tilde{\theta}\,\left(\frac{\partial u_h}{\partial t}+\mathbf{p_h}\cdot t\right)\,ds\\
     &=\int_{\Omega}(\mathbf{\tilde{p}-p_h})\cdot \mathbf{curl}\,\tilde{\theta}\,dx +\sum_{T\in\cT_h}\int_{T}(\text{rot}\,\mathbf{p_h})\,\tilde{\theta}\,dx- \sum_{e\in \cE_h^i}\int_{e}\sjump{\mathbf{p_h}\cdot t}\,\tilde{\theta}\,ds\\
     &=\int_{\Omega}(\mathbf{\tilde{p}-p})\cdot \mathbf{curl}\,\tilde{\theta}\,dx+\int_{T}(\mathbf{p-p_h})\cdot \mathbf{curl}\,\theta\,dx+\int_{T}(\text{rot}\,\mathbf{p_h})\,\theta\,dx\\
     &=\int_{\Omega}(y-\tilde{y})\,\text{div}(\mathbf{curl}\,\tilde{\theta})\,dx-\int_{\Gamma}(u-u_h)\,(\mathbf{curl}\,\tilde{\theta}\cdot n)\,ds\\
     & \quad +\int_{T}(\mathbf{p-p_h})\cdot \mathbf{curl}\,\theta\,dx+\int_{T}(\text{rot}\,\mathbf{p_h})\,\theta\,dx\\
     &=\int_{T}(\mathbf{p-p_h})\cdot \mathbf{curl}\,\theta\,dx+\int_{T}(\text{rot}\,\mathbf{p_h})\,\theta\,dx-\int_{e}(u-u_h)\,(\mathbf{curl}\,\theta \cdot n)\,ds.
    \end{split}
\end{equation*}
Using Cauchy-Schwarz inequality, Lemma \ref{l1}, \eqref{etav3}, \eqref{vt1} and \eqref{vt2}, we get
\begin{equation}
    \eta_{7,e}\lesssim |u-u_h|_{H^1(T_e)}+\|\mathbf{p-p_h}\|_{L^2(T_e)}.
\end{equation}
$\bullet$ (Lower bound of $\eta_{3,e}$:)
Let $e \in \cE_h^i$. A use of trace inequality \cite[Section 1.6]{brennerbook} yields
\begin{align*}
   |e|^{\frac{1}{2}} \|\sjump{y_h}\|_{L^2(e)}&= |e|^{\frac{1}{2}} \|\sjump{y_h-y}\|_{L^2(e)}\\
   &\lesssim \|y_h-y\|_{L^2(T_e)} +|e|\|\nabla(y_h-y)\|_{L^2(T_e)} \\
    &=\|y_h-y\|_{L^2(T_e)} +|e|\|\nabla y_h+\mathbf{p}\|_{L^2(T_e)}
    \\
    &\lesssim \|y-y_h\|_{L^2(T_e)} +|e|\|\nabla y_h+\mathbf{p_h}\|_{L^2(T_e)}+|e|\|\mathbf{p-p_h}\|_{L^2(T_e)}.
\end{align*}
We finally get the desired estimate by using \eqref{eq:Eff5}.

{
$\bullet$ (Lower bound of $\eta_{4,e}$:)
Let $e \in \cE_h^i$. A use of trace inequality \cite[Section 1.6]{brennerbook} yields
\begin{align*}
   |e|^{\frac{1}{2}} \|\sjump{z_h}\|_{L^2(e)}&= |e|^{\frac{1}{2}} \|\sjump{z_h-z}\|_{L^2(e)}\\
   &\lesssim \|z_h-z\|_{L^2(T_e)} +|e|\|\nabla(z_h-z)\|_{L^2(T_e)} \\
    &=\|z_h-z\|_{L^2(T_e)} +|e|\|\nabla z_h+\mathbf{r}\|_{L^2(T_e)}
    \\
    &\lesssim \|z-z_h\|_{L^2(T_e)} +|e|\|\nabla z_h+\mathbf{r_h}\|_{L^2(T_e)}+|e|\|\mathbf{r-r_h}\|_{L^2(T_e)}.
\end{align*}
We finally get the desired estimate by using \eqref{eq:Eff14}.}

$\bullet$ (Lower bound of $\eta_{8,e}$:)
For any $\mathbf{v}\in V$, using the equations \eqref{e1}, \eqref{e2} and Green's identity, we find that 
\begin{equation}\label{estim7}
    \begin{split}
        \sum_{e\in \cE^b_h}\int_{e}(u_h-y_h) (\mathbf{v}\cdot n)\,ds &=\sum_{T\in \cT_h}\int_{T}(\mathbf{p_h-\tilde{p}})\cdot \mathbf{v}\,dx+\sum_{T\in \cT_h}\int_{T}(\tilde{y}-y_h)\text{div}\,\mathbf{v}\,dx\\
        & \quad -\sum_{T\in \cT_h}\int_{T}(\nabla y_h+\mathbf{p_h})\,\mathbf{v}\,dx+\sum_{e\in \cE^i_h}\int_{e}\sjump{y_h}(\mathbf{v}\cdot n)\,ds.
    \end{split}
\end{equation}
Let $e \in \cE^b_h$ be arbitrary and $T$ be the triangle such that $e=\partial T \cap \Gamma.$ Let $\mathbf{b_e} \in V$ be a lowest order Raviart-Thomas basis function corresponding to an edge $e$ such that $\mathbf{b_e}\cdot n_e=1$ on $e$ and $\mathbf{b_e}\cdot n=0$ for all other edges except $e$ (see \cite{threematlab}). Define $\mathbf{v_e}=\mathbf{b_e}\,(u_h-y_h)$ on $\Omega$. Note that $\mathbf{v_e}$ and $\text{div}\,\mathbf{v_e}$ { vanish} outside $T$.  Using the standard scaling arguments, we have the following estimates  
\begin{align}
\|\text{div}\,\mathbf{v_e}\|_{L^2(T_e)}&\lesssim |e|^{-\frac{1}{2}}\|u_h-y_h\|_{L^2(e)},\label{bubble31}\\
        \|\mathbf{v_e}\|_{L^2(T_e)}&\lesssim |e|^{\frac{1}{2}}\|u_h-y_h\|_{L^2(e)}.\label{bubble32}
\end{align}
 Then using the norm equivalence on finite dimensional spaces and \eqref{estim7}, we have that 
\begin{equation*}
    \begin{split}
    \|u_h-y_h\|^2_{L^2(e)}&= \int_{e}\mathbf{v_e}\cdot n_e\,(u_h-y_h)\,ds \\
    &=\sum_{T\in \cT_h}\int_{T}(\mathbf{p_h-\tilde{p}})\cdot \mathbf{v_e}\,dx+\sum_{T\in \cT_h}\int_{T}(\tilde{y}-y_h)\text{div}\,\mathbf{v_e}\,dx\\
        & \quad -\sum_{T\in \cT_h}\int_{T}(\nabla y_h+\mathbf{p_h})\cdot \mathbf{v_e}\,dx+\sum_{e\in \cE^i_h}\int_{e}\sjump{y_h}(\mathbf{v_e}\cdot n)\,ds\\
        &=\int_{T}(\mathbf{p_h-p})\cdot \mathbf{v_e}\,dx+\int_{\Omega}(\mathbf{p-\tilde{p}})\cdot \mathbf{v_e}\,dx\\
        & \quad +\int_{T}(y-y_h)\,\text{div}\,\mathbf{v_e}\,dx+\int_{\Omega}(\tilde{y}-y)\,\text{div}\,\mathbf{v_e}\,dx\\
        & \quad -\int_{T}(\nabla y_h+\mathbf{p_h})\cdot \mathbf{v_e}\,dx\\
        & =\int_{T}(\mathbf{p_h-p})\cdot \mathbf{v_e}\,dx+\int_{T}(y-y_h)\,\text{div}\,\mathbf{v_e}\,dx -\int_{T}(\nabla y_h+\mathbf{p_h})\cdot \mathbf{v_e}\,dx\\
        & \quad -\int_{T}(u-u_h)\,\text{div}\,\mathbf{v_e}\,dx-\int_{T}\nabla(u-u_h)\cdot\mathbf{v_e}\,dx,
    \end{split}
\end{equation*}
where in the last equality, we have used equation \eqref{e6} and Lemma \ref{l1}. By the use of Cauchy-Schwarz inequality, \eqref{bubble31}, \eqref{bubble32} and \eqref{ev5}, we have the desired estimate \eqref{eq:Eff12}.

{
$\bullet$ (Lower bound of $\eta_{9,e}$:)
The desired estimate \eqref{eq:Eff16} can be obtained by applying the same set of arguments used for obtaining \eqref{eq:Eff12}.}

$\bullet$ (Lower bound of $\eta_{5,e}$:)
Let $q\in Q$ be arbitrary. Then using \eqref{b14}, we find that
\begin{align}
     \alpha a(\tilde{u}, q)-\alpha a(u_h, q) &=-\int_{\Gamma}q\,(\mathbf{r_h}\cdot n)\,ds-\alpha \int_{\Omega}\nabla u_h\cdot \nabla q\,dx\notag\\
     &=-\sum_{T\in \cT_h}\Bigg\{\int_{\Gamma \cap \partial T } q\,(\mathbf{r_h}\cdot n)\,ds+\alpha \int_{T}\nabla u_h\cdot \nabla q\,dx  \Bigg\}.\label{blx}
\end{align}
Apply integration by parts on the second term on the right hand side of \eqref{blx} to get
\begin{align}
        \alpha a(\tilde{u}-u_h, q)&=\sum_{T\in \cT_h}\Bigg\{-\int_{\Gamma \cap \partial T } q\,(\mathbf{r_h}\cdot n)\, ds+\alpha \int_{T}\Delta u_h \,q\,dx-\alpha \int_{\partial T}\frac{\partial u_h}{\partial n}\,q\,ds\Bigg\}\notag\\
        &=-\alpha\sum_{e\in\cE_h^i}\int_{e}\jump{\frac{\partial u_h}{\partial n}}q\,ds-\sum_{e\in \cE_h^b}\int_{e}\left(\alpha\frac{\partial u_h}{\partial n}+\mathbf{r_h}\cdot n\right)\,q\,ds.\label{er}
\end{align}
Let $e\in \cE_h^i$ and $b_e \in P_4(T_e)$ be an edge bubble function which vanishes on $\partial T_e \setminus e$ and takes the value one at the midpoint of $e$. Then by equivalence of two norms on finite dimensional spaces, we get 
\begin{equation*}
   \alpha\,\bigg\|\jump{\frac{\partial u_h}{\partial n}}\bigg
   \|^2_{L^2(e)}\lesssim  \alpha\int_{e}b_e \jump{\frac{\partial u_h}{\partial n}}^2\,ds.
\end{equation*}
Define $\theta=b_e\,\jump{\frac{\partial u_h}{\partial n}}$ on $T_e$ where $n$ is the outward unit normal vector on $e$ and $\tilde{\theta}\in H^1_0(\Omega)$ be an extension of $\theta$ by zero to $\Omega$. Then using the equation \eqref{er} and \eqref{e10}, we have
\begin{equation*}
    \begin{split}
        \alpha \bigg\|\jump{\frac{\partial u_h}{\partial n}}\bigg
   \|^2_{L^2(e)}&\leq  \alpha\int_{e}\theta\,\jump{\frac{\partial u_h}{\partial n}}\,ds
   =\alpha\sum_{e\in\cE_h^i}\int_{e}\tilde{\theta}\,\jump{\frac{\partial u_h}{\partial n}}\,ds\\
   &=-\alpha\sum_{T\in\cT_h}\int_{T}\nabla(\tilde{u}-u_h)\cdot\nabla \tilde{\theta}\,dx\\
   &=-\alpha\int_{\Omega}\nabla(\tilde{u}-u)\cdot\nabla \tilde{\theta}\,dx-\alpha \int_{T_e}\nabla(u-u_h)\cdot\nabla\,\theta\,dx\\
   &=\langle (\mathbf{r+r_h})\cdot n,\tilde{\theta}\rangle-\alpha \int_{T_e}\nabla(u-u_h)\cdot\nabla \theta\,dx\\
   &=-\alpha \int_{T_e}\nabla(u-u_h)\cdot\nabla \theta\,dx.
    \end{split}
\end{equation*}
Using the estimate $\|\nabla \theta\|_{L^2(T_e)}\leq |e|^{-\frac{1}{2}}\big\|\jump{\frac{\partial u_h}{\partial n}}\big\|_{L^2(e)}$ \cite{verfurthpaper1,verfurthpaper2} and Cauchy-Schwarz inequality, we obtain the estimate \eqref{eq:Eff9}.

$\bullet$ (Lower bound of $\eta_{6,e}$:)
Let $e\in \cE_h^b$  and $b_e \in P_2(T)$ be an edge bubble function which vanishes on $\partial T \setminus e$, where $T \in \cT_h$ is the triangle such that $e \subset \partial T$. Define $\theta = b_e\,\left(\frac{\partial u_h}{\partial n}+\mathbf{r_h}\cdot n\right)$ on $T$ where $n$ is the outward unit normal vector on $e$. Let $\tilde{\theta}$ be an extension of $\theta$ by zero to $\Omega$. Then using the norm equivalence on finite dimensional spaces, equations \eqref{e10} and \eqref{er}, we have 
\begin{equation*}
    \begin{split}
       \Big \|\alpha \frac{\partial u_h}{\partial n}+\mathbf{r_h}\cdot n\Big\|^2_{L^2(e)}
       &\lesssim \sum_{e\in\cE_h^b}\int_{e}\tilde{\theta}\,\left(\alpha \frac{\partial u_h}{\partial n}+\mathbf{r_h}\cdot n\right)\,ds\\
       &=-\alpha\int_{\Omega}\nabla(\tilde{u}-u_h)\cdot \nabla \tilde{\theta}\,dx\\
       &=-\alpha\int_{\Omega}\nabla(\tilde{u}-u)\cdot \nabla \tilde{\theta}\,dx-\alpha\int_{T}\nabla(u-u_h)\cdot \nabla \theta\,dx\\
       &=\int_{e}\theta\,(\mathbf{r_h-r})\cdot n\,ds-\alpha\int_{T}\nabla(u-u_h)\cdot \nabla \theta\,dx.
    \end{split}
\end{equation*}
Now, using the estimate $\|\nabla\theta\|_{L^2(T)}\lesssim |e|^{-\frac{1}{2}} \Big \|\alpha \frac{\partial u_h}{\partial n}+\mathbf{r_h}\cdot n\Big\|_{L^2(e)}$ \cite{verfurthpaper1,verfurthpaper2}, Cauchy-Schwarz inequality and Lemma \ref{l1} in the last equation, we get the desired estimate.
\end{proof}
\section{Numerical Realization}\label{secnumer}
In this section, we present the numerical results of two test examples to validate the theoretical findings. 
The results of \emph{a priori} error estimates derived in Section \ref{secapriori}  are validated by the first experiment. On the other hand the second experiment validates the reliability and the efficiency of the error estimator discussed in Section \ref{secapost}. In the first example, mesh is refined uniformly while in the second example instead of uniform mesh refinement, the following adaptive strategy is used for mesh refinement :
\begin{center}
   \textbf{SOLVE} $\rightarrow$ \textbf{ESTIMATE} $\rightarrow$ \textbf{MARK} $\rightarrow$ \textbf{REFINE}. 
\end{center}
First, we solve the discrete system \eqref{c4}-\eqref{c8} and then compute the error estimator $\eta_h$ in the step \textbf{ESTIMATE}.
We use the D\"orfler's marking strategy \cite{26} with parameter $\theta = 0.4$ for marking the elements for refinement. Using the newest vertex bisection algorithm, the marked elements are refined to obtain a new adaptive mesh.
 All the computations have been performed using the software package MATLAB and the discrete system is solved using a preconditioned projection algorithm. In this direction, we modify the cost functional with known solution. The cost functional $J$ is modified to $J_m$, which is defined by
\begin{eqnarray*}
\ J_m(w, p) = \frac{1}{2}\|w-y_d\|^2 +\frac{\alpha}{2}|p-u_d|^2_{1,\Omega},  \
\end{eqnarray*}
where $(w,p) \in W \times Q$ and $u_d \in Q$ is a given function. Then the minimization problem reads: Find $(y,u) \in W \times Q$ such that 
\begin{equation*}
J_m(y, u) = \min_{(w,p)\in W \times Q}J_m(w,p),
\end{equation*}
subject to the condition that $(w,\textbf{k})=S(f,p)$. It can be easily checked that the continuous optimality system takes the form:
\begin{align*}
(\mathbf{p},\mathbf{v})-(y,\text{div}\,\mathbf{v})&=-\langle \mathbf{v}\cdot n,u\rangle \ \ \forall \  \mathbf{v}\in V, \\
(w,\text{div}\,\mathbf{p})&=(f,w)\ \ \forall\  w\in W,  \\
(\mathbf{r},\mathbf{v})-(z,\text{div}\,\mathbf{v})&=0 \ \ \forall \ \mathbf{v}\in V , \\
(w,\text{div}\,\mathbf{r})&=(y-y_d,w) \ \ \forall\  w\in W , \\
\alpha a(u,q)&=-\langle \mathbf{r}\cdot n,q\rangle + \alpha a(u_d, q)\ \  \forall \  q\in Q .
\end{align*}
Accordingly, the discrete optimality system is modified as well.
\begin{align*}
(\mathbf{p_h},\mathbf{v_h})-(y_h,\text{div}\,\mathbf{v_h})&=-\langle \mathbf{v_h}\cdot n,u_h\rangle \ \ \forall \  \mathbf{v_h}\in V_h,\\
(w_h,\text{div}\,\mathbf{p_h})&=(f,w_h)\ \ \forall\  w_h\in W_h, \\
(\mathbf{r_h},\mathbf{v_h})-(z_h,\text{div}\,\mathbf{v_h})&=0 \ \ \forall \ \mathbf{v_h}\in V_h , \\
(w_h,\text{div}\,\mathbf{r_h})&=(y_h-y_d,w_h) \ \ \forall\  w_h\in W_h , \\
\alpha a(u_h, q_h)&=-\langle \mathbf{r_h}\cdot n,q_h\rangle+ \alpha a(u_d,q_h)\ \  \forall \  q_h\in Q_h.
\end{align*}
In this case, \emph{a posteriori} error estimator $\eta_{5,e}$  is modified as follows:
\begin{equation*}
    { \eta_{5,e}=|e|^{\frac{1}{2}}\left\|\alpha \dfrac{\partial (u_h-u_d)}{\partial n}+\mathbf{r_h}\cdot n\right\|_{L^2(e)}}.
\end{equation*}
\begin{example}\label{aprioriexmp}
In this experiment, we consider $\Omega = (0,1)\times (0,1)$ and the parameter $\alpha =1$ together with the following data: the state $y(x_1, x_2)=e^{(x_1+x_2)}$, $\mathbf{p}(x_1,x_2)=-\nabla y$, the adjoint state  $z (x_1,x_2)=x_1^2(1-x_1^2)^2x_2^2(1-x_2^2)^2$, $\mathbf{r}(x_1,x_2)=-\nabla z$ and the control $u(x_1,x_2)=e^{(x_1+x_2)}$. Then obtain $f= -\Delta y$ and $y_d=y+ \Delta z$. Note that the choice of $z $ leads to $u_d=u$. \\
\noindent
Table \ref{tab table1} shows the errors and order of convergence of the numerical approximations of the state, adjoint state and control in $L^2$-norm. The errors and the orders of convergence of the numerical approximations of the gradient of state, gradient of adjoint state and control in $H(\text{div},\Omega)$ norm and $H^1$-norm are computed and shown in Table \ref{tab table2}. It is evident from these tables that the error converges with optimal rate in the respective norms.The plots of the exact and discrete controls are shown in the Figure \ref{fig5}.

 \end{example}
        \begin{table}[!ht]
\begin{center}        
        \caption{\label{tab table1}Errors and orders of convergence in $L^2$-norm.}
 \begin{tabular}{|c| c| c| c| c| c| c|} 
 \hline
 $h$ & $\|y-y_h\|$ & Order & $\|z - z_h\|$ &Order & $\|u-u_h\|$ & Order\\ [0.5ex] 
 \hline
 $2^{-2}$ & 3.712 x $10^{-1}$ & - &  3.183 x $10^{-3}$    & - & 5.064 x $10^{-2}$ & -    \\ 
 \hline
 $2^{-3}$ & 1.877 x $10^{-1}$ & 1.187  &1.520 x $10^{-3}$ &1.286& 1.593 x $10^{-2}$&  2.014  \\
 \hline
 $2^{-4}$ & 9.405 x $10^{-2}$ &1.094 & 7.446 x $10^{-4}$ &1.130&  4.298 x $10^{-3}$ &2.074  \\
 \hline
 $2^{-5}$ &4.705 x $10^{-2}$ &1.046 &3.700 x $10^{-4}$ & 1.056&   1.102 x $10^{-3}$ &2.056  \\
 \hline
 $2^{-6}$ & 2.353 x $10^{-2}$ & 1.023 & 1.847 x $10^{-4}$&1.025 &  2.778 x $10^{-4}$ & 2.034 \\ 
 \hline
$2^{-7}$ & 1.176 x $10^{-2}$& 1.011&9.230 x $10^{-5}$& 1.012 & 6.961 x $10^{-5}$& 2.019 
 \\ [1ex] 
 \hline
\end{tabular}
\end{center}
\end{table}
 \begin{table}[!ht]
 \begin{center} 
 \caption{\label{tab table2}Errors and orders of convergence in $H(\text{div},\Omega)$ and $H^1$-norm.}
  \begin{tabular}{|c| c| c| c| c| c| c|} 
 \hline
 $h$ & $\|\mathbf{p - p_h}\|_{H(\text{div},\Omega)}$ & Order & $\|\mathbf{r - r_h}\|_{H(\text{div},\Omega)}$ &Order & $\|u-u_h\|_{1,\Omega}$ & Order\\ [0.5ex] 
 \hline
  $2^{-2}$ & 9.818 x $10^{-1}$ & - &   2.040 x $10^{-1}$    & - &  5.852 x $10^{-1}$ & -    \\ 
 \hline
   $2^{-3}$& 4.980 x $10^{-1}$ & 1.182 &1.130 x $10^{-1}$ &1.029&3.137 x $10^{-1}$&1.086    \\
 \hline
   $2^{-4}$&2.493 x $10^{-1}$ & 1.095 &  5.740 x $10^{-2}$  &1.071 & 1.609 x $10^{-1}$ & 1.057 \\
 \hline
  $2^{-5}$ &1.246 x $10^{-1}$ & 1.048 & 2.880 x $10^{-2}$  &1.042&8.118 x $10^{-2}$ & 1.034 \\
 \hline
$2^{-6}$& 6.227 x $10^{-2}$ & 1.023 & 1.441 x  $10^{-2}$ &1.022&  4.070 x $10^{-2}$ & 1.019 \\ 
 \hline
  $2^{-7}$ & 3.113 x $10^{-2}$& 1.012 & 7.207 x $10^{-3}$ &1.011& 2.037 x $10^{-2}$&  1.010 
\\[1ex] 
 \hline
\end{tabular}
\end{center} 
\end{table}
\begin{figure}[H]
    \centering
    \includegraphics[height=6cm,width=12cm]{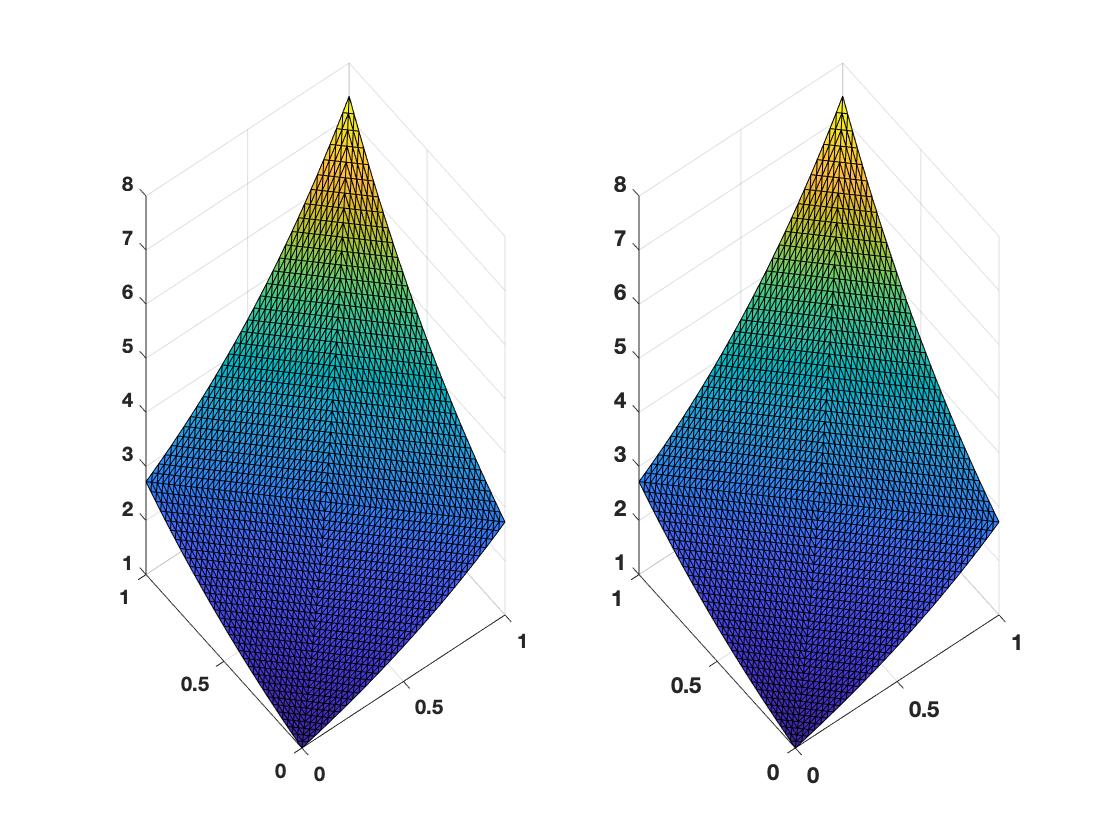}
    \caption{The computed control(left) and exact control(right) for Example \ref{aprioriexmp}.}
    \label{fig5}
\end{figure}
 \begin{example}\label{exmpapost}
In this example, we consider the $L-$shaped domain as shown in the Figure \ref{fig3} and the regularization parameter $\alpha=1$. The data is chosen as follows: the state $y(r, \theta)=r^{(2/3)}\sin{(\frac{2 \theta}{3})}$, the adjoint state  $z (x_1,x_2)=x_1^2(1-x_1^2)^2x_2^2(1-x_2^2)^2$ and control variable  $u(r, \theta)=r^{(2/3)}\sin{(\frac{2 \theta}{3})}$. We then compute $\mathbf{p}=-\nabla y$, $\mathbf{r}=-\nabla z$, $f=-\Delta y$, $y_d=y+ \Delta z$ and $u_d=u$ as in the previous example.  
 Figure \ref{fig1} { illustrates} the behavior of the error estimator $\eta_h$ and the total error $|u-u_h|_{1, \Omega}+\|z-z_h\|+\|y-y_h\|+ \|\mathbf{p-p_h}\|_{H(\text{div}, \Omega)}+ \|\mathbf{r-r_h}\|_{H(\text{div}, \Omega)}$ with respect to increasing number of total degrees of freedom (total number of unknowns for optimal state $y$, $\mathbf{p}$, optimal control $u$ and optimal adjoint state $z$, $\mathbf{r}$). This figure confirms the reliability of the error estimator and we observe that the error and the estimator both converge with the linear rate which is optimal.{ The convergence history of the estimator contributions $\eta_{1}$, $\eta_{i}$ for $5\leq i \leq 16$ is recorded in Figures \ref{fig111} and \ref{fig112}, note that $\eta_{2}=\eta_{3}=\eta_{4}=0$ for this experiment.} Figure \ref{fig2} shows the efficiency of the error estimator using the efficiency indices(estimator/total error). The adaptive mesh refinement is depicted through Figure \ref{fig3}, as expected we observe more refinement near the corner where the optimal variables have singular behavior. The plots of the exact and discrete { controls} on an adaptive mesh is shown in the Figure \ref{fig4}. {We have compared the error and estimator for different values of marking parameter $\theta$ in Figure \ref{fig11} and \ref{fig12}.}
\end{example}


\begin{figure}[t]
\centering
\begin{subfigure}{.4\textwidth}
  \centering
  \includegraphics[width=6.7cm,height=6cm]{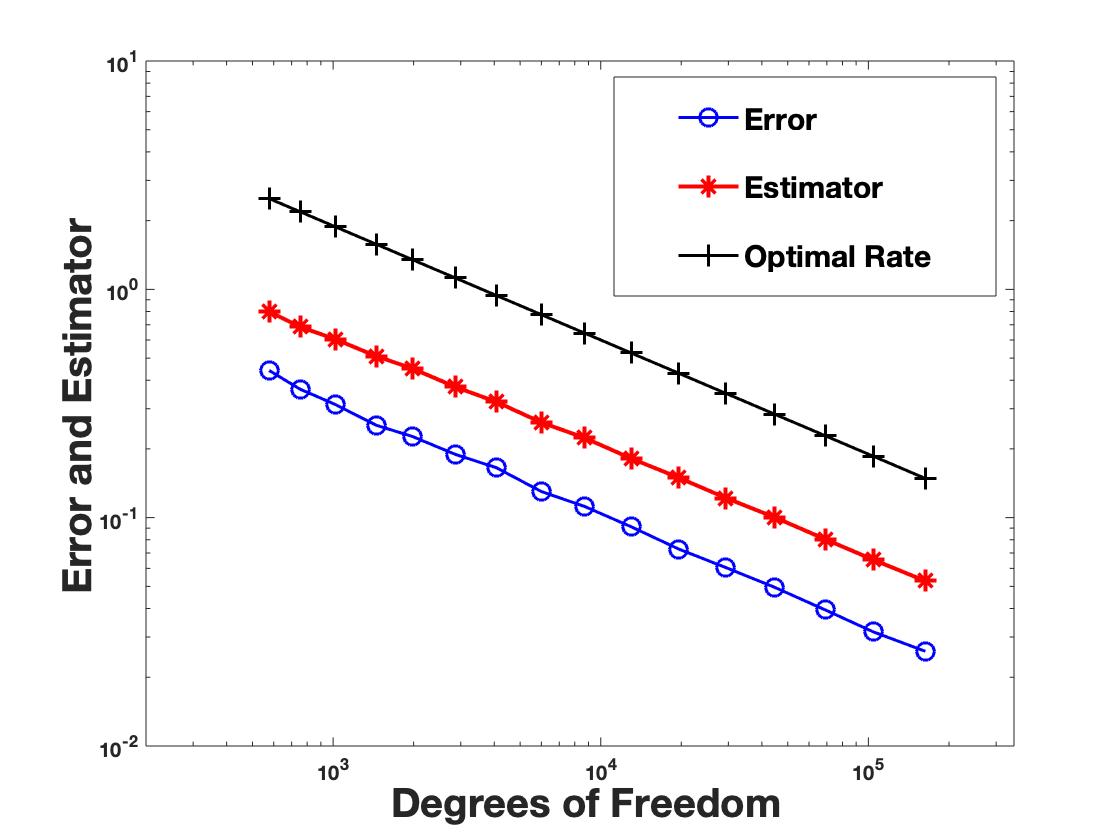}
 \caption{Error and estimator}
 \label{fig1}
\end{subfigure}%
\begin{subfigure}{.4\textwidth}
  \centering
  \includegraphics[width=7cm,height=6cm]{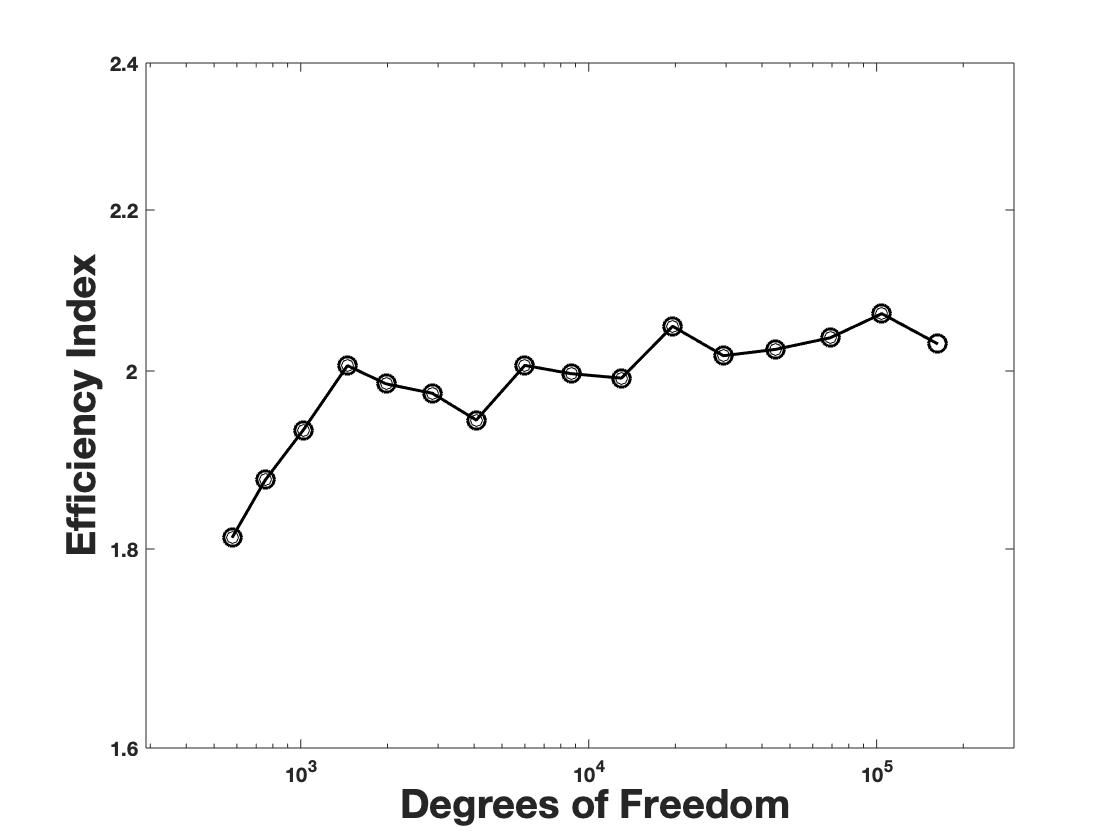}
\caption{Efficiency Index}
   \label{fig2}
\end{subfigure}
 \caption{{\large
Error, Estimator and Efficiency Index for Example \ref{exmpapost}}.} \label{fig:Error}
\end{figure}

\begin{figure}[t]
\centering
\begin{subfigure}{.4\textwidth}
  \centering
  \includegraphics[width=7cm,height=6cm]{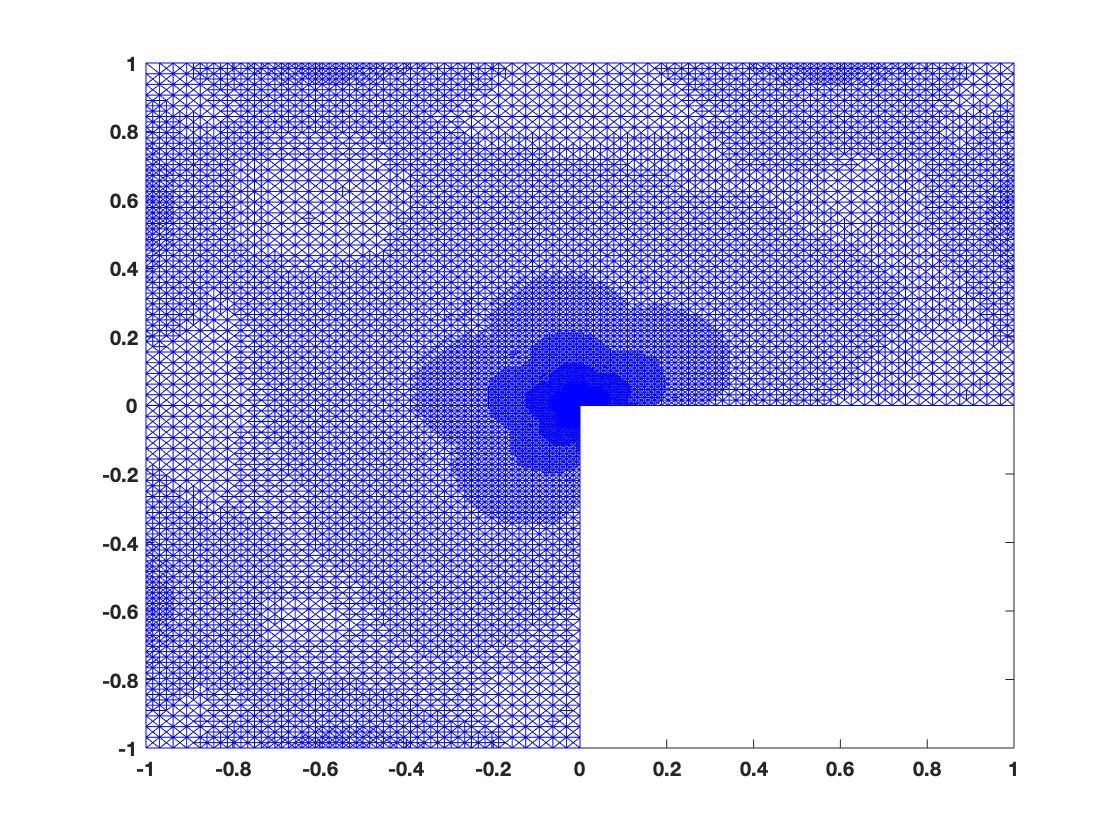}
 \caption{Adaptive mesh refinement}
    \label{fig3}
\end{subfigure}%
\begin{subfigure}{.4\textwidth}
  \centering
  \includegraphics[width=9cm,height=6cm]{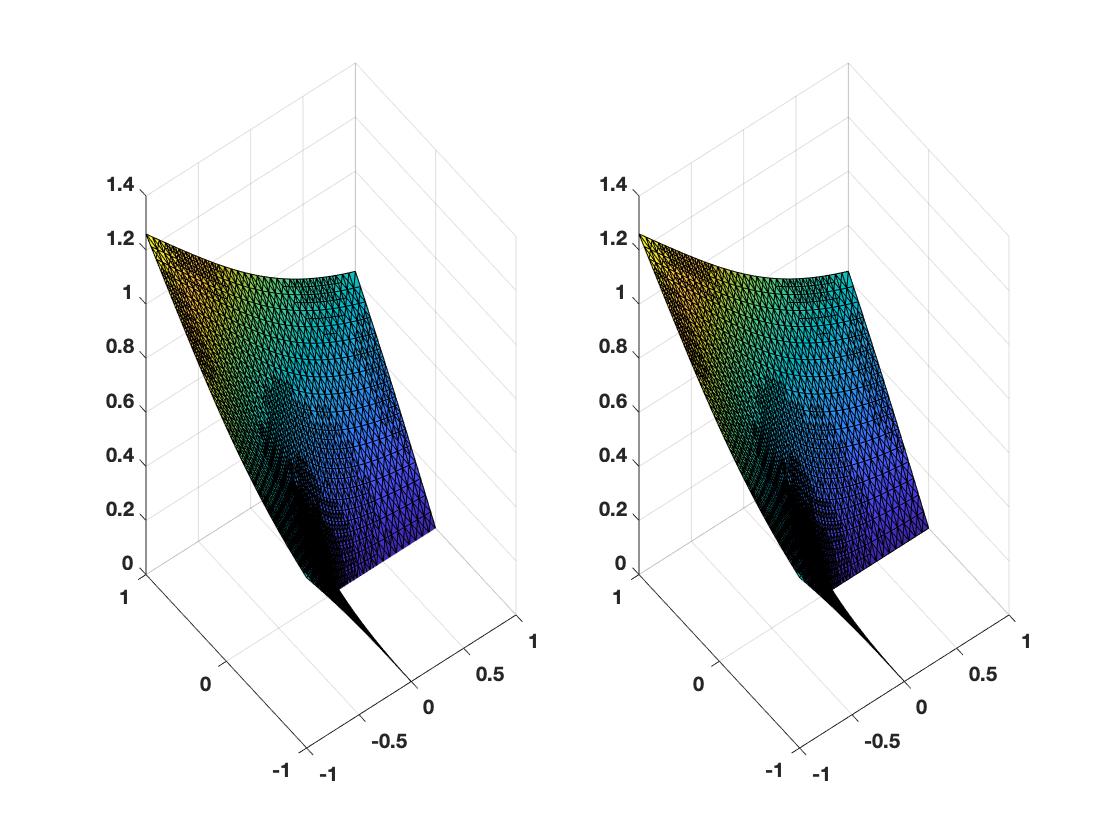}
\caption{ Discrete control and exact control }
    \label{fig4}
\end{subfigure}
 \caption{{\large
Adaptive mesh refinement and plots of the discrete (left) and the exact control (right)  for Example \ref{exmpapost}}.} \label{fig:Mesh}
\end{figure}

\begin{figure}[t]
\centering
\begin{subfigure}{.4\textwidth}
  \centering
  \includegraphics[width=6.7cm,height=6cm]{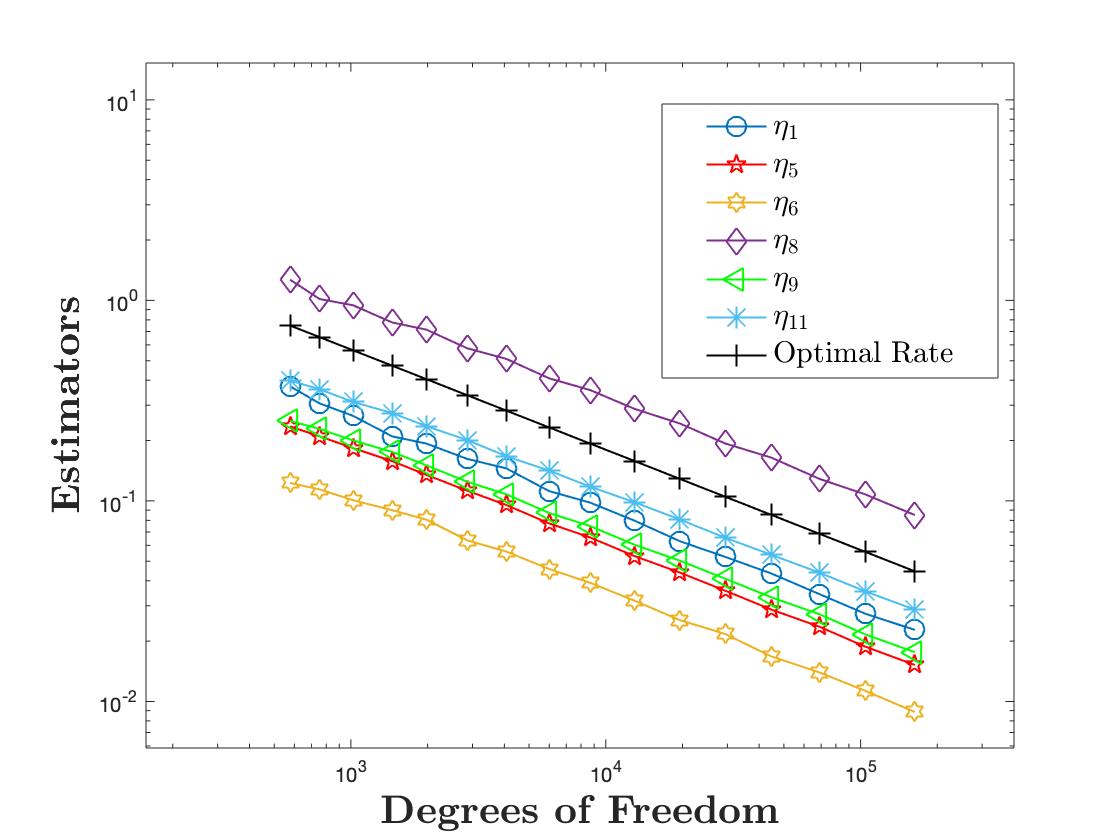}
 \caption{Volume and edge residuals}
    \label{fig111}
\end{subfigure}%
\begin{subfigure}{.4\textwidth}
  \centering
  \includegraphics[width=6.7cm,height=6cm]{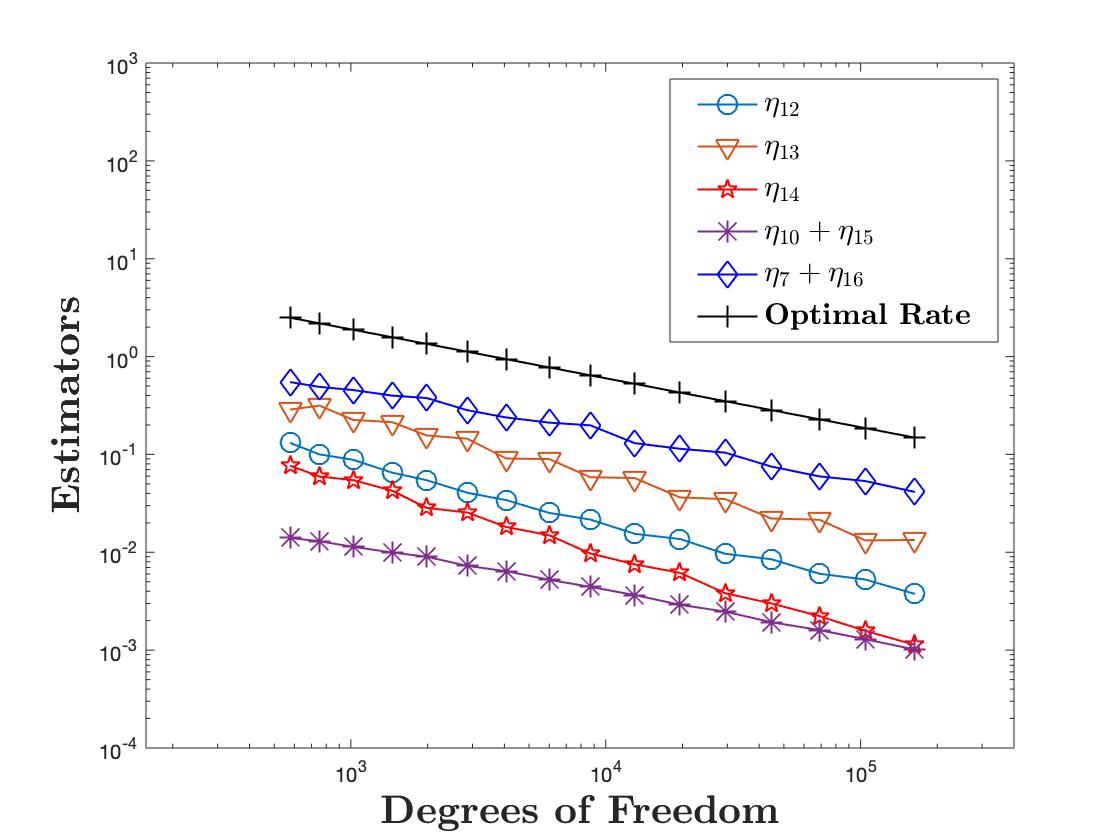}
\caption{ Boundary residuals}
    \label{fig112}
\end{subfigure}
 \caption{{\large
Estimator contributions for Example \ref{exmpapost}}.} 
\end{figure}

\begin{figure}[t]
\centering
\begin{subfigure}{.4\textwidth}
  \centering
  \includegraphics[width=6.7cm,height=6cm]{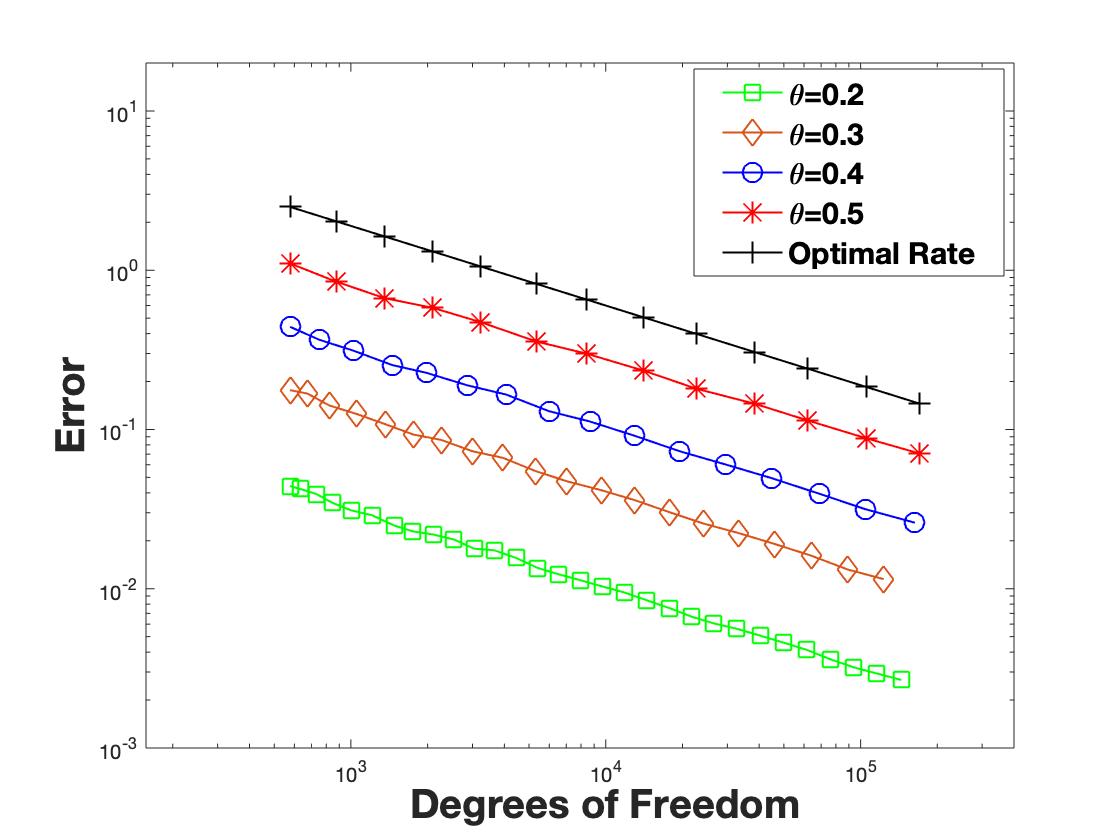}
 \caption{Error for different values of $\theta$}
    \label{fig11}
\end{subfigure}%
\begin{subfigure}{.4\textwidth}
  \centering
  \includegraphics[width=6.7cm,height=6cm]{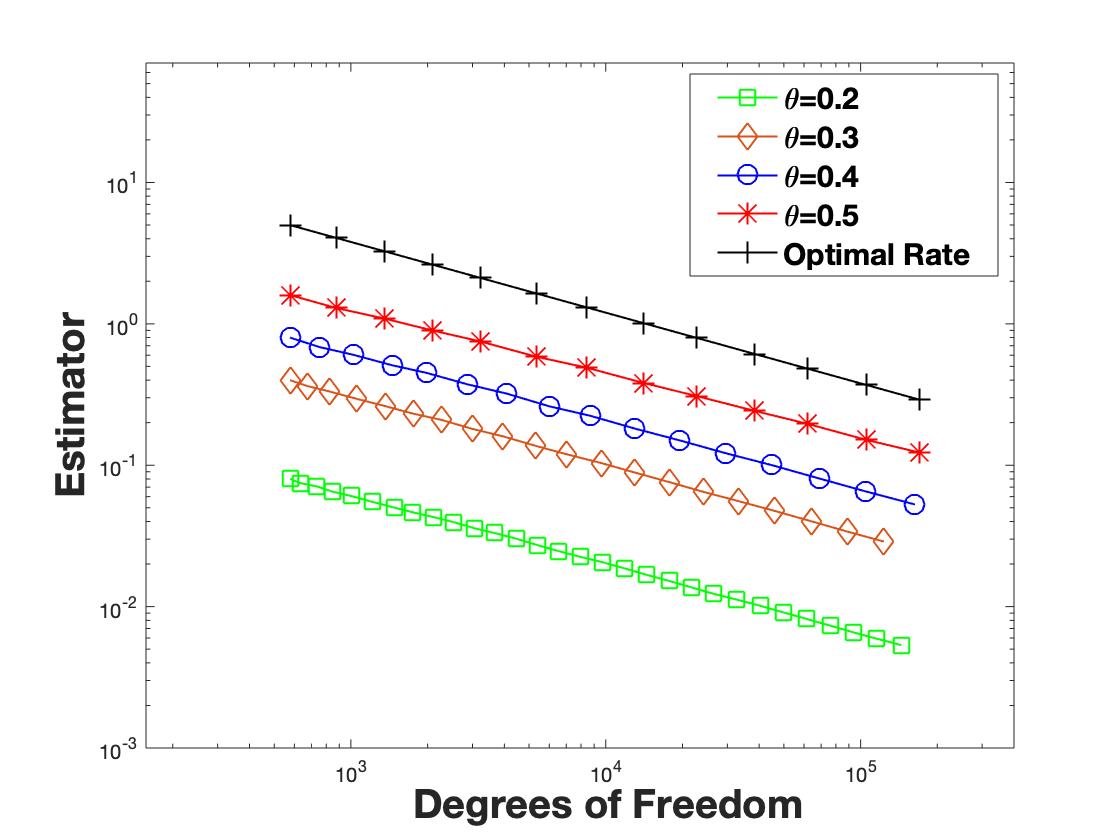}
\caption{ Estimator for different values of $\theta$ }
    \label{fig12}
\end{subfigure}
 \caption{{\large
Comparison of error and estimator for different values of  marking parameter $\theta$ for Example \ref{exmpapost}}.} 
\end{figure}


\section{Conclusions}\label{CNLS} 
\noindent
In this article we have developed a priori and a posteriori error analysis of mixed finite element method for the second order Dirichlet boundary control problem using energy space based approach. It is advantageous to use the mixed finite element methods for Dirichlet boundary control problem as it naturally incorporate  the normal derivative of costate on the boundary
in the weak formulation. The optimality system, construction of the suitable auxiliary problems and Helmholtz decomposition are crucial ingredients used in the analysis. The convergence of the method is illustrated over both uniform and adaptive meshes through numerical experiments.
 Though the analysis is carried out in two dimension but the same ideas can be extended to three dimension.\\
\par \noindent
\textbf{Acknowledgement.} We thank Professor Thirupathi Gudi for useful discussions on the Dirichlet boundary control problems.

\end{document}